\pdfoutput=1 
\documentclass{article}
\usepackage[utf8]{inputenc}
\usepackage[british]{babel}
\usepackage[mono=false]{libertine}
\usepackage[T1]{fontenc}
\usepackage[top=2.5cm, bottom=2.5cm, left=2.5cm, right=2.5cm]{geometry}
\usepackage{titlesec}
\usepackage[titletoc]{appendix}
\usepackage{amsmath}
\usepackage{amsthm}
\usepackage{amssymb}
\usepackage{url}
\usepackage{times}
\usepackage{enumitem}
\usepackage[usenames, dvipsnames]{color}
\usepackage{titlesec}
\usepackage{array}
\usepackage{mathtools}
\usepackage{amsfonts}
\usepackage{dsfont}
\usepackage{bm}
\usepackage{cases}
\usepackage{graphicx}
\usepackage{setspace}
\usepackage[perpage]{footmisc}
\usepackage{bbm}
\usepackage{booktabs}       
\usepackage{nicefrac}       
\usepackage{microtype}      
\usepackage{hyperref}
\hypersetup{colorlinks=true,urlcolor=blue,citecolor=magenta,linkcolor=blue}
\newcommand{\changelocaltocdepth}[1]{%
  \addtocontents{toc}{\protect\setcounter{tocdepth}{#1}}%
  \setcounter{tocdepth}{#1}%
}

\DeclareUnicodeCharacter{00A0}{~}

\def \({\left(}
\def \){\right)}
\def \[{\left[}
\def \]{\right]}

\newcommand{\bPhi}{\bm{\Phi}}

\newcommand{\bQ}{{\textbf {Q}}}

\newcommand{\bO}{{\textbf {O}}}

\newcommand{\bU}{{\textbf {U}}}

\newcommand{\bV}{{\textbf {V}}}

\newcommand{\bY}{{\textbf {Y}}}

\newcommand{\bM}{{\textbf {M}}}

\newcommand{\bW}{{\textbf {W}}}
\newcommand{\bZ}{{\textbf {Z}}}

\newcommand{\bA}{{\textbf {A}}}
\newcommand{\bB}{{\textbf {B}}}

\newcommand{\bX}{{\textbf {X}}}
\newcommand{\bx}{{\textbf {x}}}

\newcommand{\bGamma}{{\boldsymbol{\Gamma}}}

\newcommand{\bzeta}{{\boldsymbol{\zeta}}}

\newcommand{\bz}{{\textbf {z}}}

\newcommand{\bS}{{\textbf {S}}}

\newcommand{\ba}{{\textbf {a}}}

\newcommand{\be}{\begin{equation}}
\newcommand{\ee}{\end{equation}}

\newcommand\smallO{
  \mathchoice
    {{\scriptstyle\mathcal{O}}}
    {{\scriptstyle\mathcal{O}}}
    {{\scriptscriptstyle\mathcal{O}}}
    {\scalebox{.7}{$\scriptscriptstyle\mathcal{O}$}}
  }
\newcommand{\bea}{\begin{align}}
\newcommand{\eea}{\end{align}}
\newcommand{\norm}[1]{\left\lVert#1\right\rVert}

\newtheorem{theorem}{Theorem}[section]
\newtheorem{definition}[theorem]{Definition}

\newtheorem{conjecture}[theorem]{Conjecture}
\newtheorem{lemma}[theorem]{Lemma}
\newtheorem{proposition}[theorem]{Proposition}
\newtheorem{remark}[theorem]{Remark}
\newtheorem{corollary}[theorem]{Corollary}

\DeclareMathAlphabet{\varmathbb}{U}{bbold}{m}{n}

\newcommand{\EE}{\mathbb{E}}
\newcommand{\bbR}{\mathbb{R}}

\newcommand{\bbK}{\mathbb{K}}
\newcommand{\bbN}{\mathbb{N}}

\newcommand{\bbC}{\mathbb{C}}

\newcommand{\out}{{\rm out}}
\newcommand{\opt}{{\rm opt}}
\newcommand{\WR}{{\rm WR}}

\newcommand{\MMSE}{{\rm MMSE}}
\newcommand{\IT}{{\rm IT}}
\newcommand{\FR}{{\rm FR}}
\newcommand{\alg}{{\rm Algo}}
\newcommand{\MSE}{{\rm MSE}}

\newcounter{dummy}
\makeatletter
\newcommand\myitem[1][]{\item[#1]\refstepcounter{dummy}\def\@currentlabel{#1}}
\makeatother

\showboxdepth=\maxdimen
\showboxbreadth=\maxdimen

\begin{document}
\setcounter{tocdepth}{2}

\title{Phase retrieval in high dimensions:\\Statistical and computational phase transitions}
\author{Antoine Maillard$^{\star,\diamond}$, Bruno Loureiro$^{\dagger,\oplus}$, Florent Krzakala$^{\star,\oplus}$, Lenka Zdeborov\'a$^{\dagger,\otimes}$}
\date{}
\maketitle
{\let\thefootnote\relax\footnote{
\!\!\!\!\!\!\!\!\!\!\!\!\!\!
$\star$ Laboratoire de Physique de l'\'Ecole Normale Sup\'erieure, PSL University, 
CNRS, Sorbonne Universit\'es, Paris, France.\\
$\dagger$ Institut de Physique Th\'eorique, CNRS, CEA, Universit\'e Paris-Saclay, Saclay, France.\\
$\oplus$ IdePHICS laboratory, EPFL, Switzerland. \\
$\otimes$ SPOC laboratory, EPFL, Switzerland. \\
$\diamond$ To whom correspondence shall be sent: \href{mailto:antoine.maillard@ens.fr}{antoine.maillard@ens.fr}
}}
\setcounter{footnote}{0}

\begin{abstract}	
        We consider the \emph{phase retrieval} problem of reconstructing a $n$-dimensional real or complex signal $\bX^{\star}$ from $m$ (possibly noisy) observations 
        $Y_\mu = | \sum_{i=1}^n \Phi_{\mu i} X^{\star}_i/\sqrt{n}|$, for a large class of correlated real and complex random sensing matrices $\bPhi$, in a high-dimensional setting where $m,n\to\infty$ while $\alpha = m/n=\Theta(1)$. 
        First, we derive sharp asymptotics for the lowest possible estimation error achievable statistically 
        and we unveil the existence of sharp phase transitions for the weak- and full-recovery thresholds as a function of the singular values of the matrix $\bPhi$. 
        This is achieved by providing a rigorous proof of a result first obtained by the replica method from statistical mechanics.
        In particular, the information-theoretic transition to perfect recovery for full-rank matrices appears at $\alpha=1$ (real case) and  $\alpha=2$ (complex case).        
        Secondly, we analyze the performance of the best-known polynomial time algorithm for this problem --- approximate message-passing--- establishing the existence of a statistical-to-algorithmic gap depending, again, on the spectral properties of $\bPhi$. 
        Our work  provides an extensive classification of the statistical and algorithmic thresholds in high-dimensional phase retrieval for a broad class of random matrices.
\end{abstract}

\section{Introduction}\label{sec:introduction}
Consider the reconstruction problem of a real or complex signal %
from $m$ observations of its modulus
\begin{align}
    Y_\mu=\Big|\frac{1}{\sqrt{n}} \sum_{i=1}^n \Phi_{\mu i} X^\star_i\Big|, & \qquad \mu=1,\cdots, m,
    \label{eq:intro:task}
\end{align}
\noindent where the $m\times n$ \emph{sensing} matrix $\bPhi \in \bbK^{m \times n}$ is known, with 
$\bX^\star \in \bbK^n$ ($\mathbb{K}\!\in\!\{\bbR,\bbC\}$).
More generally, measurements can be a noisy function of the modulus, for example by an additive Gaussian noise.
This inverse problem, known in the literature under the umbrella of \emph{phase retrieval}, is relevant to a series of signal processing \cite{fienup1982phase,unser1988maximum,dremeau2015reference} and statistical estimation \cite{candes2015phase,candes2015bphase,jaganathan2015phase,waldspurger2015phase} tasks. It appears in setups in optics and crystallography where detectors can often only measure information about the amplitude of signals, thus losing the information about its phase. It is also a challenging example of a non-convex problem and non-convex optimization with a complex loss landscape \cite{netrapalli2013phase,sun2018geometric,NIPS2018_8127}. Here we are interested in understanding the fundamental limitations of phase retrieval. We focus on the following questions:
\begin{enumerate}[label=\roman*),itemsep=0mm,nosep]
    \item What is the \emph{lowest possible error} one can get in estimating the signal $\bX^{\star}$?
    \item What is the \emph{minimal} number of measurements needed to produce an estimator positively correlated with the signal (that is with non-trivial error in the $n,m \to \infty$ limit)?
    \item How to \emph{efficiently} reconstruct $\bX^{*}$ {\it in practice} with a polynomial time algorithm?
\end{enumerate}
 We provide a sharp answer to these questions for a large set of random sensing matrices $\bPhi$ that hold with high probability in the \emph{high-dimensional limit} where $m,n\!\to\!\infty$ keeping the rate $\alpha \!=\!m/n$ fixed.  
 
\paragraph{Main contributions and related work ---}  There has been an extensive amount of work on phase retrieval with random matrices.
The performance of the Bayes-optimal estimator has been heuristically derived for real orthogonally invariant matrices $\bPhi$ and real signals drawn from generic but separable  distributions
\cite{kabashima2008inference,takahashi2020macroscopic}. Results for the i.i.d.\ (real) Gaussian matrix case were rigorously proven in \cite{barbier2019optimal}, where the algorithmic gap is also studied. This analysis was later non-rigorously extended to the case of non-separable prior distributions \cite{aubin2019exact}. The weak-recovery transition discussed here was studied in detail in \cite{Mondelli_2018,luo2019optimal} for i.i.d.\ Gaussian matrices, while the case of unitary-column matrices was discussed in \cite{ma2017orthogonal,ma2019spectral,dudeja2020analysis}. Our analysis extends these results by considering arbitrary matrices with orthogonal or unitary invariance properties, encapsulating all the cases described above. Message passing algorithms, in particular the generalized vector-approximate message-passing (G-VAMP), have been studied in \cite{rangan2017vector, schniter2016vector}. In the present setting these algorithms are conjectured to be optimal among all polynomial-time ones. To test the performance of the G-VAMP algorithm, we used the TrAMP library \cite{baker2020tramp} that provides an open-source implementation.
In the present work we derive sharp asymptotics for the lowest possible estimation error achievable statistically and algorithmically,  locate the phase transitions for weak- and full-recovery as a function of the singular values of the matrix $\bPhi$ and also discuss the existence of a statistical-to-algorithmic gap. Our main contributions are:
\begin{itemize}[leftmargin=*,nosep]
    \item  We extend the results of \cite{takahashi2020macroscopic} to the complex case, 
by using the heuristic replica method from statistical physics to derive a unified single-letter formula for the performance of the Bayes-optimal estimator 
    under a separable signal distribution $P_{0}$, and for $\bPhi$ taken from a right-orthogonally (unitarily in the complex case) invariant ensemble with arbitrary spectrum. 
    \item We rigorously prove the aforementioned formula in two particular cases. First, when the distribution $P_0$ is Gaussian (real or complex) and $\bPhi = \bW \bB$ is the product of a Gaussian matrix $\bW$ with an arbitrary matrix $\bB$. Second, for a Gaussian matrix $\bPhi$ (real or complex) with any separable distribution $P_0$. These are non-trivial extensions of the the proofs of \cite{barbier2019optimal,barbier2018mutual,aubin2018committee,barbier2019adaptive}.
    \item In the $n,m \to \infty$ limit, with $\alpha = m/n = \Theta(1)$, we identify (as a function of the singular values distribution of $\bPhi$) the \emph{algorithmic weak-recovery} threshold $\alpha_\mathrm{WR,\alg}$ above which better-than-random inference reconstruction of $\bX^\star$ is  possible in polynomial time. 
    \item We establish the \emph{information-theoretic full recovery} threshold $\alpha_{\mathrm{FR,IT}}$ above which full reconstruction of $\bX^\star$ (meaning that the recovery is perfect up to the possible rank deficiency of $\bPhi$) is statistically possible, as a function of the singular values distribution of $\bPhi$.
    \item We provide a measure of the intrinsic algorithmic hardness of phase retrieval by studying the performance of the G-VAMP algorithm, which can be rigorously tracked for orthogonally (unitarily) invariant $\bPhi$ \cite{rangan2017vector, schniter2016vector}. 
    We use this rigorous analysis to numerically establish the existence or absence of a statistical-to-algorithmic gap for reconstruction in the following cases $\bPhi\in\{\text{real/complex Gaussian, orthogonal/unitary, product of complex Gaussians}\}$, for which such an analysis was, to the best of our knowledge, lacking.
\end{itemize}
Our findings for the statistical and algorithmic thresholds are summarized in Table \ref{table:results}, for different real and complex ensembles of $\bPhi$.
Entries in bold emphasize new results obtained in this manuscript, filling a gap between the different previous works in the phase retrieval literature.

Throughout the manuscript we adopt the following notation. Let $\beta \in \{1,2\}$. We denote $\bbK = \bbR$ if $\beta = 1$ and $\bbK = \bbC$ if $\beta = 2$. $\mathcal{U}_\beta(n)$ denotes the orthogonal (respectively unitary) group. For $m \geq n$, a matrix $\bA \in \bbK^{m \times n}$ is said to be \emph{column-orthogonal (unitary)} if $\bA^\dagger \bA = \mathbbm{1}_n$.
For $x,y \in \bbK$, we define a `dot product' as 
$x \cdot y \equiv xy$ if $\bbK = \bbR$ and $x \cdot y \equiv \mathrm{Re}[\overline{x}y]$ if $\bbK = \bbC$.
In particular $x \cdot x = |x|^2$.
The Gaussian measure $\mathcal{N}_\beta(0,1)$ is defined as
${\cal D}_\beta z \equiv (2 \pi/\beta)^{-\beta/2} \ \exp(-\beta |z|^2 / 2) \ \mathrm{d}z$ and
$\text{D}_{\text{KL}}$ is the Kullback–Leibler divergence.  
$\nu$ will denote the asymptotic spectral density of $\bPhi^\dagger \bPhi / n$ and we designate $\langle f(\lambda)\rangle_\nu \equiv \int \nu(\mathrm{d}\lambda) f(\lambda)$ the linear statistics of $\nu$.
\begin{table}[t]
    \centering
\begin{tabular}{|c|c|c|c|}
  \hline
   Matrix ensemble and value of $\beta$ & $\alpha_{ \mathrm{WR}, \alg}$ & $\alpha_{\mathrm{FR,IT}}$ & $\alpha_{\mathrm{FR,Algo}}$ \\
  \hline
  Real Gaussian $\bPhi$ ($\beta = 1$) & $0.5$ \cite{Mondelli_2018,luo2019optimal} & $1$ \cite{candes2006near} & $\simeq 1.12$ \cite{barbier2019optimal} \\
  Complex Gaussian $\bPhi$ ($\beta = 2$)& $1$ \cite{Mondelli_2018,luo2019optimal} & $\bm{2}$ & $\bm{\simeq 2.027}$ \\
  Real column-orthogonal $\bPhi$ ($\beta = 1$)& $\bm{1.5}$ & $1$ \cite{candes2006near} & $\bm{\simeq 1.584}$ \\
  Complex column-unitary $\bPhi$ ($\beta = 2$)& $2$ \cite{ma2017orthogonal,ma2019spectral} & $\bm{2}$ & $ \bm{\simeq 2.265}$ \\ 
  $\bPhi = \bW_1 \bW_2$ ($\beta = 1$, aspect ratio $\gamma$) & $\gamma/(2(1+\gamma))$ \cite{aubin2019exact} & $\min(1,\gamma)$ \cite{candes2006near} & Thm.~\ref{thm:main_product_gaussian_rotinv} \cite{aubin2019exact} \\
  $\bPhi = \bW_1 \bW_2$ ($\beta = 2$, aspect ratio $\gamma$) & $\bm{\gamma/(1+\gamma)}$ & $\bm{\min(2,2\gamma)}$ & \textbf{Thm.~\ref{thm:main_product_gaussian_rotinv}} \\
  $\bPhi$, $\beta \in \{1,2\}$, $\text{rk}[\bPhi^\dagger \bPhi]/n = r$ & \textbf{Eq.~\eqref{eq:wr_threshold_phase_retrieval}}  & $\bm{\beta} \bm{r}$ & \textbf{Conj.~\ref{conjecture:replicas_general}} \\
  Gauss. $\bPhi$, $\beta \in \{1,2\}$, symm. $P_0$, $P_\mathrm{out}$ & Eq.~\eqref{eq:wr_threshold_gaussian_matrix} \cite{Mondelli_2018,luo2019optimal} & \textbf{Thm.~\ref{thm:main_product_gaussian_rotinv}} & \textbf{Thm.~\ref{thm:main_product_gaussian_rotinv}} \\
  $\bPhi$, $\beta \in \{1,2\}$, symm. $P_0$, $P_\mathrm{out}$ & \textbf{Eq.~\eqref{eq:wr_threshold_general}} & \textbf{Conj.~\ref{conjecture:replicas_general}}  & \textbf{Conj.~\ref{conjecture:replicas_general}} \\
  \hline
\end{tabular}
\vspace{0.3cm}
    \caption{
    Values of the algorithmic weak recovery, information-theoretic full recovery, and algorithmic full recovery thresholds for several random matrix ensembles. When the ensemble of $\bPhi$ is not specified, we consider any right-orthogonally (unitarily) invariant ensemble with well-defined asymptotic spectral density.
    The last two lines are given for any symmetric (cf eq.~\eqref{eq:def_symmetry}) prior $P_0$ and channel $P_\mathrm{out}$, while all other results are for Gaussian $P_0$ and a noiseless phase retrieval channel. We reference results of this manuscript when the value is not given by a closed-form expression, but can be computed from the formulas herein.
    In some particular ensembles, we have numerically analyzed these equations in Section~\ref{sec:numerical}.
    The new results obtained in our work are written in bold style, and we give references to papers in which the previously known thresholds were computed. 
    }
    \label{table:results}
\end{table}
\vspace{-0.3cm}
\paragraph{Some consequences of our results ---} 
We list here some interesting (and often surprising) consequences of our analysis. Since our rigorous results concern a subclass of orthogonally invariant matrices, proving and/or interpreting these statements more generally is an interesting future direction.
\begin{itemize}[leftmargin=*,nosep]
    \item 
 One sees from eq.~\eqref{eq:wr_threshold_general} that maximizing $\alpha_{\WR, \alg}$ implies maximizing $\langle \lambda \rangle_\nu^2 / \langle \lambda^2 \rangle_\nu$.
The highest ratio is reached when $\nu$ is a delta distribution: for any symmetric channel and prior (see~\eqref{eq:def_symmetry}) the 
  ensemble that maximizes $\alpha_{\WR,\alg}$ is thus the one of uniformly-sampled
column-orthogonal ($\beta\!=\!1$) or column-unitary ($\beta\!=\!2$) matrices. Conversely, $\alpha_\mathrm{WR,\alg}$ can be made arbitrarily small using a product of many Gaussian matrices, both in the real and complex cases. 
\item In complex noiseless phase retrieval the information-theoretic weak-recovery threshold for column-unitary matrices is located at $\alpha_{\mathrm{WR}, \IT} = 2$ \cite{ma2019spectral}.
Our results (Table~\ref{table:results}) imply that this corresponds to an ``all-or-nothing'' transition located precisely at $\alpha\!=\!2$.
Moreover, the derivations of $\alpha_{\WR,\alg}$ and $\alpha_{\FR,\IT}$ in Sections~\ref{sec:weak_recovery},\ref{sec:numerical} show that for any complex matrix 
$\alpha_{\mathrm{WR}, \alg}\!=\!2 \langle \lambda \rangle_\nu^2 / \langle \lambda^2 \rangle_\nu\!\leq\!\alpha_\mathrm{FR,IT}\!=\!2 (1-\nu(\{0\}))$, with the equality only being attained for $\nu$ a delta distribution.
Uniformly sampled column-unitary matrices are thus the only right-unitarily invariant complex matrices which present an "all-or-nothing" transition in complex noiseless phase retrieval (for a Gaussian prior).
To the best of our knowledge, this is a first establishment of such a transition in a ``dense'' problem (as opposed to a sparse setting \cite{gamarnik2017high,reeves2019all}). Investigating further the existence of these transitions, e.g.\ as a function of the prior, is left for future work.
\item Consider again noiseless phase retrieval with Gaussian prior. For real orthogonal matrices, one has $\alpha_\mathrm{WR,\alg}\!-\! \alpha_\mathrm{FR,IT}\!=\!0.5\!>\!0$. Since $\alpha_\mathrm{WR,\alg}$ is a smooth function of the eigenvalue density $\nu$, we expect that the inequality holds for many real random matrix ensembles. However, in the complex case, by our previous point, $\alpha_\mathrm{WR,\alg}\!\leq\!\alpha_\mathrm{FR,IT}$. The gap thus only occurs in the real setting.
\end{itemize}
\section{Analysis of information-theoretically optimal estimation}\label{sec:replica_result_and_proof}

The phase reconstruction task introduced in eq.~\eqref{eq:intro:task} belongs to the large class of \emph{generalized linear estimation problems}. In this section, we provide a Bayesian analysis of the statistically optimal estimator $\hat{\bX}_{\opt}\in\mathbb{K}^{n}$ for this general class of problems. In the sections that follow, we draw the consequences for the case of the phase reconstruction problem we are interested in in this manuscript.
\\
In the generalized linear model, the goal is to reconstruct a 
\emph{signal} $\bX^{\star} \in \bbK^n$, with components drawn i.i.d.\ from a fixed prior distribution $P_0$ over $\mathbb{K}$,  from the \emph{observations} $\bY \in \bbR^m$ generated as:
\begin{align}\label{eq:alternative_pout}
    Y_\mu = \varphi_\mathrm{out}\Big(\frac{1}{\sqrt{n}} \sum_{i=1}^n \Phi_{\mu i} X_i^{\star}, A_\mu \Big), \quad 1 \leq \mu \leq m,
\end{align}
\noindent where $(A_\mu)_{\mu=1}^m \in \bbK^m$ are i.i.d.\ random variables with (known) distribution $P_A$ accounting for a possible noise, $\varphi_\mathrm{out}$ is the observation channel and $\bPhi$ is a random matrix with elements in $\bbK$. We let $P_{\out}(\cdot|z)$ denote the probability density function associated to the stochastic function $\varphi_{\out}(z,A)$. Further, we assume that $P_{0}$ has a second moment given by $\rho \equiv \EE[|x|^2] > 0$.
Note that the phase reconstruction problem introduced in eq.~\eqref{eq:intro:task} corresponds to a likelihood $P_{\rm out}(y|z)$ that only depends on $z$ through $|z|^2$.
For instance, for Gaussian additive noise it is explicitly given by $P_{\out}(y|z) =\mathcal{N}_{1}(y;|z|^2,\Delta)$, 
while the noiseless case corresponds to the limit $\Delta\!\downarrow\!0$ : $P_\mathrm{out}(y|z) = \delta(y-|z|^2)$.
In this work, we consider a large class of random matrices $\bPhi$ distributed as $\bPhi \overset{d}{=} \bU \bS \bV^\dagger$, with arbitrary $\bU \in \mathcal{U}_\beta(m)$, $\bV$ drawn uniformly from ${\cal U}_\beta(n)$, and $\bS$ the pseudo-diagonal of singular values of $\bPhi$.
 We assume that the spectral measure of $\bS^\intercal \bS / n$ almost surely converges (in the weak sense)
\footnote{We actually assume the following, which is (slightly) stronger: the convergence should happen 
at a rate at least $n^{1+\epsilon}$ for an $\epsilon > 0$. This condition was not precised in the replica calculation of \cite{takahashi2020macroscopic} for real matrices. In practice, in classical orthogonally (unitarily)-invariant random matrix ensembles, we often have $\epsilon = 1$.}
to a probability measure $\nu$ with compact support $\mathrm{supp}(\nu) \subset \bbR_+$.
Crucially, we assume that the statistician knows how the observations were generated - i.e.\ she has access to $P_{0}, P_{\text{out}}$ and the distribution of $\bPhi$, therefore reducing the problem to the reconstruction of the specific realization of $\bX^{\star}$. 
In this setting, commonly known as \emph{Bayes-optimal}, the statistically optimal estimator $\hat{\bX}$ minimizing the mean-squared error $\text{mse}(\hat{\bX}) \equiv ||\hat{\bX}-\bX^{\star}||^2_{2}$ is simply given by the posterior mean $\hat{\bX}_{\mathrm{opt}} = \mathbb{E}[\bx|\bY]$, where the posterior distribution is explicitly given by:
\begin{align}
    P(\mathrm{d}\bx | \bY) &\equiv \frac{1}{{\cal Z}_n(\bY)} \prod_{i=1}^n P_0(\mathrm{d}x_i) \, \prod_{\mu=1}^m P_\mathrm{out}\Big(Y_\mu \Big| \frac{1}{\sqrt{n}} \sum_{i=1}^n \Phi_{\mu i} x_i\Big).
    \label{eq:posterior}
\end{align}
Exact sampling from the posterior is intractable for large values of $n,m\in\mathbb{N}^\star$. However, certain information theoretical quantities are accessible analytically precisely in this limit. Indeed, our first set of results concerns a rigorous evaluation of the mutual information $I(\bX^{\star};\bY) \equiv \text{D}_{\text{KL}}(P_{X,Y}|P_{0}\otimes P_{Y})$ between the signal $\bX^{\star}$ and the observations $\bY$ for the generalized linear model in the high-dimensional limit of $n,m \to \infty$ with $m/n \to \alpha > 0$ fixed. This quantity fully characterizes the asymptotic performance of the Bayes-optimal estimator $\hat{\bX}_{\opt}$ in high dimensions via the I-MMSE theorem \cite{Guo_2005}.
\paragraph{Asymptotic mutual information and minimum mean-squared error---} The mutual information between the observations and the hidden variables can be decomposed into two terms:
\begin{align}\label{eq:def_mutual_info}
    I(\bX^{\star};\bY | \bPhi) = H(\bY|\bPhi)-H(\bY | \bX^\star,\bPhi).
\end{align}
The entropy $H(\bY|\bX^\star,\bPhi) = \EE \ln P(\bY | \bX^\star,\bPhi) = -m \EE \ln P_\mathrm{out}(Y_1 | (\bPhi \bX^\star)_1/\sqrt{n})$ is easily computed in the high-dimensional limit for a given channel $P_{\mathrm{out}}$:
\begin{align}\label{eq:conditional_entropy}
 \lim_{n \to \infty} -\frac{1}{n} H(\bY|\bX^\star,\bPhi) &
= \alpha \int_\bbR \mathrm{d}y \int_\bbK \mathcal{D}_\beta \xi \ P_\mathrm{out}(y|\sqrt{Q_z}\xi) \ln P_\mathrm{out}(y|\sqrt{Q_z}\xi),
\end{align}
with $Q_z \equiv \rho \langle \lambda \rangle_\nu / \alpha$.
Indeed, as $n \to \infty$, the law of $(\bPhi \bX^\star)_1 / \sqrt{n}$
asymptotically approaches $\mathcal{N}_\beta(0,Q_z)$ by the central limit theorem.
The challenge in computing the mutual information therefore reduces to the evaluation of the \emph{free entropy} $H(\bY | \bPhi) = \mathbb{E}\ln\mathcal{Z}_{n}(\bY)$, related to the log-normalization of the posterior. Our first result is a single-letter formula for the asymptotic free entropy density of right-orthogonally (unitarily) invariant sensing matrices:
\begin{conjecture}\label{conjecture:replicas_general}
Under the assumptions above,
the asymptotic free entropy density for the posterior distribution defined in eq.~\eqref{eq:posterior} with right-orthogonally (unitarily) invariant sensing matrix $\bPhi$ is:
\begin{align}
  \lim\limits_{n\to\infty}\frac{1}{n}\mathbb{E}_{\bY,\bPhi}\ln\mathcal{Z}_{n}(\bY)= \sup_{q_x \in [0,\rho]} \sup_{q_z \in [0,Q_z]} [I_0(q_x) + \alpha I_\mathrm{out}(q_z) + I_\mathrm{int}(q_x,q_z) ], 
  \label{eq:replica:freeen}
\end{align}
\noindent 
\begin{align*}
\text{where~~}   & I_0(q_x) \equiv \inf_{\hat{q}_x \geq 0} \Big[-\frac{\beta \hat{q}_x q_x}{2} +\mathbb{E}_{\xi} \mathcal{Z}_{0}(\sqrt{\hat{q}_{x}}\xi, \hat{q}_{x})\log{\mathcal{Z}_{0}(\sqrt{\hat{q}_{x}}\xi, \hat{q}_{x})}\Big], \\
    &I_{\out}(q_z) \equiv \inf_{\hat{q}_z \geq 0} \Bigg[ -\frac{\beta \hat{q}_z q_z}{2} - \frac{\beta}{2} \ln (\hat{Q}_z + \hat{q}_z) + \frac{\beta \hat{q}_z}{2 \hat{Q}_z}  \notag \\  \hspace{3mm} +  \mathbb{E}_{\xi}& \int_\bbR \mathrm{d}y~ \mathcal{Z}_{\out}\Big(y;\sqrt{\frac{\hat{q}_z}{\hat{Q}_z(\hat{Q}_z + \hat{q}_z)}}\xi, \frac{1}{\hat{Q}_z + \hat{q}_z}\Big)\log{\mathcal{Z}_{\out}\Big(y;\sqrt{\frac{\hat{q}_z}{\hat{Q}_z(\hat{Q}_z + \hat{q}_z)}}\xi, \frac{1}{\hat{Q}_z + \hat{q}_z} \Big)} \Bigg], \\
   & I_\mathrm{int}(q_x,q_z) \equiv \inf_{\gamma_x ,\gamma_z \geq 0} \Big[\frac{\beta}{2} (\rho - q_x) \gamma_x + \frac{\alpha \beta}{2} (Q_z - q_z) \gamma_z - \frac{\beta}{2} \langle \ln (\rho^{-1} + \gamma_x + \lambda \gamma_z) \rangle_\nu \Big] \\
    & \hspace{1cm} - \frac{\beta}{2} \ln(\rho-q_x) - \frac{ \beta q_x}{2 \rho} - \frac{\alpha \beta}{2} \ln(Q_z - q_z) - \frac{\alpha \beta q_z}{2 Q_z}. \nonumber
\end{align*}
We defined $Q_z \equiv \rho \langle \lambda \rangle_\nu/\alpha$ and $\hat{Q}_z \equiv 1/Q_z$, $\xi\sim\mathcal{N}_{\beta}(0,1)$
and the following auxiliary functions:
\begin{align}
    \mathcal{Z}_{0}(b, a) \equiv \mathbb{E}_{z}\big[P_{0}(z)e^{-\frac{\beta}{2}a|z|^2+\beta b\cdot z}\big],&& \mathcal{Z}_{\out}(y;\omega, v) &\equiv \mathbb{E}_{z}\big[ P_{\out}(y\Big|\sqrt{v}z+\omega)\big],\label{eq:auxiliary}
\end{align}
\noindent with $z\sim\mathcal{N}_{\beta}(0,1)$.
Moreover, the asymptotic minimum mean squared error, achieved by the Bayes-optimal estimator, is equal to $\rho - q_x^\star$, with $q_x^\star$ the solution of the above extremization problem;
\begin{align}\label{eq:mmse}
   \lim_{n \to \infty} \mathrm{MMSE} &= \lim_{n \to \infty} \frac{1}{n} \EE \lVert \bX^\star - \hat{\bX}_\mathrm{opt} \rVert^2 = \rho - q_x^\star. 
\end{align}
\end{conjecture}
This formula, derived in Appendix \ref{sec:app:replicas} using the heuristic (hence the \emph{conjecture}) replica method from statistical physics \cite{mezard1987spin}, holds for any separable signal distribution $P_{0}$ and for any choice of likelihood $P_{\out}$. It extends the formula from \cite{takahashi2020macroscopic} to complex signals $\bX^\star$ and sensing matrices $\bPhi$. In particular, it also holds in the case of complex matrices $\bPhi$ and real signal $\bX^{\star}$, by adding a constraint on the imaginary part of $\bX^{\star}$ in $P_{0}$. It also encompasses the case of sparse signals, which is of wide interest in the compressive sensing literature \cite{donoho2006compressed, donoho2009message, krzakala2014variational, krzakala2012probabilistic, schniter2014compressive}.
Proving Conjecture~\ref{conjecture:replicas_general} is a challenging open problem. We provide a significant step by proving Conjecture~\ref{conjecture:replicas_general} for a broad class of likelihoods $P_\mathrm{out}$ and in two settings: a restricted signal distribution $P_{0}$ and a broad class of real and complex likelihoods and sensing matrices $\bPhi$, or a broad class of prior distribution $P_0$ and (real or complex) Gaussian $\bPhi$.
\begin{theorem}\label{thm:main_product_gaussian_rotinv}
    Let us denote
  \begin{enumerate}[label=(h\arabic*),leftmargin=*,nosep]
    \myitem[$(H0)$] \label{hyp:channel}
    $\varphi_\mathrm{out}: \bbK^2 \to \bbR$ is $\mathcal{C}^2$, and $(z,a) \mapsto (\varphi_\mathrm{out}(z,a),\partial_z \varphi_\mathrm{out}(z,a),\partial^2_z \varphi_\mathrm{out}(z,a))$ is bounded. 
    \item \label{hyp:prior}
    $P_0$ is a centered Gaussian distribution, without loss of generality $P_0 = \mathcal{N}_\beta(0,1)$.
    \item \label{hyp:phi_product}
    $\bPhi$ is distributed as $\bPhi \overset{d}{=} \bW \bB/ \sqrt{p}$, with $\bW \in \bbK^{m \times p}$ an i.i.d.\ standard Gaussian matrix, and $\bB \in \bbK^{p \times n}$ an arbitrary matrix (random or deterministic), independent of $\bW$. 
    Moreover, as $n \to \infty$, $p / n \to \delta >0$.
    \item \label{hyp:convergence_spectrum_B}
    The empirical spectral distribution of $\bB^\dagger \bB / n$ weakly converges (a.s.) to a compactly-supported measure $\nu_B \neq \delta_0$.
    Moreover, there is $\lambda_\mathrm{max} \geq 0$ such that 
    a.s.\ $\lambda_\mathrm{max}(\bB^\dagger \bB / n) \to_{n \to \infty} \lambda_\mathrm{max}$.
    \myitem[$(h'1)$]\label{hyp:any_p0_gaussian_phi}$P_0$ has a finite second moment, and $\Phi_{\mu i} \overset{\mathrm{i.i.d}}{\sim} \mathcal{N}_\beta(0,1)$.
  \end{enumerate}
  Assume that all \ref{hyp:channel},\ref{hyp:prior},\ref{hyp:phi_product},\ref{hyp:convergence_spectrum_B} or that all \ref{hyp:channel},\ref{hyp:any_p0_gaussian_phi} stand.
  Then Conjecture~\ref{conjecture:replicas_general} holds with $\nu$ the asymptotic eigenvalue distribution of $\bPhi^\dagger \bPhi/ n$\footnote{The rigorous statement on the limit of the MMSE requires adding a side information channel with arbitrarily small signal, cf Appendix~\ref{subsec:app_proof_mmse}.}.
\end{theorem}
\noindent
The proof is based on the adaptive interpolation method\footnote{In Theorem~\ref{thm:main_product_gaussian_rotinv}, we rely on some Gaussianity, either in the prior or in the data matrix. This is a not specific to our setting, but rather a fundamental limitation of the adaptive interpolation method used for the proof.} \cite{barbier2019adaptive}, and is provided in Appendix~\ref{sec:app_proof_theorem}.
In particular, Theorem~\ref{thm:main_product_gaussian_rotinv} allows to rigorously compute the asymptotic minimum mean-squared error (MMSE) achieved by the Bayes-optimal estimator.
Theorem~\ref{thm:main_product_gaussian_rotinv} extends the rigorous results of \cite{barbier2019optimal} to a larger class of sensing matrices and to the complex case, including both real orthogonally invariant matrices  and the products of i.i.d.\ Gaussian matrices, heuristically studied respectively in \cite{takahashi2020macroscopic} and \cite{aubin2019exact}. 
\begin{remark}\label{remark:relaxing_hypotheses}
  Following the arguments of \cite{barbier2019optimal,aubin2018committee}, hypothesis \ref{hyp:channel} can be relaxed to continuity a.e.\ and the existence of moments of $\varphi_\mathrm{out}$,
  so that our theorem also covers noiseless phase retrieval.
\end{remark}
This single-letter formula  reduces the high-dimensional computation of the MMSE to a simple low-dimensional extremization problem.
The MMSE as a function of the sample complexity $\alpha$ can be readily computed from eqs.~\eqref{eq:replica:freeen} and \eqref{eq:mmse} for a given signal distribution $P_{0}$ (determining $I_{0}$), likelihood $P_{\text{out}}$ (determining $I_{\text{out}}$) and spectral density $\nu$ (determining $I_{\text{int}}$).
\paragraph{Statistical vs algorithmic performance ---} Conjecture~\ref{conjecture:replicas_general} and Theorem~\ref{thm:main_product_gaussian_rotinv} show that the global maximum of the potential in eq.~\eqref{eq:replica:freeen} describes the performance of the statistically optimal estimator $\hat{\bX}_{\opt}$ for generalized linear estimation. 
Interestingly, eq.~\eqref{eq:replica:freeen} also contains rich information about the algorithmic aspects of this problem. Indeed, it has been shown that the performance of the G-VAMP algorithm, the best-known polynomial time algorithm for this problem,  corresponds precisely to the MSE achieved by running gradient descent on the potential in eq.~\eqref{eq:replica:freeen} from the trivial initial condition $q_{x}=q_{z}=0$ \cite{rangan2017vector, schniter2016vector}. In the sections that follow, we exploit this result to derive the thresholds characterizing the statistical and algorithmic limitations of signal estimation.
We adopt the subscript $\IT$ for the thresholds related to the Bayes-optimal estimator and $\alg$ for the G-VAMP ones\footnote{Even though we do not provide a proof for the optimality of-GVAMP,
we chose such notation in accordance with the previous literature on this topic, in which this optimality is often assumed.}.

\section{Weak-recovery transition}
\label{sec:weak_recovery}
A natural question to ask is: what is the minimum sample complexity $\alpha_{\WR, \alg}\geq 0$ such that for all $\alpha \geq \alpha_{\WR,\alg}$ we can algorithmically reconstruct $\bX^{\star}$ better than a trivial random draw from the known signal distribution $P_{0}$? Also known as the \emph{algorithmic weak-recovery} threshold, $\alpha_{\WR,\alg}$ can also be characterized in terms of the MSE achieved by G-VAMP:
\begin{align*}
    \alpha_{\WR,\alg} \equiv \underset{\alpha \geq 0}{\text{argmin}}\{\MSE_\mathrm{GVAMP}(\alpha) < \rho\}.
\end{align*}
In this section, we establish sufficient conditions for the existence of the algorithmic weak-recovery threshold $\alpha_{\WR, \alg}\geq 0$, and we derive an analytical expression for this threshold.

\paragraph{G-VAMP State Evolution  ---}
Recalling that $q_{x}\in [0,\rho]$, from eq.~\eqref{eq:mmse} it is easy to see that the weak-recovery threshold is the smallest sample complexity $\alpha$ such that the potential of eq.~\eqref{eq:replica:freeen} has no longer a local maximum in $q_x = 0$. In opposition, the region for which the MSE is maximal ($\MSE = \rho$) corresponds to the existence of a \emph{trivial} maximum in eq.~\eqref{eq:replica:freeen} with $q_{x}=q_{z}=0$. The extrema of the potential in eq.~\eqref{eq:replica:freeen} can be characterized by the solutions of the following \emph{State Evolution} (SE) equations, obtained by looking at the zero-gradient points:
\begin{subnumcases}{\label{eq:se_general}}
    q_x = \mathbb{E}_{\xi} \mathcal{Z}_{0}|f_{0}|^2, & $q_z =\frac{1}{\hat{Q}_z + \hat{q}_z} \big[\frac{\hat{q}_z}{\hat{Q}_z} + \mathbb{E}_{\xi}\int \mathrm{d}y~\mathcal{Z}_{\out} |f_{\out}|^2\big]$, \\
    \hat{q}_x = \frac{q_x}{\rho(\rho-q_x)} - \gamma_x , &
    $\hat{q}_z = \frac{q_z}{Q_z(Q_z-q_z)} - \gamma_z$, \\
    \rho-q_x = \Big\langle \frac{1}{\rho^{-1} + \gamma_x + \lambda \gamma_z}\Big\rangle_\nu, &
    \label{eq:self_consistent_gamma}
    $\alpha(Q_z-q_z) = \Big\langle \frac{\lambda}{\rho^{-1} + \gamma_x + \lambda \gamma_z}\Big\rangle_\nu$.
\end{subnumcases}
\noindent where $f_{0}(b,a) = \partial_{b}\log\mathcal{Z}_{0}(b,a)$ and $f_{\out}(y;\omega,v) = \partial_{\omega}\log\mathcal{Z}_{\out}(y;\omega,v)$ are evaluated at $(b, a) = (\sqrt{\hat{q}_{x}}\xi, \hat{q}_{x})$ and $(\omega, v)=\big(\sqrt{\frac{\hat{q}_z}{\hat{Q}_z(\hat{Q}_z + \hat{q}_z)}}\xi, \frac{1}{\hat{Q}_z + \hat{q}_z}\big)$ respectively.
Note in particular that eq.~\eqref{eq:self_consistent_gamma} has to be solved over $(\gamma_x,\gamma_z)$ in order to be iterated.
Since the algorithmic performance is characterized by precisely maximizing eq.~\eqref{eq:replica:freeen} starting from the trivial point, the algorithmic weak-recovery threshold $\alpha_{\WR, \alg}$ can be analytically computed from a local stability analysis of this point. Note that in general $\alpha_{\WR,\IT} \neq \alpha_{\WR,\alg}$ since $q_{x}=q_{z}=0$ can be just a local maximum of eq.~\eqref{eq:replica:freeen}.

\paragraph{Existence and location of the weak-recovery threshold ---}
It is easy to verify that the state evolution equations~\eqref{eq:se_general}
admit a trivial fixed point in which $q_x = q_z = \hat{q}_x = \hat{q}_z = \gamma_x = \gamma_z = 0$ when
 $P_0$ and $P_\mathrm{out}$ are \emph{symmetric}, that is when for any $y \in \bbR$ and $x_1,x_2,z_1,z_2 \in \bbK$:
\begin{align}\label{eq:def_symmetry}
    |x_1| = |x_2| \Rightarrow P_0(x_1) = P_0(x_2) \hspace{1cm} \mathrm{and} \hspace{0.5cm} |z_1| = |z_2| \Rightarrow P_\mathrm{out}(y|z_1) = P_\mathrm{out}(y|z_2).
\end{align}
In particular, this symmetry condition holds for the phase retrieval likelihood and for Gaussian signals considered here.
When it exists, the trivial extremizer $q_{x}=q_{z}=0$ can be a (local) maximum or a minimum, corresponding to whether the trivial fixed point of the state evolution equations is stable or unstable. The weak-recovery threshold can therefore be determined by looking at the Jacobian around the trivial fixed point. The details of the stability analysis are given in Appendix \ref{sec:app:weakrecovery}. The result is that a linear instability of the trivial fixed point appears at $\alpha = \alpha_{\WR, \alg}$ satisfying the equation:
\begin{align}\label{eq:wr_threshold_general}
    \alpha_{\WR, \alg} = \frac{\langle \lambda \rangle_\nu^2}{\langle \lambda^2 \rangle_\nu}\Big(1+\Big[\int_\bbR \mathrm{d}y \frac{\Big|\int_\bbK \mathcal{D}_\beta z \ (|z|^2-1) \  P_\mathrm{out}\big(y\big|\sqrt{\frac{\rho \langle \lambda \rangle_\nu}{\alpha_{\WR,\alg}}} z\big)\Big|^2}{\int_\bbK \mathcal{D}_\beta z \ P_\mathrm{out}\big(y\big|\sqrt{\frac{\rho \langle \lambda \rangle_\nu}{\alpha_{\WR, \alg}}} z\big)}\Big]^{-1}\Big).
\end{align}
Note that the integrand and the averages $\langle \cdot \rangle_\nu$ depend on $\alpha_{\WR,\alg}$, so that this is an implicit equation on $\alpha_{\WR, \alg}$.
Eq.~\eqref{eq:wr_threshold_general} is the most generic formula for the weak recovery threshold for any data matrix $\bPhi$ and phase retrieval channel $P_\mathrm{out}$.
As emphasized in the following examples, it generalizes in particular several previously known formulas for different channels and random matrix ensembles.
\paragraph{Gaussian sensing matrix --- }
For Gaussian i.i.d.\ matrices, $\langle \lambda \rangle_\nu = \alpha$ and $\langle \lambda^2 \rangle_\nu = \alpha^2+\alpha$, so that
\begin{align}\label{eq:wr_threshold_gaussian_matrix}
    \alpha_{\WR, \alg} &= \Big[\int_\bbR \mathrm{d}y \frac{|\int_\bbK \mathcal{D}_\beta z \ (|z|^2-1) \  P_\mathrm{out}(y|\sqrt{\rho} z)|^2}{\int_\bbK \mathcal{D}_\beta z \ P_\mathrm{out}(y|\sqrt{\rho} z)}\Big]^{-1},
\end{align}
a result which was previously derived in \cite{Mondelli_2018} in the real and complex cases.

\paragraph{Noiseless phase retrieval --- }
In the noiseless phase retrieval problem, one has $P_\mathrm{out}(y|z) = \delta(y-|z|^2)$. In particular, one can easily check that this implies:
\begin{align}\label{eq:wr_threshold_phase_retrieval}
    \alpha_{\WR, \alg} &= \Big(1+\frac{\beta}{2}\Big) \frac{\langle \lambda \rangle_\nu^2}{\langle \lambda^2 \rangle_\nu}.
\end{align}
This last formula allows to retrieve and generalize many results previously derived in the literature. For instance, for a Gaussian i.i.d.\ matrix, we find $\alpha_{\WR, \alg} = \beta/2$ , which was derived in \cite{Mondelli_2018,luo2019optimal}. 
For an orthogonal or unitary column matrix, $\alpha_{\WR, \alg} = 1 + (\beta/2)$, which was already known for $\beta = 2$ \cite{Mondelli_2018} (but not for $\beta = 1$). For the product of $p$ i.i.d.\ Gaussian matrices with sizes $k_0,\cdots,k_p$, with $k_0 = m$ and $k_p = n$, and $\gamma_l \equiv n/k_l$ for $0\leq l<p$, we have $\alpha_{\WR, \alg} = (\beta/2) [1+\sum_{l=1}^p \gamma_l]^{-1}$, which generalizes the previously-known real case \cite{aubin2019exact}. We emphasize that eq.~\eqref{eq:wr_threshold_phase_retrieval} encapsulates all these results and goes beyond by considering an arbitrary spectrum for the sensing matrix, while eq.~\eqref{eq:wr_threshold_general} also considers 
arbitrary channels $P_\mathrm{out}$.

\paragraph{The weak-recovery IT transition --- } So far, we only considered the \emph{algorithmic} weak-recovery threshold.
Extending our analysis to the \emph{information-theoretic} treshold $\alpha_{\WR,\IT}$ is an interesting open direction, which requires understanding the appearance of a global maximum in the replica-symmetric potential of eq.~\eqref{eq:replica:freeen}, but not necessarily continuously from the $q_x\!=q_z\!=\!0$ solution.
At the moment, we are not able to carry such an analysis, which is left for future work.
\begin{figure}
\centering
\includegraphics[scale=0.45]{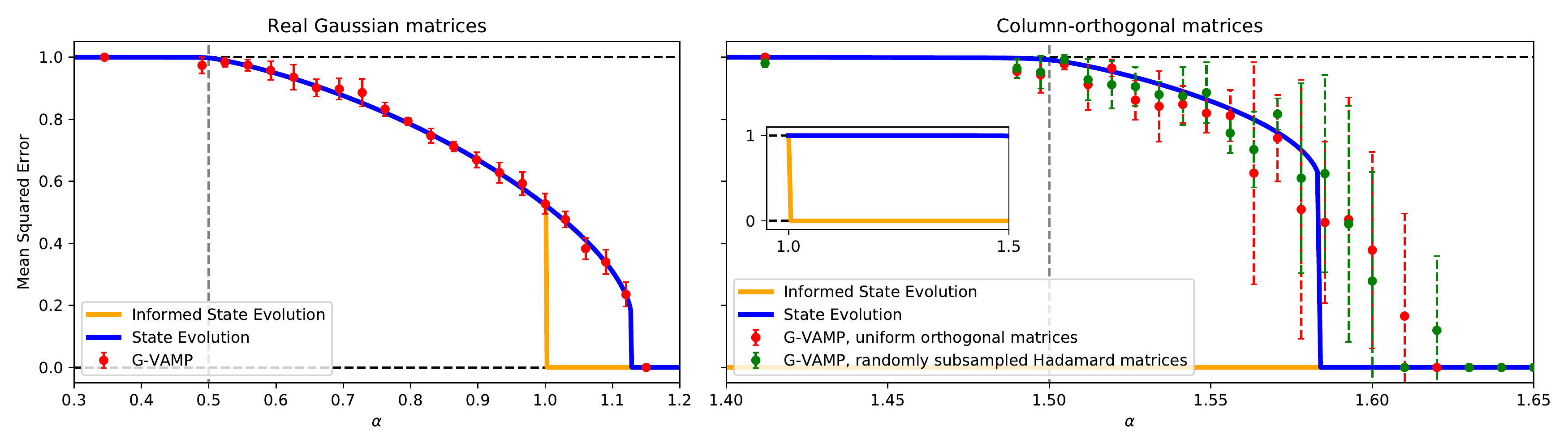}
\caption{
Comparison of MSE achieved by the Bayes-optimal estimator and the G-VAMP algorithm, for an i.i.d.\ real Gaussian (left) and a column-orthogonal (right) sensing matrix $\bPhi$ (i.e.\ $\bPhi^\intercal \bPhi /n = \mathds{1}_n$), with a real Gaussian prior.
Dots correspond to finite size simulations of G-VAMP (the mean and std are taken over $5$ instances, with $n = 8000$ in the Gaussian case and $m = 8192$ in the orthogonal case), while full lines are obtained from the state evolution equations. 
The vertical grey dashed lines denote the algorithmic weak recovery threshold $\alpha_{\WR,\alg}$. 
Note the presence of a statistical-to-algorithmic gap in both ensembles, and that for column-orthogonal matrices $\alpha_\mathrm{WR,\alg} > \alpha_\mathrm{FR,IT}$. 
}
\label{fig:mse_orthogonal}
\end{figure}
\section{Statistical and algorithmic analysis of 
noiseless phase retrieval}\label{sec:numerical}
While our results hold for any generalized estimation problem of the type introduced in Section~\ref{sec:replica_result_and_proof} we now focus especially on noiseless phase retrieval.
We fix $P_{\text{out}}(y|z)\!=\! \delta(y-|z|^2)$ and take $P_{0}\!=\! \mathcal{N}_\beta(0,1)$. We can indeed consider $\rho\!=\!1$, as the scaling is irrelevant under a noiseless channel. 

\paragraph{Full-recovery threshold for Gaussian signals ---}
We now turn our attention to the information-theoretical \emph{full-recovery} threshold $\alpha_{\FR,\IT}$. For high number of samples $\alpha\gg 1$, we expect the MMSE to plateau at a minimum achievable reconstruction error $\MMSE_{0} \equiv \inf_{\alpha}\MMSE(\alpha)$, which is a function of the statistics of $\bPhi$. In this case, we define the information-theoretical full-recovery threshold $\alpha_{\FR, \IT}$ as the smallest sample complexity such that $\MMSE_{0}$ is attained.
In Appendix~\ref{sec:app_perfect_recovery} we show that the full-recovery can be \emph{perfect} ($\MMSE_{0} = 0$) or \emph{partial} ($\MMSE_{0} > 0$) depending on the rank of $\bPhi$. Indeed, we show that:
\begin{align}\label{eq:full_recovery_it}
  \alpha_{\FR, \IT} &\equiv \beta(1-\nu(\{0\})).
\end{align}
Informally, $\nu(\{0\})$, the fraction of zeros in the spectrum of $\bPhi^\dagger \bPhi /n$, is the fraction of the signal ``lost'' by the sensing matrix.
The stationary point of eq.~\eqref{eq:replica:freeen} that corresponds to full recovery satisfies $\mathrm{MMSE}_0 = \nu(\{0\})$, while the reconstruction of the vector $\bPhi \bx$ is perfect.
The effect of rank deficiency is illustrated in  Fig.~\ref{fig:mse_complex_product_gaussians}-left, with the case of $\bPhi$ given by a product of two Gaussian matrices.
We emphasize that $\alpha_{\FR, \IT}$ is in general not well-defined for an arbitrary channel, which is why we only derived eq.~\eqref{eq:full_recovery_it} in the noiseless case.

\paragraph{Evaluation of the thresholds and comparison to simulations ---} 
Algorithmic weak-recovery and information-theoretical full-recovery thresholds can be readily obtained from eqs.~\eqref{eq:wr_threshold_phase_retrieval},\eqref{eq:full_recovery_it}.  Below, we solve the state evolution equations~\eqref{eq:se_general} for different real and complex ensembles of sensing matrix $\bPhi$, and compare it to numerical simulations of G-VAMP.

\paragraph{Real case ---} The case of a real signal $\bX^\star\in\mathbb{R}^{n}$ has been previously studied in the literature for particular ensembles of real-valued sensing matrix $\bPhi$. A formula analogous to eq.~\eqref{eq:replica:freeen} has been heuristically derived for real orthogonally invariant matrices $\bPhi$ and real signals drawn from generic but separable $P_{0}$ \cite{takahashi2020macroscopic}, and the specific i.i.d.\ Gaussian matrix case was rigorously proven in \cite{barbier2019optimal}. The heuristic analysis was later extended to non-separable signal distributions $P_{0}$ \cite{aubin2019exact}. 
In Fig.~\ref{fig:mse_orthogonal}, we illustrate the case of real Gaussian and real column-orthogonal sensing matrix $\bPhi$, the latter not having been investigated previously in the literature. We compute the MMSE by solving the State Evolution equations starting from an \emph{informed} solution (close to full recovery). 
The minimal mean-squared error achievable with the G-VAMP algorithm is computed using the State Evolution equations starting from the \emph{uninformed} $q_z = 0$ solution. 
We compare these predictions with numerical simulations of the G-VAMP algorithm on Gaussian matrices and uniformly sampled orthogonal matrices, as well 
as randomly subsampled Hadamard matrices. The simulations are in very good agreement with the prediction, and
our results on Hadamard matrices suggest that the curves of Fig.~\ref{fig:mse_orthogonal}-right are valid for more general ensembles than uniformly sampled orthogonal matrices, and that one can allow some controlled structure in the matrix without harming the performance of the algorithm.

\paragraph{Complex case ---} Previous works on complex signals $\bX^\star\in\mathbb{C}^{n}$ have (to the best of our knowledge) focused solely on the study of the weak recovery threshold $\alpha_{\rm WR}$ (statistical or algorithmic), which was located for i.i.d.\ complex Gaussian matrices \cite{Mondelli_2018,luo2019optimal} and uniformly sampled column-unitary matrices \cite{ma2017orthogonal, dudeja2020analysis}. 
We begin by extending the aforementioned results by identifying the full recovery threshold $\alpha_{\mathrm{FR,IT}}$ in these cases, and comparing the performance of the G-VAMP algorithm to the SE solution.  
Fig.~\ref{fig:mse_complex} illustrates our results for these two ensembles.
The algorithmic full-recovery threshold $\alpha_{\mathrm{FR,\alg}}$ is found numerically from the state evolution equations
and is in good agreement with finite size simulations.
The existence of a statistical-to-algorithmic gap $\Delta = \alpha_{\FR,\alg} - \alpha_{\FR,\IT}\geq 0$ reflects the intrinsic hardness of phase retrieval in the real and complex case. However, it is interesting to note that even though full-recovery in the complex case requires more data than in the real case, the size of the statistical-to-algorithmic gap in the complex ensembles is smaller than in their real counterparts.
\begin{figure}
  \centering
\includegraphics[scale=0.45]{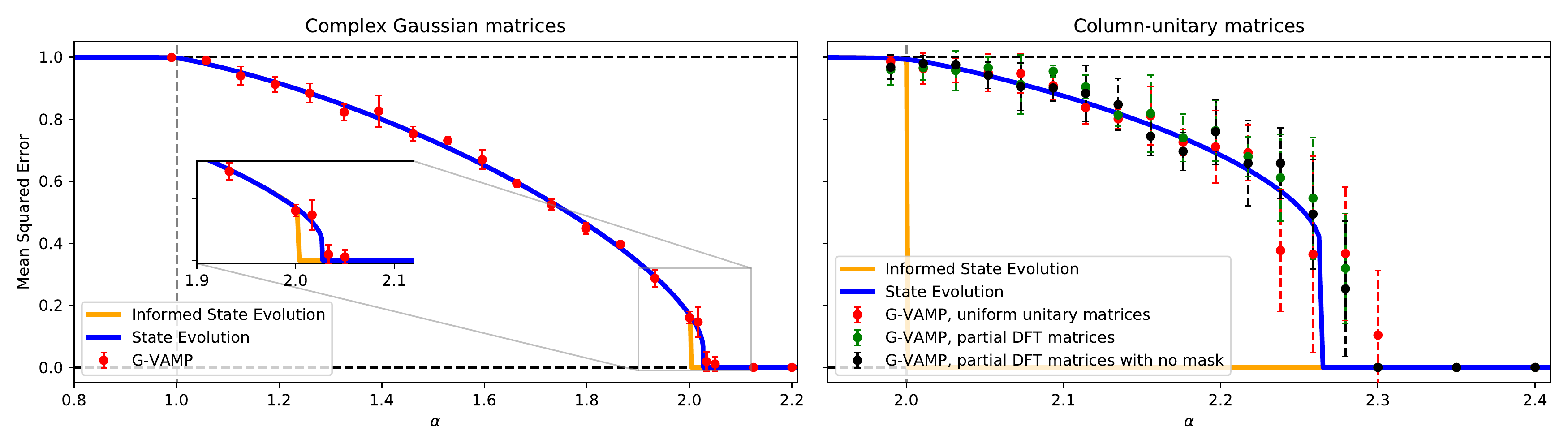}
\caption{Comparison of MSE achieved by the Bayes-optimal estimator and G-VAMP algorithm for phase retrieval, for the case of an i.i.d.\ complex Gaussian (left) and a column-unitary (right) sensing matrix $\bPhi$ (i.e.\ $\bPhi^\dagger \bPhi / n = \mathds{1}_n$), with a complex Gaussian prior. 
Dots correspond to finite size simulations of G-VAMP (with $n = 5000$, the mean and std are taken over $5$ independent instances), while full lines are obtained from the state evolution equations. Note the presence of a statistical-to-algorithmic gap in both ensembles. }
\label{fig:mse_complex}
\end{figure}
In Fig.~\ref{fig:mse_complex_product_gaussians} we analyze the case of a product of two i.i.d.\ standard Gaussian matrices $\bPhi = \bW_1 \bW_2$, with $\bW_{1}\in\mathbb{C}^{m\times p}$ and $\bW_{2}\in\mathbb{C}^{p\times n}$ for different aspect ratios $\gamma \equiv p/n$. 
We can identify the presence of a threshold
$\alpha_{\WR,\alg} = \gamma / (1+\gamma)$ (computed in Section~\ref{sec:weak_recovery}) that delimits the possibility of weak recovery both information-theoretically and in polynomial time.
The information-theoretic full-recovery is achieved at $\alpha_\mathrm{FR,IT} = \min(2,2\gamma)$, in agreement with eq.~\eqref{eq:full_recovery_it}.
Consistently with the real case results of \cite{aubin2019exact}, the full recovery algorithmic threshold is very close to the information-theoretic one, and precisely equal for $\gamma=1$, although the gap is too small to be visible in the left and right parts of Fig.~\ref{fig:mse_complex_product_gaussians}. Therefore, the performance of G-VAMP is exactly given by the Bayes-optimal estimator, apart for $\gamma \neq 1$ in a very small range $(\alpha_{\FR,\IT},\alpha_{\FR,\alg})$, whose size is of order $10^{-3}$ for $\gamma \in \{0.5,1.5\}$. As $\gamma \to \infty$, one recovers the statistical-to-algorithmic gap present in the complex Gaussian case, which is again very small (around $0.027$, cf Table~\ref{table:results}). Although this hard phase is very small, we therefore postulate its existence for all $\gamma \neq 1$, generalizing the real case results of \cite{aubin2019exact}.
\begin{figure}[ht!]
\centering
\includegraphics[scale=0.435]{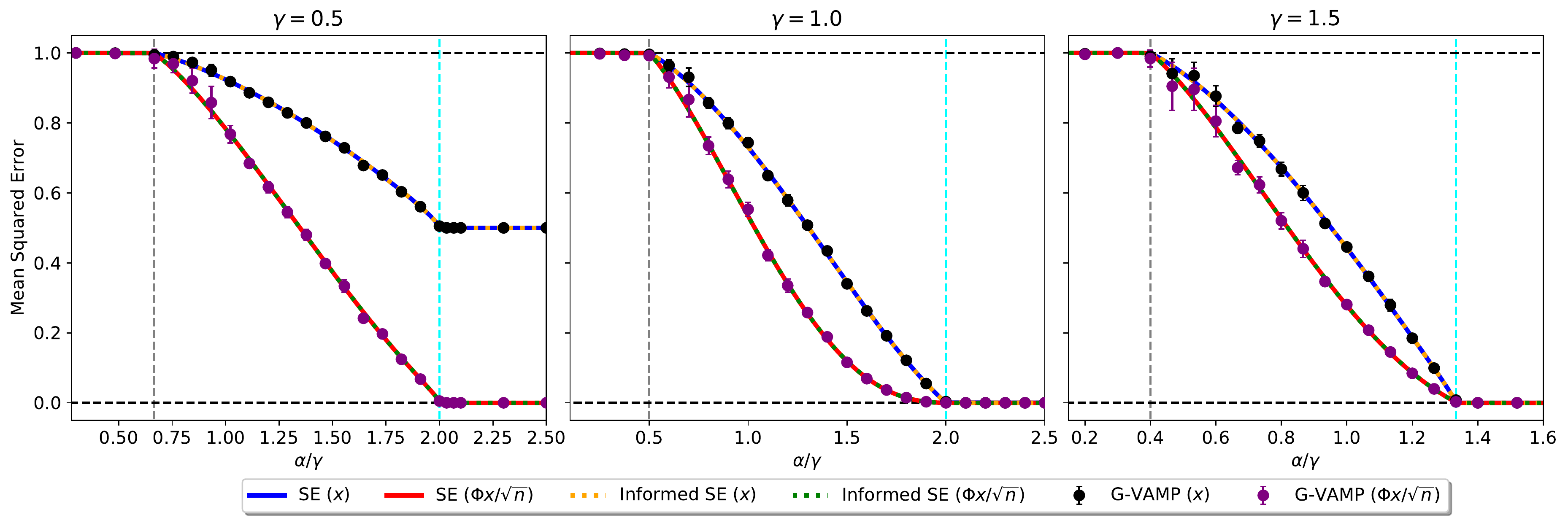}
\caption{Mean squared error as a function of the measurement rate $\alpha$, for a sensing matrix $\bPhi = \bW_{1}\bW_{2}$ a product of two complex i.i.d.\ standard Gaussian matrices $\bW_{1}\in\bbC^{m\times p}$, $\bW_{2}\in\bbC^{p\times n}$ with aspect ratios $\gamma = p/n \in \{0.5,1.0,1.5\}$. Red curves denote the recovery on $\bPhi\bX^{\star}/\sqrt{n}$ and blue curves on $\bX^{\star}$.
Cyan dashed lines denote the full reconstruction threshold $\alpha_{\FR,\IT}$.
The G-VAMP experiments were performed with $n = 5000$, and the mean and std are taken over $5$ instances.
}
\label{fig:mse_complex_product_gaussians}
\end{figure}

\paragraph{Application to images ---}
Importantly, while the knowledge of the distribution of the true signal is required for our theoretical analysis, 
the G-VAMP algorithm is also well-defined beyond this scope, e.g.\ it can be used to infer natural images with Fourier matrices. 
Using a Gaussian prior to infer the image can then actually be seen as the minimal assumption on the underlying signal, as it amounts to simply fix its norm: our theory can thus predict the performance of this G-VAMP algorithm for any signal, structured or not. 
We conducted a simple experiment on a natural image with a randomly subsampled DFT matrix $\bPhi$, described in Fig.~\ref{fig:real_image}.
Although we are far from a Bayes-optimal setting, the achieved MSE is very close to values of Fig.~\ref{fig:mse_complex}
of the paper, for all values of $\alpha$. 
In particular, we achieve perfect recovery for $\alpha\!\geq\!2.3$, just above $\alpha_{\FR,\alg}\!\simeq\!2.27$ which was derived for random unitary matrices, i.i.d.\ data and in the Bayes-optimal setting. 
\begin{figure}
\centering
\includegraphics[width=0.95\textwidth]{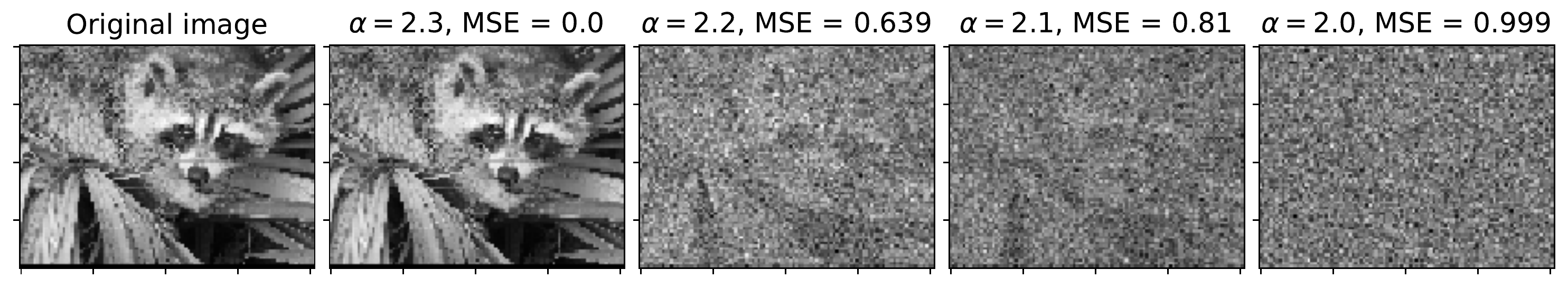}
\caption{Performance of the G-VAMP algorithm for noiseless phase retrieval. We wish to recover a 77x102 image (on the left), and we use a complex Gaussian prior to infer the signal. The data matrix $\bPhi$ is a randomly subsampled DFT matrix.}
    \label{fig:real_image}
\end{figure}

\newpage 
\section*{Acknowledgments}
The authors would like to thank Yoshiyuki Kabashima for insightful discussions on the replica computations with orthogonally invariant matrices, and Yue M.\ Lu for fruitful discussions at the beginning of this work.
Additional funding is acknowledged by AM from ``Chaire de
recherche sur les mod\`eles et sciences des donn\'ees'', Fondation CFM pour la Recherche-ENS.  This work is supported by the ERC under the European Union’s Horizon 2020 Research and Innovation Program
714608-SMiLe, as well as by the French Agence Nationale de la Recherche under grant ANR-17-CE23-0023-01
PAIL and ANR-19-P3IA-0001 PRAIRIE.

\bibliographystyle{alpha}
\bibliography{refs}

\changelocaltocdepth{1}
\appendix
\newpage
\begin{center}
 \LARGE SUPPLEMENTARY MATERIAL
\end{center}
\noindent
Many notations and definitions used throughout this supplementary material are given in Sections~\ref{subsec:app_definitions},\ref{subsec:app_derivatives}.
The Python code that produced the numerical data used in Figures~\ref{fig:mse_orthogonal},\ref{fig:mse_complex},\ref{fig:mse_complex_product_gaussians}, as well as the data itself, are given in the following \href{https://github.com/sphinxteam/PhaseRetrieval_demo}{Github repository} \cite{github_demo_repo}, and is dependent on the open-source TrAMP library \cite{baker2020tramp}. We provide in particular an ``example'' notebook which contains a detailed presentation of the functions necessary to generate both the state evolution and the G-VAMP data for the complex Gaussian matrix case.

\section{The replica computation of the free entropy}
\label{sec:app:replicas}

In this section, which has a more pedagogical purpose, we perform the replica calculation that gives Conjecture~\ref{conjecture:replicas_general}.
This calculation for real matrices was already performed in \cite{takahashi2020macroscopic},
and as we will see it generalizes to complex valued signal and matrices.
Note that we restricted ourselves to a Bayes-optimal inference problem, while the setting of \cite{takahashi2020macroscopic}
includes possibly mismatched models\footnote{For a mismatched model, the replica symmetry assumption, discussed below, is generically not valid.}. 

\subsection{Setting}

We let $n,m \to \infty$ with $m/n \to \alpha > 0$. We assume that we have access to a prior distribution $P_0$ on $\bbK$ and a channel distribution 
$P_{\rm out}(y|z)$, of ``observations'' $y \in \bbR$ conditioned by a latent variable $z \in \bbK$. 
We are given data $\bY \in \bbR^m$ generated as:
\begin{align*}
    Y_\mu \sim P_\mathrm{out}\Big(\cdot \Big| \frac{1}{\sqrt{n}} \sum_{i=1}^n \Phi_{\mu i} X^\star_i\Big),
\end{align*}
in which $X^\star_i \overset{\mathrm{i.i.d.}}{\sim} P_0$ (with $\EE|X^\star|^2 = \rho > 0$), and $\bPhi \in \bbK^{m \times n}$ is a matrix that is both left and right orthogonally (respectively unitarily) invariant, 
meaning that for all $\bO,\bU \in \mathcal{U}_\beta(m) \times \mathcal{U}_\beta(n)$, $\bPhi \overset{d}{=} \bO \bPhi \bU$. Compared to Conjecture~\ref{conjecture:replicas_general}, we added a left-invariance hypothesis. However the analysis of G-VAMP \cite{rangan2017vector,schniter2016vector} shows that this left invariance is actually not needed for the result, and thus we state Conjecture~\ref{conjecture:replicas_general} for matrices that are only right-invariant, but we use the left invariance to simplify the following (heuristic) calculation.
Moreover, we assume that the asymptotic eigenvalue distribution of $\bPhi^\dagger \bPhi / n$ is well-defined and we denote it $\nu$, and that the eigenvalue distribution of $\bPhi^\dagger \bPhi / n$ has large deviations in a scale at least $n^{1+\eta}$ for an $\eta > 0$.
The partition function is:
\begin{align*}
    \mathcal{Z}_n(\bY) &\equiv \int_{\bbK^n} \prod_{i=1}^n P_0(\mathrm{d}x_i) \ \prod_{\mu=1}^m P_{\mathrm{out}}\Big(Y_\mu\Big|\frac{1}{\sqrt{n}} \sum_{i=1}^n \Phi_{\mu i} x_i\Big).
\end{align*}
The replica trick \cite{mezard1987spin} consists in computing the $p$-th moment of the partition function for arbitrary integer $p$, 
before extending this expression analytically to any $p > 0$ and using the formula:
\begin{align*}
    \lim_{n \to \infty} \frac{1}{n} \EE_{\bPhi,\bY} \ln \mathcal{Z}_n(\bY) &= \lim_{p \downarrow 0} \lim_{n \to \infty} \frac{1}{np} \ln \EE_{\bPhi,\bY}[\mathcal{Z}_n(\bY)^p].
\end{align*}
This method is obviously non-rigorous given the inversion of limits $p \downarrow 0$ and $n \to \infty$,
as well as the analytic continuation to arbitrary $p > 0$ of the $p$-th moment. 
However, it has achieved tremendous success in the study of spin glasses and inference problems, see e.g.\ \cite{zdeborova2016statistical}.

\subsection{Computing the \texorpdfstring{$p$}{p}-th moment of the partition function}

Thanks to Bayes-optimality, we can easily write the average of $\mathcal{Z}_n(\bY)^p$ as an average over $p+1$ replicas of the system, 
by considering $\bX^\star$ as the replica of index $0$. We obtain for any $p \geq 1$:
\begin{align}\label{eq:pthmoment_1}
    \EE[\mathcal{Z}_n(\bY)^p] &= \EE_{\bPhi} \int_{\bbR^m} \mathrm{d}\bY \prod_{a=0}^p \Big\{\Big[\int_\bbK \prod_{i=1}^n P_0(\mathrm{d}x_i^a)\int_\bbK \prod_{\mu=1}^m \mathrm{d}z_\mu^a P_\mathrm{out}(Y_\mu|z_\mu^a) \Big]\delta\Big(\bz^a - \frac{\bPhi \bx^a}{\sqrt{n}}\Big)\Big\}.
\end{align}
The first step is to decompose eq.~\eqref{eq:pthmoment_1} into three terms, corresponding to the prior $P_0$, the channel $P_\mathrm{out}$, and the ``delta'' term. 
Note that the matrix $\bPhi$ only appears in the last ``delta'' term.
By left and right orthogonal (resp.\ unitary) invariance of $\bPhi$, the quantity
\begin{align*}
    \EE_{\bPhi} \Big[\prod_{a=0}^p\delta\Big(\bz^a - \frac{1}{\sqrt{n}} \bPhi \bx^a\Big)\Big]
\end{align*}
is determined by the value of the \emph{overlaps} $\bQ^z \equiv \{(\bz^a)^\dagger \bz^b / m\}_{a,b=0}^p$ and $\bQ^x \equiv \{(\bx^a)^\dagger \bx^b / n\}_{a,b=0}^p$, which are positive symmetric (Hermitian in the complex case) matrices. 
As is standard in such replica calculations, we will constraint the terms in eq.~\eqref{eq:pthmoment_1} by the value of these overlaps, 
before performing a Laplace method on the resulting function of the overlaps.
By $A_n \simeq B_n$, we will mean equivalence at leading exponential order, that is $ (\ln A_n)/n = (\ln B_n)/n + \smallO_n(1)$.
We introduce in eq.~\eqref{eq:pthmoment_1} the term:
\begin{align*}
    1 \simeq \int \prod_{0 \leq a \leq b \leq p} \mathrm{d}Q^x_{ab}\ \mathrm{d}Q^z_{ab} \Big[\prod_{a \leq b} \delta(n Q^x_{ab} - (\bx^a)^\dagger  \bx^b)\Big]\Big[\prod_{a \leq b} \delta(m Q^z_{ab} - (\bz^a)^\dagger \bz^b)\Big].
\end{align*}
We can use a Fourier transformation of the delta terms, which allows in the end to transform eq.\eqref{eq:pthmoment_1} into the product of three independent terms.
Performing the saddle-point on $\bQ^x,\bQ^z$, we obtain the corresponding result:
\begin{align*}
    \lim_{n \to \infty} \frac{1}{n} \ln \EE_{\bY,\bPhi} [\mathcal{Z}_n(\bY)^p] &= \sup_{\bQ^x,\bQ^z} [I_0(p,\bQ^x) + \alpha I_\mathrm{out}(p,\bQ^z) + I_\mathrm{int}(p,\bQ^x,\bQ^z)],
\end{align*}
in which the supremum is made over positive symmetric (Hermitian) matrices, and $I_0,I_\mathrm{out}$ and $I_\mathrm{int}$ are functions whose calculation will be detailed below.

\subsubsection{The prior term \texorpdfstring{$I_0(p,\bQ^x)$}{}}

We have by the Laplace method after Fourier transformation of the delta terms:
\begin{align*}
    I_0(p,\bQ^x) &\simeq \frac{1}{n} \ln \int \prod_{0 \leq a \leq b \leq p} \mathrm{d}\hat{Q}^x_{ab} \int_\bbK \prod_{a=0}^p \prod_{i=1}^n P_0(\mathrm{d}x_i^a) e^{-\frac{\beta}{2} \sum_{a,b =0 }^p \hat{Q}^x_{ab} (\sum_i \overline{x^a_i} x^b_i - n Q^x_{ab} )}, \\
    &\simeq \inf_{\hat{\bQ^x}} \Big[\frac{\beta}{2} \sum_{a,b} Q^x_{ab} \hat{Q}^x_{ab} + \ln \int_\bbK \prod_{a=0}^p P_0(\mathrm{d}x^a) e^{-\frac{\beta}{2} \sum_{a,b} \hat{Q}^x_{ab} \overline{x^a} x^b}\Big].
\end{align*}
The infimum is again over positive symmetric (Hermitian) matrices.
We also made use of the fact that the prior $P_0$ is i.i.d.\ over the elements of $\bx$.
A very important assumption of our calculation is replica symmetry. It amounts to assume that all the $(p+1)$ replicas 
are equivalent, and that this symmetry is not broken by the system at the solution of the Laplace method. 
Replica symmetry and replica symmetry breaking are a very rich field of study in statistical physics \cite{mezard1987spin}.
It has been argued that for an inference problem in the Bayes-optimal setting (as is the present case), replica symmetry is never broken \cite{zdeborova2016statistical}.
We can therefore assume a replica symmetric form of $\bQ^x,\hat{\bQ}^x$ at the point at which the saddle point is reached, that we write as:
\begin{align}\label{eq:rs_Qx}
    \bQ^x &= \begin{pmatrix}
        Q_x & q_x & \cdots & q_x \\
        q_x & Q_x & \cdots & q_x \\
        \vdots & \vdots & \ddots & \vdots \\
        q_x & q_x & \cdots & Q_x
    \end{pmatrix}, \hspace{1cm}
    \hat{\bQ}^x = \begin{pmatrix}
        \hat{Q}_x & -\hat{q}_x & \cdots & -\hat{q}_x \\
        -\hat{q}_x & \hat{Q}_x & \cdots & -\hat{q}_x \\
        \vdots & \vdots & \ddots & \vdots \\
        -\hat{q}_x & -\hat{q}_x & \cdots & \hat{Q}_x
    \end{pmatrix}.
\end{align}
Note that for $\beta \in \{1,2\}$, we have $Q_x,q_x,\hat{Q}_x,\hat{q}_x \in \bbR$.
After a simple Gaussian transformation of the squared term using the general identity for $x \in \bbK$:
\begin{align*}
    \exp\Big(\frac{\beta}{2} |x|^2\Big) &= \int_\bbK \mathcal{D}_\beta \xi \ \exp(\beta x \cdot \xi),
\end{align*}
we reach the final expression:
\begin{align}\label{eq:prior_term}
     &I_0(p,Q_x,q_x) =  \\
     &\inf_{\hat{Q}_x,\hat{q}_x} \Big\{\frac{\beta(p+1)}{2} Q_x \hat{Q}_x - \frac{\beta p(p+1)}{2} q_x \hat{q}_x + \ln \int_\bbK \mathcal{D}_\beta\xi \Big[\int_\bbK P_0(\mathrm{d}x) e^{-\frac{\beta(\hat{Q}_x + \hat{q}_x)}{2} |x|^2 + \beta \sqrt{\hat{q}_x} x \cdot \xi }\Big]^{p+1}\Big\}.
     \nonumber
\end{align}

\subsubsection{The channel term \texorpdfstring{$I_\mathrm{out}(p,\bQ^z)$}{}}

This term is very similar to the prior term detailed in the previous section.
We use completely similar replica symmetric assumptions for the overlaps $\bQ^z$ to the ones on $\bQ^x$ described in eq.~\eqref{eq:rs_Qx}.
We reach:
\begin{align}\label{eq:channel_term}
    I_\mathrm{out}(p,Q_z,q_z) &= \inf_{\hat{Q}_z,\hat{q}_z} \Big\{\frac{\beta(p+1)}{2} Q_z \hat{Q}_z - \frac{\beta p(p+1)}{2} q_z \hat{q}_z + \frac{\beta(p+1)}{2} \ln (2 \pi/ (\beta \hat{Q}_z)) \Big. \\
    &\Big. + \ln \int_\bbR \mathrm{d}y \int_\bbK \mathcal{D}_\beta\xi \Big[\int_\bbK \mathrm{d}z\Big(\frac{2 \pi}{\beta \hat{Q}_z}\Big)^{-\beta/2} P_\mathrm{out}(y|z) \ e^{-\beta\frac{\hat{Q}_z + \hat{q}_z}{2} |z|^2 + \beta \sqrt{\hat{q}_z} z \cdot \xi }\Big]^{p+1}\Big\}. \nonumber
\end{align}
We normalized the integrals so that in the limit $p \to 0$, the term inside the logarithm goes to $1$, which will be a useful remark.

\subsubsection{The delta term \texorpdfstring{$I_\mathrm{int}(p,\bQ^x,\bQ^z)$}{}}\label{subsubsec:Ic_computation}
We now turn to the computation of the delta term:
\begin{align}\label{eq:2}
   I_\mathrm{int}(p,\bQ^x,\bQ^z) &\equiv \lim_{n \to \infty}\frac{1}{n} \ln \EE_{\bPhi} \Big[\prod_{a=0}^p\delta\Big(\bz^a - \frac{1}{\sqrt{n}} \bPhi \bx^a\Big)\Big],
\end{align}
assuming that $\bQ^x,\bQ^z$ are known. 
Computing this term is central in this replica calculation.
We use, as is done in \cite{takahashi2020macroscopic}, the identity:
\begin{align}\label{eq:3}
   \frac{1}{n} \ln \EE_{\bPhi} \Big[\prod_{a=0}^p\delta\Big(\bz^a - \frac{1}{\sqrt{n}} \bPhi \bx^a\Big)\Big] &= \lim_{\epsilon \downarrow 0} \frac{1}{n} \ln \EE_{\bPhi} \Big[\frac{\exp\big\{-\frac{\beta}{2\epsilon} \sum_a\big\lVert\bz^a - \frac{1}{\sqrt{n}} \bPhi \bx^a \big\rVert^2\big\}}{(2\pi \epsilon / \beta)^{\frac{\beta m(p+1)}{2}}}\Big],
\end{align}
and we invert the $n \to \infty$ and the $\epsilon \to 0$ limit.
Let us rewrite the 
right-hand-side of eq.~\eqref{eq:3}. Since $\bPhi$ is orthogonally (resp.\ unitarily) invariant, we can write this term as:
\begin{align}\label{eq:4}
\EE \Big[\frac{\exp\big\{-\frac{\beta}{2\epsilon} \sum_a\big\lVert\bz^a - \frac{1}{\sqrt{n}} \bPhi \bx^a \big\rVert^2\big\}}{(2\pi \epsilon / \beta)^{\frac{\beta m(p+1)}{2}}}\Big]
    &= \EE \Big[\frac{\exp\Big\{-\frac{\beta}{2\epsilon} \sum_a\big\lVert\bO\bz^a - \frac{1}{\sqrt{n}} \bPhi \bU \bx^a \big\rVert^2\Big\}}{(2\pi \epsilon / \beta)^{\frac{\beta m(p+1)}{2}}}\Big],
\end{align}
in which the average on the right hand side is made over $(\bPhi,\bO,\bU)$, with $(\bO,\bU)$ uniformly sampled over the orthogonal groups $\mathcal{U}_\beta(m),\mathcal{U}_\beta(n)$.
Note that since the overlap matrices $\bQ^z,\bQ^x$ are fixed, one can show that when $\bU$ is uniformly distributed over $\mathcal{U}_\beta(n)$, the set of vectors $\{\bU \bx^a\}_{a=0}^p$ 
is uniformly distributed over the set of $(p+1)$ vectors in $\bbK^n$ with overlap matrix $\bQ^x$. There is a completely similar result for $\bz$ as well.
The consequence is that we can replace in eq.~\eqref{eq:4} the average over $\bO,\bU$ by an average over the vectors satisfying this constraint:
\begin{align}\label{eq:5}
    &I_\mathrm{int}(p,\bQ^x,\bQ^z) \\
    & \simeq \frac{1}{n} \ln \EE_{\bPhi} \frac{\int_\bbK \prod_{a} \mathrm{d}\bx^a \ \mathrm{d}\bz^a \Big[\prod_{a\leq b} \delta(n Q_{ab}^x - (\bx^a)^\dagger \bx^b) \delta(m Q_{ab}^z - (\bz^a)^\dagger \bz^b)\Big] \frac{e^{-\frac{\beta}{2\epsilon} \sum_a \lVert\bz^a - \frac{1}{\sqrt{n}} \bPhi \bx^a \rVert^2}}{(2 \pi \epsilon / \beta )^{\beta m(p+1)/2}}}{\int_\bbK \prod_{a} \mathrm{d}\bx^a \ \mathrm{d}\bz^a \Big[\prod_{a\leq b} \delta(n Q_{ab}^x - (\bx^a)^\dagger \bx^b) \delta(m Q_{ab}^z - (\bz^a)^\dagger \bz^b)\Big] } \nonumber.
\end{align}
The numerator and the denominator correspond to two terms, that we denote $I_\mathrm{int}(p,\bQ^x,\bQ^z) = I_c^{(n)}(p,\bQ^x,\bQ^z) - I_c^{(d)}(p,\bQ^x,\bQ^z)$.
We can introduce the Fourier-transform of the delta distribution to compute both terms, as in the previous sections.
Let us start with the denominator. It reduces after Fourier-transformation to a Gaussian integral involving a block-diagonal matrix:
\begin{align*}
    I_\mathrm{int}^{(d)}(p,\bQ^x,\bQ^z) &\simeq \frac{\beta}{2}\inf_{\bGamma^x,\bGamma^z} \Big[\mathrm{Tr}[\bQ^x \bGamma^x] + \alpha \mathrm{Tr}[\bQ^z \bGamma^z] + (\alpha+1)(p+1) \ln \frac{2\pi}{\beta} \\ 
    & \hspace{1cm}- \ln \det \bGamma^x - \alpha \ln \det \bGamma^z \Big],
\end{align*}
with symmetric (Hermitian) positive matrices $\bGamma^x,\bGamma^z$ of size $(p+1)$. The infimum is readily solved by $\bGamma^x = (\bQ^x)^{-1}$ and $\bGamma^z = (\bQ^z)^{-1}$, which yields:
\begin{align}\label{eq:Icd}
    I_\mathrm{int}^{(d)}(p,\bQ^x,\bQ^z) &\simeq \frac{\beta(\alpha+1)(p+1)}{2}(1+ \ln \frac{2\pi}{\beta}) + \frac{\beta}{2} \ln \det \bQ^x + \frac{\alpha \beta}{2} \ln \det \bQ^z.
\end{align}
Let us now compute the numerator with the same technique. We obtain:
\begin{align}\label{eq:Icn}
    I_\mathrm{int}^{(n)}(p,\bQ^x,\bQ^z) &\simeq \frac{\beta(p+1)}{2} \ln \frac{2 \pi}{\beta \epsilon^\alpha} + \frac{\beta}{2} \inf_{\bGamma^x,\bGamma^z} \Big[\mathrm{Tr}[\bQ^x \bGamma^x] + \alpha\mathrm{Tr}[\bQ^z \bGamma^z] - \frac{1}{n}\ln \det \bM_n \Big],
\end{align}
with a Hermitian matrix $\bM_n$ having a block structure, that we write here in the tensor product form:
\begin{align}
    \bM_n &\equiv \begin{pmatrix}
       (\bGamma^z + \frac{1}{\epsilon} \mathbbm{1}_{p+1}) \otimes \mathbbm{1}_m & \frac{1}{\epsilon} \mathbbm{1}_{p+1} \otimes \frac{\bPhi}{\sqrt{n}} \\
       \frac{1}{\epsilon} \mathbbm{1}_{p+1} \otimes \frac{\bPhi^\dagger}{\sqrt{n}} & \bGamma^x \otimes \mathbbm{1}_n + \frac{1}{\epsilon} \mathbbm{1}_{p+1} \otimes \frac{\bPhi^\dagger \bPhi}{n}
    \end{pmatrix}.
\end{align}
Using the block-matrix determinant calculation:
\begin{align*}
    \det \begin{pmatrix}
        A & B \\
        C & D
    \end{pmatrix} &= \det A \times \det(D - C A^{-1} B),
\end{align*}
we reach:
\begin{align*}
    &\frac{1}{n} \ln \det \bM_n =  \alpha \ln \det\Big(\bGamma^z + \frac{1}{\epsilon} \mathbbm{1}_{p+1}\Big) \\ 
    & \nonumber \hspace{2cm} + \frac{1}{n} \ln \det \Big(\bGamma^x \otimes \mathbbm{1}_n + \frac{1}{\epsilon} \mathbbm{1}_{p+1} \otimes \frac{\bPhi^\dagger \bPhi}{n}- \frac{1}{\epsilon^2} \Big(\bGamma^z + \frac{1}{\epsilon} \mathbbm{1}_{p+1}\Big)^{-1} \otimes \frac{\bPhi^\dagger \bPhi}{n} \Big), \\ 
    &= (\alpha-1)\ln \det\Big(\bGamma^z + \frac{1}{\epsilon} \mathbbm{1}_{p+1}\Big) + \frac{1}{n} \ln \det \Big(\bGamma^x \bGamma^z \otimes \mathbbm{1}_n + \frac{1}{\epsilon} \bGamma^x \otimes \mathbbm{1}_n + \frac{1}{\epsilon} \bGamma^z \otimes \frac{\bPhi^\dagger \bPhi}{n}\Big) \nonumber, \\
    &= (\alpha-1)\ln \det\Big(\bGamma^z + \frac{1}{\epsilon} \mathbbm{1}_{p+1}\Big) + \Big \langle \ln \det \Big(\bGamma^x \bGamma^z + \frac{1}{\epsilon} (\bGamma^x + \lambda\bGamma^z) \Big) \Big\rangle_\nu ,
\end{align*}
with $\lambda$ distributed according to $\nu$, the asymptotic eigenvalue distribution of $\bPhi^\dagger \bPhi/n$.
This allows to write $I_\mathrm{int}^{(n)}$ from eq.~\eqref{eq:Icn} and to take the $\epsilon \downarrow 0$ limit, keeping the terms that do not vanish:
\begin{align}\label{eq:Icn_2}
    I_\mathrm{int}^{(n)}(p,\bQ^x,\bQ^z) &\simeq \frac{\beta}{2} \inf_{\bGamma^x,\bGamma^z} [\mathrm{Tr}[\bQ^x \bGamma^x] + \alpha \mathrm{Tr}[\bQ^z \bGamma^z] - \langle \ln \det (\bGamma^x + \lambda \bGamma^z) \rangle_\nu].
\end{align}
Finally, we again consider a replica-symmetric assumption for $\bGamma^x,\bGamma^z$, in the form:
\begin{align}
    \bGamma^x = \begin{pmatrix}
        \Gamma_x & -\gamma_x & \cdots & -\gamma_x \\
        -\gamma_x & \Gamma_x & \cdots & -\gamma_x \\
        \vdots & \vdots & \ddots & \vdots \\
        -\gamma_x & -\gamma_x & \cdots & \Gamma_x
    \end{pmatrix}, \hspace{1cm}
    \bGamma^z = \begin{pmatrix}
        \Gamma_z & -\gamma_z & \cdots & -\gamma_z \\
        -\gamma_z & \Gamma_z & \cdots & -\gamma_z \\
        \vdots & \vdots & \ddots & \vdots \\
        -\gamma_z & -\gamma_z & \cdots & \Gamma_z
    \end{pmatrix}.
\end{align}
As for the overlap matrices, we have $\gamma_x,\gamma_z \in \bbR$.
Combining eqs.~\eqref{eq:Icd} and \eqref{eq:Icn_2} and using the replica symmetric assumption, we obtain:
\begin{align}\label{eq:Ic_term}
    &\frac{2}{\beta}I_\mathrm{int}(p,\bQ_x,\bQ_z) = \inf_{\Gamma_x,\gamma_x,\Gamma_z,\gamma_z} [(p+1) Q_x \Gamma_x - p(p+1) q_x \gamma_x + \alpha(p+1) Q_z \Gamma_z - \alpha p(p+1) q_z \gamma_z \nonumber \\ 
    & - p  \langle \ln (\Gamma_x + \gamma_x + \lambda \Gamma_z + \lambda \gamma_z)  \rangle_\nu - \langle \ln [\Gamma_x - p \gamma_x + \lambda (\Gamma_z- p \gamma_z)] \rangle_\nu] - (\alpha+1)(p+1)\ln 2\pi e/\beta \nonumber \\
    & + (p+1) \ln \frac{2 \pi}{\beta} - p \ln (Q_x-q_x) - \ln (Q_x + pq_x) - \alpha p \ln (Q_z-q_z) - \alpha \ln (Q_z + pq_z).
\end{align}
\paragraph{A note on quenched and annealed averages}
Note that here we did not consider the average over $\bPhi$ to compute $I_\mathrm{int}$. 
Indeed, the result only depends on the eigenvalue distribution of $\bPhi^\dagger \bPhi / n$, which (by hypothesis)
has large deviations in a scale at least $n^{1+\eta}$ with $\eta > 0$. Since we are looking at a scale exponential in $n$, 
we can thus consider that this eigenvalue distribution is equal to its limit value $\nu$. However, one must be careful that this argument
breaks down if our result starts to be sensitive to the extremal eigenvalues of $\bPhi^\dagger \bPhi / n$. Since these variables typically have 
large deviations in the scale $n$ (for instance for Wigner or Wishart matrices \cite{dean2006large}), this could invalidate our calculation. 
This phenomenon is well-known in the study of so-called ``HCIZ'' spherical integrals, cf \cite{guionnet2005fourier} for an example of a rigorous analysis.
We argue in Section~\ref{subsec:quenched_annealed} that this possible issue, not discussed in \cite{takahashi2020macroscopic}, 
never arises for physical values of the overlaps.

\subsubsection{Expressing the \texorpdfstring{$p$}{p}-th moment}

Combining the results of the three previous sections,
we finally obtain the asymptotics of the $p$-th moment of the partition function as:
\begin{align}\label{eq:pth_moment}
    \lim_{n \to \infty} \frac{1}{n} \ln \EE \mathcal{Z}_n(\bY)^p &= \sup_{\substack{Q_x,q_x \\ Q_z,q_z}} [I_0(p,Q_x,q_x) + \alpha I_\mathrm{out}(p,Q_z,q_z) + I_\mathrm{int}(p,Q_x,q_x,Q_z,q_z)],
\end{align}
in which the three terms are given by eqs.~\eqref{eq:prior_term},\eqref{eq:channel_term},\eqref{eq:Ic_term}.

\subsection{The \texorpdfstring{$p \downarrow 0$}{p to 0} limit}

One can easily see that the function described in eq.~\eqref{eq:pth_moment} is analytic in $p$.
The next step of the replica method is to analytically extend this expression to arbitrary $p > 0$, before considering the limit $p \downarrow 0$. 

\subsubsection{Consistency of the limit}\label{subsubsec:consistency_p0}

One must be careful that, when extending our expression to arbitrarily small $p > 0$, 
we satisfy the trivial condition $ \lim_{p \downarrow 0}\ln \EE Z^p = 0$. 
As we will see, this condition will yield constraints on the diagonals of the overlap matrices.
Taking the limit $p = 0$ in the three terms of eq.~\eqref{eq:pth_moment} yields:
\begin{align}
    I_0(0,Q_x,q_x) &= \inf_{\hat{Q}_x} \Big\{\frac{\beta}{2} Q_x \hat{Q}_x + \ln \int_\bbK P_0(\mathrm{d}x) e^{-\frac{\beta \hat{Q}_x}{2} |x|^2}\Big\}, \\
    I_\mathrm{out}(0,Q_z,q_z) &= \inf_{\hat{Q}_z} \Big\{\frac{\beta}{2} Q_z \hat{Q}_z + \frac{\beta}{2} \ln \Big(\frac{2 \pi}{\beta \hat{Q}_z}\Big) \Big\},\\
    I_\mathrm{int}(0,Q_x,q_x,Q_z,q_z) &= \inf_{\Gamma_x,\Gamma_z} \Big[\frac{\beta}{2} Q_x \Gamma_x + \frac{\alpha \beta}{2} Q_z \Gamma_z - \frac{\beta}{2} \langle \ln [\Gamma_x + \lambda \Gamma_z] \rangle_\nu\Big] \\ 
    & - \frac{\beta(\alpha+1)}{2} (1+\ln \frac{2\pi}{\beta})  + \frac{\beta}{2} \ln \frac{2 \pi}{\beta}- \frac{\beta}{2} \ln Q_x - \frac{\alpha \beta}{2} \ln Q_z. \nonumber
\end{align}
One can easily solve the saddle point equations on $Q_z, \hat{Q_z}$, they give $\Gamma_z = 0$ and $\hat{Q}_z = 1/Q_z$.
One can then find all the remaining variables easily: $Q_x = \rho$, $\hat{Q}_x = 0$, $\Gamma_x = \rho^{-1}$, $Q_z = \rho \langle \lambda \rangle_\nu / \alpha$, $\hat{Q}_z = 1/Q_z$, $\Gamma_z = 0$.
Plugging these parameters yields (we drop the vacuous dependency on $q_x,q_z$):
\begin{subnumcases}{\label{eq:0th_moment}}
    I_0(0,Q_x=\rho) = 0, &\\ 
    I_\mathrm{out}\Big(0,Q_z = \frac{\rho \langle \lambda \rangle_\nu}{\alpha}\Big) = \frac{\beta}{2} + \frac{\beta}{2} \ln \Big(\frac{2 \pi \rho \langle \lambda \rangle_\nu}{\beta \alpha}\Big), &\\
    I_\mathrm{int}\Big(0,Q_x = \rho,Q_z = \frac{\rho \langle \lambda \rangle_\nu}{\alpha}\Big) =- \frac{\beta\alpha}{2} \Big(1+\ln \frac{2 \pi}{\beta}\Big)  - \frac{\alpha \beta}{2} \ln \frac{\rho\langle \lambda \rangle_\nu}{\alpha}.&
\end{subnumcases}
Recall that we have 
\begin{align*}
   \lim_{p \downarrow 0} \lim_{n \to \infty} \frac{1}{n} \ln \EE \mathcal{Z}_n(\bY)^p &= I_0 + \alpha I_\mathrm{out} + I_\mathrm{int},
\end{align*}
so that we obtain from eq.~\eqref{eq:0th_moment} that indeed the limit is consistent.

\subsubsection{The replica symmetric result}\label{subsubsec:rs_result}

Using eq.~\eqref{eq:pth_moment} for the $p$-th moment and the consistency conditions we just derived, 
we obtain after using the replica trick:
\begin{align}\label{eq:final_replicas}
    \lim_{n \to \infty} \frac{1}{n} \EE \ln \mathcal{Z}_n(\bY) &= \sup_{q_x,q_z} [I_0(q_x) + \alpha I_\mathrm{out}(q_z) + I_\mathrm{int}(q_x,q_z)],
\end{align}
with the auxiliary functions:
\begin{align*}
    &I_0(q_x) = \inf_{\hat{q}_x \geq 0} \Big[-\frac{\beta\hat{q}_x q_x}{2} + \int_\bbK \mathcal{D}_\beta\xi P_0(\mathrm{d}x) e^{-\frac{\beta\hat{q}_x}{2} |x|^2 + \beta \sqrt{\hat{q}_x} x \cdot \xi} \ln \int_\bbK P_0(\mathrm{d}x) e^{-\frac{\beta\hat{q}_x}{2} |x|^2 + \beta \sqrt{\hat{q}_x} x \cdot \xi}\Big], \\
    &I_\mathrm{out}(q_z) = \inf_{\hat{q}_z \geq 0} \Big\{-\frac{\beta \hat{q}_z q_z}{2} - \frac{\beta}{2} \ln (\hat{Q}_z + \hat{q}_z) + \frac{\beta\hat{q}_z}{2 \hat{Q}_z} + \int \mathrm{d}y \mathcal{D}_\beta \xi \ J(\hat{q}_z,y,\xi) \ln J(\hat{q}_z,y,\xi) \Big\}, \\
    &I_\mathrm{int}(q_x,q_z )= \inf_{\gamma_x,\gamma_z \geq 0} \Big[\frac{\beta}{2} (\rho - q_x) \gamma_x + \frac{\alpha\beta}{2} (Q_z - q_z) \gamma_z - \frac{\beta}{2} \langle \ln (\rho^{-1} + \gamma_x + \lambda \gamma_z)  \rangle_\nu \Big] \\
    & \hspace{2cm} - \frac{\beta}{2} \ln(\rho-q_x) - \frac{\beta q_x}{2 \rho} - \frac{\alpha\beta}{2} \ln(Q_z - q_z) - \frac{\alpha \beta q_z}{2 Q_z}, \nonumber
\end{align*}
with $Q_z = \rho \langle \lambda \rangle_\nu / \alpha$ and $\hat{Q}_z = 1/Q_z$.
Moreover, the domain of the supremum is $q_x \in [0,\rho]$ and $q_z \in [0,Q_z]$. The function $J(\hat{q}_z,y,\xi)$ appearing in the expression of $I_\mathrm{out}$ is defined as:
\begin{align*}
    J(\hat{q}_z,y,\xi) &\equiv \int_\bbK \mathcal{D}_\beta z P_\mathrm{out}\Big(y\Big|\frac{z}{\sqrt{\hat{Q}_z + \hat{q}_z}} + \sqrt{\frac{\hat{q}_z}{\hat{Q}_z(\hat{Q}_z + \hat{q}_z)}}\xi\Big). 
\end{align*}
Note that compared to the calculation presented in the previous sections, we moved a term $(\beta \alpha / 2) (1+  \ln 2\pi/\beta)$ between $I_\mathrm{out}$ and $I_\mathrm{int}$, 
and we also made a few straightforward change of variables in the expression of $I_\mathrm{out}$.
This is exactly the result given in Conjecture~\ref{conjecture:replicas_general}, which ends our replica calculation.

\subsection{Concentration of the spectrum of \texorpdfstring{$\bPhi^\dagger \bPhi / n$}{FF/n} and the absence of saturation}\label{subsec:quenched_annealed}

As emphasized in the end of Section~\ref{subsubsec:Ic_computation}, 
our calculation assumed that the extremization equations on $(\gamma_x,\gamma_z)$ always admitted a solution. Moreover, 
we assumed that this solution is not sensitive to the extremal eigenvalues of $\bPhi^\dagger \bPhi / n$. If this assumption is indeed true, 
the concentration of the spectrum of $\bPhi^\dagger \bPhi / n$ was assumed to be fast enough to justify our calculation.
This important condition can be phrased by saying that for all 
physical values of $(q_x,q_z)$, we must not touch the edge of the spectrum:
\begin{align}\label{eq:condition_no_saturation}
    \frac{1}{\rho} + \gamma_x + \gamma_z \lambda_{\mathrm{min}}(\nu) > 0.
\end{align}
We justify here eq.~\eqref{eq:condition_no_saturation} for all physical values of $(q_x,q_z)$.
We will combine three arguments:
\begin{itemize}
    \item[$(i)$] Note that in the replica calculation, cf Section~\ref{subsubsec:Ic_computation}, 
        the matrix $\bGamma^z$ is assumed to be Hermitian positive in the $p \downarrow 0$ limit. Since $\Gamma_z = 0$, this implies that we must have 
        $\lambda_z \geq 0$.
    \item[$(ii)$] The saddle point equation on $q_x$ yields\footnote{This relation is valid even if $\lambda_x$ would ``saturate'' to a constant value that does not depend on $(q_x,q_z)$.}:
        \begin{align}\label{eq:fixed_point_qx}
            \hat{q}_x = \frac{q_x}{\rho(\rho-q_x)} - \gamma_x.
        \end{align}
    \item[$(iii)$] Finally, we will derive a lower bound on $q_x$. Note that, as one can see in $I_0$ from Section~\ref{subsubsec:rs_result}, $q_x$ is the 
    optimal overlap achievable in the following scalar inference problem \cite{barbier2019optimal}:
    \begin{align}
        Y_0 = \sqrt{\hat{q}_x} X^\star + Z,
    \end{align}
    in which one observes $Y_0$ and is given $P_0$ the prior distribution on $X^\star$, and the noise $Z$ is distributed according to  $\mathcal{N}_\beta(0,1)$.
    It is known that the optimal estimator is given by the average of $\EE[x|Y]$ under the posterior distribution, whose density is proportional to $P_0(x) e^{-\frac{\beta}{2}|y-\sqrt{\hat{q}_x}x|^2}$.
    If this is untractable for generic $P_0$, we can consider a suboptimal estimation by using a Gaussian prior with variance $\rho$ in the estimation procedure (so that the problem is mismatched). This yields the bound:
    \begin{align}
        q_x &\geq \int{\cal D}_\beta \xi\frac{\Big[\int_\bbK P_0(\mathrm{d}x) \ x \ e^{-\frac{\beta\hat{q}_x}{2} |x|^2 + \beta \sqrt{\hat{q}_x} x \cdot \xi}\Big] \cdot \Big[\int_\bbK \mathrm{d}x \ x \ e^{-\frac{\beta|x|^2}{2 \rho}} \ e^{-\frac{\beta \hat{q}_x}{2} |x|^2 + \beta \sqrt{\hat{q}_x} x \cdot \xi}\Big]}{
        \int_\bbK \mathrm{d}x \ e^{-\frac{\beta|x|^2}{2 \rho}} \ e^{-\frac{\beta \hat{q}_x}{2} |x|^2 + \beta \sqrt{\hat{q}_x} x \cdot \xi}}. 
    \end{align}
    This can easily be simplified by performing the Gaussian integral, and yields the bound:
    \begin{align}\label{eq:bound_qx}
        q_x \geq \frac{\rho^2 \hat{q}_x}{1+\rho\hat{q}_x}.
    \end{align}
\end{itemize}
Combining $(ii)$ and $(iii)$ gives:
\begin{align}
    q_x \geq \rho - \frac{\rho-q_x}{1 - \gamma_x (\rho-q_x)}.
\end{align}
Since $q_x \in [0,\rho]$, this implies in particular that $\gamma_x \geq 0$.
Using this along with $(i)$, this implies:
\begin{align}
    \frac{1}{\rho} + \gamma_x + \gamma_z \lambda_\mathrm{min}(\nu) \geq \frac{1}{\rho} > 0,
\end{align}
which is what we wanted to show.
\section{Derivation of the weak-recovery threshold}
\label{sec:app:weakrecovery}
We detail here the derivation of the algorithmic weak-recovery threshold $\alpha_{\WR,\alg}$. As discussed in Section~\ref{sec:weak_recovery}, the weak-recovery threshold can be identified as the sample complexity for which the trivial fixed point $q_x = q_z = \hat{q}_x = \hat{q}_z = \gamma_x = \gamma_z = 0$ of the state evolution equations becomes linearly unstable (when it no longer is a local maximum of the free entropy potential).
Consider therefore the state evolution equations, which we repeat here for convenience in a detailed form:
\begin{subnumcases}{\label{eq:app:se_general}}
    q_x = \int_\bbK \mathcal{D}_\beta \xi \frac{\big|\int_\bbK P_0(\mathrm{d}x) \ x \  e^{- \frac{\beta}{2} \hat{q}_x |x|^2 +\beta \sqrt{\hat{q}_x} x \cdot \xi}\big|^2}{\int_\bbK P_0(\mathrm{d}x) \ e^{- \frac{\beta}{2} \hat{q}_x |x|^2 +\beta \sqrt{\hat{q}_x} x \cdot \xi}},&\\
    q_z = \frac{1}{\hat{Q}_z + \hat{q}_z} \Big[\frac{\hat{q}_z}{\hat{Q}_z} + \int \mathrm{d}y \ \mathcal{D}_\beta \xi \frac{\Big|\int \mathcal{D}_\beta z \ z \  P_\mathrm{out}\Big(y\Big|\frac{z}{\sqrt{\hat{Q}_z + \hat{q}_z}} + \sqrt{\frac{\hat{q}_z}{\hat{Q}_z(\hat{Q}_z + \hat{q}_z)}}\xi\Big)\Big|^2}{\int \mathcal{D}_\beta z P_\mathrm{out}\Big(y\Big|\frac{z}{\sqrt{\hat{Q}_z + \hat{q}_z}} + \sqrt{\frac{\hat{q}_z}{\hat{Q}_z(\hat{Q}_z + \hat{q}_z)}}\xi\Big)}\Big],& \\
    \hat{q}_x = \frac{q_x}{\rho(\rho-q_x)} - \gamma_x ,& \\
    \hat{q}_z = \frac{q_z}{Q_z(Q_z-q_z)} - \gamma_z, &\\
    \rho-q_x = \Big\langle \frac{1}{\rho^{-1} + \gamma_x + \lambda \gamma_z}\Big\rangle_\nu,& \\
    \alpha(Q_z-q_z) = \Big\langle \frac{\lambda}{\rho^{-1} + \gamma_x + \lambda \gamma_z}\Big\rangle_\nu.&
\end{subnumcases}
Letting $q_x = q_z = \hat{q}_x = \hat{q}_z = \gamma_x = \gamma_z = 0$, it is clear that the equations are satisfied if the signal distribution $P_{0}$ and the likelihood $P_{\out}$ satisfy the following symmetry conditions:
\begin{align*}
    |x_1| = |x_2| \Rightarrow P_0(x_1) = P_0(x_2) \hspace{1cm} \mathrm{and} \hspace{0.5cm} |z_1| = |z_2| \Rightarrow P_\mathrm{out}(y|z_1) = P_\mathrm{out}(y|z_2).
\end{align*}
Assuming these conditions hold, we are interested in studying the linear stability of this local maximum. Recalling that $Q_z = \rho \langle \lambda \rangle_\nu / \alpha$,
the first, third and fourth equations of eq.~\eqref{eq:app:se_general} can be linearized:
\begin{align}{\label{eq:app:se_wr_1}}
   \delta q_x = \rho^2 \delta \hat{q}_x, &&
   \delta \hat{q_x} = \frac{\delta q_x}{\rho^2} - \delta \gamma_x, && 
   \delta \hat{q_z} = \frac{\alpha^2 \delta q_z }{\rho^2 \langle \lambda \rangle_\nu^2} - \delta \gamma_z.
\end{align}
Now focusing on the second state evolution equation~\eqref{eq:app:se_general}, it can be linearized to give:
\begin{align}\label{eq:app:se_wr_2}
    \delta q_z &= \frac{\rho^2 \langle \lambda \rangle_\nu^2}{\alpha^2} \delta \hat{q}_z \Big(1 + \int_\bbR \mathrm{d}y \frac{\Big|\int_\bbK \mathcal{D}_\beta z \ (|z|^2-1) \  P_\mathrm{out}\big(y\big|\sqrt{\frac{\rho \langle \lambda \rangle_\nu}{\alpha}} z\big)\Big|^2}{\int_\bbK \mathcal{D}_\beta z \ P_\mathrm{out}\big(y\big|\sqrt{\frac{\rho \langle \lambda \rangle_\nu}{\alpha}} z\big)}\Big).
\end{align}
Finally, it remains to compute the infinitesimal variation for $\delta \gamma_x, \delta \gamma_z$:
\begin{subnumcases}{\label{eq:app:se_wr_3}}
   \delta \gamma_x = \frac{\langle \lambda^2 \rangle_\nu}{\rho^2[\langle \lambda^2 \rangle_\nu - \langle \lambda \rangle_\nu^2]} \delta q_x - \frac{\alpha \langle \lambda \rangle_\nu}{\rho^2[\langle \lambda^2 \rangle_\nu - \langle \lambda \rangle_\nu^2]} \delta q_z,& \\ 
   \delta \gamma_z = -\frac{\langle \lambda \rangle_\nu}{\rho^2[\langle \lambda^2 \rangle_\nu - \langle \lambda \rangle_\nu^2]} \delta q_x + \frac{\alpha}{\rho^2[\langle \lambda^2 \rangle_\nu - \langle \lambda \rangle_\nu^2]} \delta q_z.&
\end{subnumcases}
Combining eqs.~\eqref{eq:app:se_wr_1},\eqref{eq:app:se_wr_2},\eqref{eq:app:se_wr_3}, we can simplify the system to a closed set equations with only $(\delta q_x,\delta \hat{q}_x,\delta q_z,\delta \hat{q}_z)$. Given the usual heuristics of the replica method and its link with the state evolution equations of message-passing algorithms \cite{takahashi2020macroscopic,zdeborova2016statistical,krzakala2012probabilistic}, 
one can conjecture that the simplest iteration scheme corresponds to the state evolution of the G-VAMP message passing algorithm:
\begin{subnumcases}{\label{eq:wr_linearized_final}}
   \delta q_x^{t+1} = \rho^2 \delta \hat{q}_x^t ,& \\ 
   \delta q_z^{t+1} = \frac{\rho^2 \langle \lambda \rangle_\nu^2}{\alpha^2} \delta \hat{q}_z^t \Big(1 + \int_\bbR \mathrm{d}y \frac{\Big|\int_\bbK \mathcal{D}_\beta z \ (|z|^2-1) \  P_\mathrm{out}\big(y\big|\sqrt{\frac{\rho \langle \lambda \rangle_\nu}{\alpha}} z\big)\Big|^2}{\int_\bbK \mathcal{D}_\beta z \ P_\mathrm{out}\big(y\big|\sqrt{\frac{\rho \langle \lambda \rangle_\nu}{\alpha}} z\big)}\Big),& \\
   \delta \hat{q_x}^t = -\frac{\langle \lambda \rangle_\nu^2}{\rho^2[\langle \lambda^2 \rangle_\nu - \langle \lambda \rangle_\nu^2]} \delta q_x^t + \frac{\alpha \langle \lambda \rangle_\nu}{\rho^2[\langle \lambda^2 \rangle_\nu - \langle \lambda \rangle_\nu^2]} \delta q_z^t,& \\ 
   \delta \hat{q_z}^t =  \frac{\langle\lambda\rangle_\nu}{\rho^2[\langle \lambda^2 \rangle_\nu - \langle \lambda \rangle_\nu^2]} \delta q_x^t + [\frac{\alpha^2}{\rho^2 \langle \lambda \rangle_\nu^2} - \frac{\alpha}{\rho^2[\langle \lambda^2 \rangle_\nu - \langle \lambda \rangle_\nu^2]}] \delta q_z^t.&
\end{subnumcases}
From these equations, one can easily see that a linear instability of the trivial fixed points appears at $\alpha = \alpha_{\WR,\alg}$ satisfying the equation:
\begin{align}\label{eq:app:wr_threshold_general}
    \alpha_{\WR,\alg} = \frac{\langle \lambda \rangle_\nu^2}{\langle \lambda^2 \rangle_\nu}\Big(1+\Big[\int_\bbR \mathrm{d}y \frac{\Big|\int_\bbK \mathcal{D}_\beta z \ (|z|^2-1) \  P_\mathrm{out}\big(y\big|\sqrt{\frac{\rho \langle \lambda \rangle_\nu}{\alpha_{\WR,\alg}}} z\big)\Big|^2}{\int_\bbK \mathcal{D}_\beta z \ P_\mathrm{out}\big(y\big|\sqrt{\frac{\rho \langle \lambda \rangle_\nu}{\alpha_{\WR,\alg}}} z\big)}\Big]^{-1}\Big).
\end{align}
Indeed at $\alpha = \alpha_{\WR,\alg}$, the modulus of all the eigenvalues of the size-$4$ matrix of the linear system~\eqref{eq:wr_linearized_final} cross $1$.
\section{The full recovery transition}\label{sec:app_perfect_recovery}

In this section, we assume a Gaussian standard prior $P_0 = \mathcal{N}_\beta(0,1)$ and a noiseless phase retrieval channel,
and we show that information-theoretic full recovery is achieved exactly at $\alpha = \alpha_{\FR,\IT} \equiv \beta(1-\nu(\{0\}))$.
We can assume without loss of generality that $\langle \lambda \rangle_\nu = \alpha$, as this amounts to a simple rescaling of $\bPhi$, irrelevant under the noiseless channel.
This implies in particular that $Q_z = \hat{Q}_z = 1$.

\subsection{The state evolution equations}

Since we assumed a Gaussian prior, we have, with $P_\mathrm{out}(y|z) = \delta(y-|z|^2)$:
\begin{subnumcases}{\label{eq:se_pr_gaussian_prior}}
\label{eq:se_pr_gaussian_prior_qz}
    q_z = \frac{1}{1 + \hat{q}_z} \Big[\hat{q}_z + \int \mathrm{d}y\int_\bbK \mathcal{D}_\beta \xi \frac{\Big|\int_\bbK \mathcal{D}_\beta z \ z \  P_\mathrm{out}\Big(y\Big|\frac{z}{\sqrt{1 + \hat{q}_z}} + \sqrt{\frac{\hat{q}_z}{1 + \hat{q}_z}}\xi\Big)\Big|^2}{\int_\bbK \mathcal{D}_\beta z P_\mathrm{out}\Big(y\Big|\frac{z}{\sqrt{1 + \hat{q}_z}} + \sqrt{\frac{\hat{q}_z}{1 + \hat{q}_z}}\xi\Big)}\Big],& \\
\label{eq:se_pr_gaussian_prior_qhx}
    \hat{q}_x = \frac{q_x}{1-q_x},& \\
\label{eq:se_pr_gaussian_prior_qhz}
    \hat{q}_z = \frac{q_z}{1-q_z} - \gamma_z, &\\
\label{eq:se_pr_gaussian_prior_qx}
    q_x = \alpha \gamma_z (1-q_z), & \\
\label{eq:se_pr_gaussian_prior_gammaz}
    \alpha(1-q_z) = \Big\langle \frac{\lambda}{1 + \lambda \gamma_z}\Big\rangle_\nu.&
\end{subnumcases}
Comparing these equations to Conjecture~\ref{conjecture:replicas_general}, 
one can see that we imposed $\gamma_x = 0$, a straightforward consequence of the Gaussian prior (see Section~\ref{sec:app_equivalence_replicas} where this calculation is detailed for a different purpose).

\subsection{Noisy phase retrieval with small variance}

We wish to show that the free entropy of the full recovery solution is the global maximum of the free entropy potential 
for $\alpha > \alpha_\mathrm{IT}$, while it is never the case for $\alpha < \alpha_\mathrm{IT}$.
However, under a noiseless channel, the free entropy potential might diverge in this point, which 
indicates towards a regularization procedure.
Therefore we consider a noisy Gaussian channel with noise $\Delta > 0$:
\begin{align}
    P_\mathrm{out}(y|z) = \frac{1}{\sqrt{2 \pi \Delta}} \exp \Big\{-\frac{1}{2\Delta}(y-|z|^2)^2\Big\}.
\end{align}
We will compute the limit, as $\Delta \downarrow 0$, of the free entropy of the ``almost perfect'' recovery fixed point.
We look for a solution close to the point which corresponds to the best possible recovery:
\begin{subnumcases}{}
    q_x = 1 - \nu(\{0\}), &\\
    q_z = 1.&
\end{subnumcases}
Indeed it is easy to see that $q_x \leq 1 -\nu(\{0\})$ since $\mathrm{rk}[\bPhi^\dagger \bPhi] \sim n (1-\nu(\{0\}))$.
We are thus looking for a fixed point of the state evolution equations~\eqref{eq:se_pr_gaussian_prior}
 that satisfies:
 \begin{subnumcases}{}
    q_x = 1 - \nu(\{0\}) + \smallO_\Delta(1), & \\
    q_z = 1 + \smallO_\Delta(1), & \\ 
    \hat{q}_x^{-1} = \nu(\{0\}) / (1-\nu(\{0\})) + \smallO_\Delta(1), &\\
    \hat{q}_z^{-1} = \smallO_\Delta(1).&
 \end{subnumcases}
 Let us now precise the asymptotics of these quantities as $\Delta \downarrow 0$.
By eq.~\eqref{eq:se_pr_gaussian_prior_qx}, we find easily:
\begin{align}
    \gamma_z \sim \frac{1-\nu(\{0\})}{\alpha(1-q_z)}.
\end{align}
Then from eq.~\eqref{eq:se_pr_gaussian_prior_qhz}, we also have:
\begin{align}\label{eq:qhz_scaling}
    \hat{q}_z \sim \frac{\alpha-1+\nu(\{0\})}{\alpha(1-q_z)}.
\end{align}
Note that if $\alpha \leq 1$, then necessarily $\nu(\{0\}) \geq 1-\alpha$, so that the quantity in the numerator is always positive.
We now turn to eq.~\eqref{eq:se_pr_gaussian_prior_qz}.
We assume the scaling $\hat{q}_z^{-1} = c \Delta + \smallO_\Delta(\Delta)$.
We have by Gaussian integration by parts and using the specific form of $P_\mathrm{out}$:
\begin{align*}
     \int \mathrm{d}y \mathcal{D}_\beta \xi & \frac{\Big|\int \mathcal{D}_\beta z \ z \  P_\mathrm{out}\Big(y\Big|\frac{z}{\sqrt{1 + \hat{q}_z}} + \sqrt{\frac{\hat{q}_z}{1 + \hat{q}_z}}\xi\Big)\Big|^2}{\int \mathcal{D}_\beta z P_\mathrm{out}\Big(y\Big|\frac{z}{\sqrt{1 + \hat{q}_z}} + \sqrt{\frac{\hat{q}_z}{1 + \hat{q}_z}}\xi\Big)} \\
     &= \frac{1}{(1+\hat{q}_z)} \int \mathrm{d}y \mathcal{D}_\beta \xi \frac{\Big|\int \mathcal{D}_\beta z \  P_\mathrm{out}'\Big(y\Big|\frac{z}{\sqrt{1 + \hat{q}_z}} + \sqrt{\frac{\hat{q}_z}{1 + \hat{q}_z}}\xi\Big)\Big|^2}{\int \mathcal{D}_\beta z P_\mathrm{out}\Big(y\Big|\frac{z}{\sqrt{1 + \hat{q}_z}} + \sqrt{\frac{\hat{q}_z}{1 + \hat{q}_z}}\xi\Big)} \sim \frac{4}{\Delta(1+\hat{q}_z)} \sim 4c.
\end{align*}
Gaussian integration by parts and our conventions for derivatives of real functions of complex variables are summarized in Section~\ref{subsec:app_derivatives}.
This yields that $1-q_z = \Delta c(1-4c) + \smallO_\Delta(1)$.
Combining this result with eq.~\eqref{eq:qhz_scaling}, we have
\begin{align*}
    c (1-4c)&= c \Big[\frac{\alpha-1+\nu(\{0\})}{\alpha}\Big].
\end{align*}
This implies $c = (1-\nu(\{0\}))/(4 \alpha)$, and we finally obtain the leading order asymptotics of $q_z,\hat{q_z},\gamma_z$ as $\Delta \to 0$:
\begin{subnumcases}{\label{eq:final_asymptotics}}
   \hat{q}_z = \frac{4 \alpha}{(1-\nu(\{0\}))\Delta} + \smallO_\Delta\(\Delta^{-1}\), & \\ 
   1- q_z = \frac{(1-\nu(\{0\})(\alpha -1 + \nu(\{0\}))}{4 \alpha ^2} \Delta + \smallO_\Delta(\Delta), & \\
   \gamma_z = \frac{4 \alpha}{\Delta(\alpha-1+\nu(\{0\}))}+ \smallO_\Delta(\Delta^{-1}) . &
\end{subnumcases}
Let us now compute the asymptotics of the three auxiliary functions $I_0,I_\mathrm{out}$ and $I_\mathrm{int}$ of Conjecture~\ref{conjecture:replicas_general}:
\begin{align*}
    I_0(q_x) &= \frac{\beta}{2} [q_x + \ln(1-q_x)], \\
    I_\mathrm{out}(q_z) &= -\frac{\beta\hat{q}_z q_z}{2} - \frac{\beta}{2} \ln (1 + \hat{q}_z) + \frac{\beta \hat{q}_z}{2} + \int \mathrm{d}y \mathcal{D}\xi \ J(\hat{q}_z,y,\xi) \ln J(\hat{q}_z,y,\xi), \\
    J(\hat{q}_z,y,\xi) &\equiv \int \mathcal{D}z P_\mathrm{out}\Big(y\Big|\frac{z}{\sqrt{1 + \hat{q}_z}} + \sqrt{\frac{\hat{q}_z}{1 + \hat{q}_z}}\xi\Big), \nonumber \\
    I_\mathrm{int}(q_x,q_z) &= \frac{\beta}{2}[\alpha (1 - q_z) \gamma_z - \langle \ln (1 + \lambda \gamma_z) \rangle_\nu - \ln(1-q_x) - q_x - \alpha \ln(1 - q_z) - \alpha q_z ].
\end{align*}
Using eq.~\eqref{eq:final_asymptotics} and the specific form of the channel, we reach:
\begin{align*}
    I_0(q_x) + I_\mathrm{int}(q_x,q_z) &\sim - \frac{\beta(\alpha-1+\nu(\{0\}))}{2} \ln \Delta, \\
    I_\mathrm{out}(q_z) &\sim \frac{(\beta-1)}{2} \ln \Delta.
\end{align*}
Therefore when considering the total free entropy we have
\begin{align*}
    I_0(q_x) + I_\mathrm{int}(q_x,q_z) + \alpha I_\mathrm{out}(q_z) &\sim \frac{\alpha(\beta-1) - \beta(\alpha-1+\nu(\{0\}))}{2} \ln \Delta , \\
    & \sim \frac{\beta(1-\nu(\{0\}))-\alpha}{2} \ln \Delta.
\end{align*}
This implies that the full recovery point has a free entropy of $- \infty$ for $\alpha < \alpha_{\FR,\IT} \equiv \beta(1-\nu(\{0\}))$, and $+ \infty$ for $\alpha > \alpha_{\FR,\IT}$. 
Thus this point is always the global maximum of the free entropy for $\alpha > \alpha_{\FR,\IT}$, while it 
is never the case for $\alpha < \alpha_{\FR,\IT}$, which ends our argument.
\section{Proof of Theorem~\ref{thm:main_product_gaussian_rotinv}}\label{sec:app_proof_theorem}

In all this section, we provide the proof of Theorem~\ref{thm:main_product_gaussian_rotinv} under \ref{hyp:channel},\ref{hyp:prior},\ref{hyp:phi_product},\ref{hyp:convergence_spectrum_B}, and we will work under these hypotheses. In Section~\ref{subsec:app_gaussian_matrix_proof}, we show how the proof can be extended to hypotheses~\ref{hyp:channel},\ref{hyp:any_p0_gaussian_phi}.
\\
First, we simplify the conjectured expression of the free entropy of Conjecture~\ref{conjecture:replicas_general} using the particular form of the prior $P_0$ and of the sensing matrix $\bPhi$.
Finally, using \ref{hyp:prior},\ref{hyp:phi_product},\ref{hyp:convergence_spectrum_B} and a proof similar to the one of \cite{barbier2019optimal,aubin2018committee}, we give a rigorous 
derivation of this simplified expression.
Note that with respect to the analysis of \cite{barbier2019optimal,aubin2018committee}, there are two main novelties in our setting:
\begin{itemize}
  \item[$(i)$] The sensing matrix $\bPhi$ is not i.i.d.\ but has a well-controlled structure, see \ref{hyp:phi_product}.
  \item[$(ii)$] The variables can be complex numbers. We will argue that the arguments generalize to this case. The physical reason of this generalization is that even in the complex setting, the overlap will concentrate 
  on a real positive number, as a consequence of Bayes-optimality.
\end{itemize}
First, we note that we can simplify the replica conjecture under the considered hypotheses:
\begin{proposition}\label{prop:replicas_product_structure}
  Under \ref{hyp:channel},\ref{hyp:prior},\ref{hyp:phi_product},\ref{hyp:convergence_spectrum_B}, the replica conjecture~\ref{conjecture:replicas_general} for the free entropy $f_n \equiv \frac{1}{n} \EE \ln \mathcal{Z}_n(\bY)$ is equivalent to:
  \begin{align}\label{eq:Phi_noniid}
    \lim_{n \to \infty} f_n &= \sup_{\hat{q}\geq 0} \inf_{q \in [0,Q_z]}\Big[\frac{\beta\hat{q}}{2} (\EE_{\nu_B}[X] - \delta q) - \frac{\beta}{2} \EE_{\nu_B} \ln (1+\hat{q}X) + \alpha \Psi_{\mathrm{out}}(q)\Big],
  \end{align}
  with $Q_z = \EE_{\nu_B}[X]/\delta$ and $\Psi_{\mathrm{out}}$ defined in terms of the auxiliary functions introduced in eq.~\eqref{eq:auxiliary}:
  \begin{align*}
    \Psi_{\mathrm{out}}(q) \equiv \mathbb{E}_{\xi}\int_\bbR \mathrm{d} y ~ \mathcal{Z}_{\out}(y;\sqrt{q}\xi, Q_z - q) \ln \mathcal{Z}_{\out}(y;\sqrt{q}\xi, Q_z - q).
  \end{align*}
\end{proposition}
\noindent
Proposition~\ref{prop:replicas_product_structure} is proven in Section~\ref{sec:app_equivalence_replicas}.
To prove the free entropy statement of Theorem~\ref{thm:main_product_gaussian_rotinv}, we therefore just need to show:
\begin{lemma}\label{lemma:noniid_phi}
  Under the assumptions of Proposition~\ref{prop:replicas_product_structure}, the limit of the free entropy $f_n \equiv \frac{1}{n} \EE \ln \mathcal{Z}_n(\bY)$ is given by eq.~\eqref{eq:Phi_noniid}.
\end{lemma}
\noindent 
The following of this section  is dedicated to the proof of Lemma~\ref{lemma:noniid_phi}. 
We will conclude the proof of Theorem~\ref{thm:main_product_gaussian_rotinv} in Section~\ref{subsec:app_proof_mmse} and Section~\ref{subsec:app_gaussian_matrix_proof}, 
dedicated respectively to the proof of the MMSE statement and the extension of the proof to hypotheses \ref{hyp:channel},\ref{hyp:any_p0_gaussian_phi}.
\\
The main idea of our proof is to reduce the problem of Lemma~\ref{lemma:noniid_phi}
to a Generalized Linear Model with a Gaussian sensing matrix, 
but a non-i.i.d.\ prior.
We make use of the ``SVD'' decomposition of $\bB/\sqrt{n} = \bU \bS \bV^\dagger$, with $\bU \in \mathcal{U}_\beta(p)$, $\bV \in \mathcal{U}_\beta(n)$, and 
$\bS \in \bbR^{p \times n}$ a pseudo-diagonal matrix with positive elements.
Leveraging on the fact that the prior $P_0$ is Gaussian, and that $\bW$ is an i.i.d.\ Gaussian matrix independent of $\bB$, 
one can see that our estimation problem is formally equivalent to an usual Generalized Linear Model with $m$ measurements, a signal of dimension $p$, and a Gaussian i.i.d.\ sensing matrix.
This is very close to the setup of \cite{barbier2019optimal}, a key difference being that here
the prior distribution on the data $\bZ^\star \in \bbK^{p}$ is defined as
\begin{itemize}
   \item If $\delta \leq 1$, for every $k \in \{1,\cdots,p\}$, $Z^\star_k$ is distributed as $S_k X^\star_k$ with $X^\star_k \overset{\mathrm{i.i.d.}}{\sim} P_0$.
   \item If $\delta \geq 1$, for every $k \in \{1,\cdots,n\}$, $Z^\star_k$ is distributed as $S_k X^\star_k$ with $X^\star_k \overset{\mathrm{i.i.d.}}{\sim} P_0$, while for every 
   $k \in \{n+1, \cdots,p\}$, $Z^\star_k$ is almost surely $0$.
\end{itemize}
More precisely, we can define rigorously the prior $P_0^{(\bS)}$ described above by its linear statistics. For any continuous bounded function $g : \bbK^p \to \bbR$, 
one has:
\begin{align}\label{eq:def_P0S}
    \int_{\bbK^p} P_0^{(\bS)}(\mathrm{d}\bz) g(\bz) &\equiv \int_{\bbK^n} \Big\{\prod_{i=1}^n P_0(\mathrm{d} x_i) \Big\} g(\{\mathbbm{1}[k \leq n] S_k x_k\}_{k=1}^p).
\end{align}
Hypothesis~\ref{hyp:prior} implies that we will consider $P_0 = \mathcal{N}_\beta(0,1)$.
In the following of the section, we give the detailed sketch of the proof of Lemma~\ref{lemma:noniid_phi}. 
Some facts and lemmas will be a generalization or a consequence of the works of \cite{barbier2019optimal} and \cite{aubin2018committee}, and we will refer to them when necessary.

\subsection{Interpolating estimation problem}\label{subsec:app_interpolating_model}

Recall that $Q_z \equiv \rho \langle \lambda \rangle_\nu / \alpha = \EE_{\nu_B}[X]/\delta$, and the definition of $\Psi_\mathrm{out}$ in Proposition~\ref{prop:replicas_product_structure}.
We define as well:
\begin{align}
    r_\mathrm{max} &\equiv \sup_{q \in [0,Q_z]} \Psi_\mathrm{out}(q), \\
    \Psi_0^{(\nu)}(r) &\equiv \frac{\beta}{2} \big[r \EE_{\nu_B}[X] - \EE_{\nu_B} \ln (1+r X) \big] , \hspace{1cm} 0 \leq r \leq r_\mathrm{max}.
\end{align}
Since $\nu_B \neq \delta_0$ by hypothesis, we can easily check that $\Psi_{0}^{(\nu)}$ is strictly convex, $\mathcal{C}^2$ and non-decreasing on $[0,r_\mathrm{max}]$ .
By Proposition~18 of \cite{barbier2019optimal}, which directly generalizes to the complex case, we know as well that 
$\Psi_\mathrm{out}$ is convex, $\mathcal{C}^2$, 
and non-decreasing on $[0,Q_z]$, and thus $r_\mathrm{max} = \Psi_\mathrm{out}(Q_z)$.
Let us fix an arbitrary sequence $s_n > 0$ that goes to $0$ as $n$ goes to infinity. 
We fix $\epsilon_2 \in [s_n,2s_n]$, and $\epsilon_1 \in \mathcal{D}_n^\beta$, with 
\begin{align*}
    \mathcal{D}_n^\beta &\equiv \{\lambda \in \mathcal{S}_\beta(\bbR) \, : \, \forall l \in \{1,\beta\}, \lambda_{ll} \in (2 \beta s_n,(2\beta+1)s_n), \, \forall l\neq l' \in \{1,\beta\}, \lambda_{ll'} \in (s_n,2s_n)\}.
\end{align*}
$\mathcal{D}_n^\beta$ is composed of strictly diagonally dominant matrices with positive entries, which implies that $\mathcal{D}_n \subset \mathcal{S}_\beta^+(\bbR)$.
Let $q_\epsilon : [0,1] \to [0,Q_z]$, $r_\epsilon : [0,1] \to [0,r_\mathrm{max}]$ be two continuous ``interpolation'' functions.
For all $\epsilon \in \mathcal{D}_n^\beta \times [s_n,2s_n]$, and all $t \in [0,1]$ we define:
\begin{align}\label{eq:def_R1R2}
   \mathcal{S}_\beta^+(\bbR) \ni R_1(t,\epsilon) &\equiv \epsilon_1 + \Big(\int_0^t r_\epsilon(v) \mathrm{d}v\Big) \mathbbm{1}_\beta , \hspace{1cm} \bbR_+ \ni R_2(t,\epsilon) \equiv \epsilon_2 + \int_0^t q_\epsilon(v) \mathrm{d}v.
\end{align}
We consider the following decoupled observation channels:
\begin{subnumcases}{\label{eq:auxiliary_channels}}
    \label{eq:auxiliary_channel_Pout}
    \Big\{Y_{t,\mu} \sim P_\mathrm{out}\Big(\cdot\Big| \sqrt{\frac{1-t}{p}} [\bW \bZ^\star]_\mu + \sqrt{R_2(t,\epsilon)} V_\mu + \sqrt{Q_z t - R_2(t,\epsilon) + 2 s_n} A_\mu^\star\Big)\Big\}_{\mu=1}^m & \  \\
    \label{eq:auxiliary_channel_prior}
    \tilde{\bY}_t = (R_1(t,\epsilon))^{1/2} \star \bZ^\star + \bzeta, &
\end{subnumcases}
where $V_\mu,A^\star_\mu \overset{\textrm{i.i.d.}}{\sim} \mathcal{N}_\beta(0,1)$, and $\bzeta \sim \mathcal{N}_\beta(0,\mathbbm{1}_p)$.
The prior distribution on $\bZ^\star$ is given by $P_0^{(\bS)}$ in eq.~\eqref{eq:def_P0S}.
We assume that $\{V_\mu\}_{\mu=1}^m$ is known, and the inference problem is to recover both $\bA^\star \in \bbK^m$ and 
$\bZ^\star \in \bbK^p$ from the observations $(\tilde{\bY}_t,\{Y_{t,\mu}\}_{\mu=1}^m)$.
Note that $R_1 \in \mathcal{S}_\beta^+(\bbR)$, so its (matrix) square root is always uniquely defined.
Recall finally the definition of the $\star$ product in Section~\ref{subsec:app_definitions}.
In the following we will study the system of eq.~\eqref{eq:auxiliary_channels}.
In order to state our results fully rigorously, we need to add an hypothesis that can easily be relaxed:
  \begin{enumerate}[label=($h\arabic*^\star$)]
    \item \label{hyp:bounded_prior} The prior $P_0$ has bounded support.
  \end{enumerate}
  Under this hypothesis, $P_0^{(\bS)}$ is still defined by eq.~\eqref{eq:def_P0S}, and we can study the system of eq.~\eqref{eq:auxiliary_channels}.
  Nonetheless, this assumption \emph{a priori} rules out a Gaussian prior for $P_0$, and thus the correspondence between the system of eq.~\eqref{eq:auxiliary_channels} and 
  our original model.
  However, following the arguments of \cite{barbier2019optimal}, hypothesis~\ref{hyp:bounded_prior} can very easily be relaxed to the existence of 
  the second moment of $P_0$, which is then consistent with a Gaussian prior. In the following, we will thus work under hypothesis~\ref{hyp:prior}, 
  but we will sometimes as well use hypothesis~\ref{hyp:bounded_prior} without loss of generality.
We define $u_y(z) \equiv \ln P_\mathrm{out}(y|z)$, and
\begin{align}
    \label{eq:def_Stmu}
    S_{t,\mu} &\equiv \sqrt{\frac{1-t}{n}} [\bW \bZ^\star]_\mu + \sqrt{R_2(t,\epsilon)} V_\mu + \sqrt{Q_z t - R_2(t,\epsilon) + 2 s_n} A^\star_\mu, \\
    \label{eq:def_stmu}
    s_{t,\mu} &\equiv \sqrt{\frac{1-t}{n}} [\bW \bz]_\mu + \sqrt{R_2(t,\epsilon)} V_\mu + \sqrt{Q_z t - R_2(t,\epsilon) + 2 s_n} a_\mu.
\end{align}
The posterior distribution in this model can then be written as:
\begin{align}\label{eq:def_posterior_interpolated}
    \mathbb{P}_{n,t,\epsilon}\Big(\bz,\ba \Big| \bY_t, \tilde{\bY}_t\Big) \mathrm{d}\bz \ \mathrm{d}\ba &\equiv \frac{1}{\mathcal{Z}_{n,t,\epsilon}(\bY_t,\tilde{\bY}_t)} P_0^{(\bS)}(\mathrm{d}\bz) \mathcal{D}_\beta \ba \ e^{-\mathcal{H}_{t,\epsilon}(\bz,\ba; \bY_t,\tilde{\bY}_t, \bW, \bV)}.
\end{align}
To keep the notations lighter we omitted the conditioning on the variables $\bV,\bW$ which are assumed to be known. We defined the Hamiltonian:
\begin{align}\label{eq:def_Hamiltonian}
    \mathcal{H}_{t,\epsilon}(\bz,\ba; \bY_t,\tilde{\bY}_t, \bW, \bV) &\equiv - \sum_{\mu=1}^m u_{Y_{t,\mu}}(s_{t,\mu}) + \frac{\beta}{2} \sum_{k=1}^p \Big|\tilde{Y}_{t,k} - (R_1(t,\epsilon))^{1/2} \star z_k\Big|^2.
\end{align}
For any $t \in (0,1)$, we define the free entropy (the expectation is over all ``quenched'' variables, including $\bS$ if it is random):
\begin{align*}
    f_{n,\epsilon}(t) &\equiv \frac{1}{n} \EE \ln \mathcal{Z}_{n,t,\epsilon}(\bY_t,\tilde{\bY}_t).
\end{align*}
The following lemma gives the $t = 0$ and $t = 1$ limits of the free entropy:
\begin{lemma}\label{lemma:free_entropy_t01}
    $f_{n,\epsilon}(t)$ admits the following limit values for $t \in \{0,1\}$:
    \begin{subnumcases}{\hspace{-1.0cm}}
        \nonumber
        f_{n,\epsilon}(0) = f_n - \frac{\beta \delta}{2} + \smallO_n(1),& \\ 
        \nonumber
        f_{n,\epsilon}(1) = \Psi_0^{(\nu)}\Big(\int_0^1 r_\epsilon(t)\mathrm{d}t\Big) - \frac{\beta}{2} \Big[\delta + \EE_{\nu_B}[X] \int_0^1 r_\epsilon(t) \mathrm{d}t\Big] + \alpha \Psi_\mathrm{out}\Big(\int_0^1 q_\epsilon(t) \mathrm{d}t\Big) + \smallO_n(1). &
    \end{subnumcases}
\end{lemma}
\begin{proof}[Proof of Lemma~\ref{lemma:free_entropy_t01}]
    Using Lemma~5.1 of \cite{aubin2018committee}, there exists a constant $C>0$ such that for all $\epsilon \in \mathcal{D}_n^\beta \times [s_n,2s_n]$, one has $|f_{n,\epsilon}(0) - f_{n,(0,0)}(0)| \leq C s_n$.
    The proof of the value of $f_{n,\epsilon}(0)$ is then straightforwardly done by plugging $t = 0$ into the definition of $f_{n,\epsilon}$.
    At $t = 1$, the interpolation channels of eq.~\eqref{eq:auxiliary_channels} decouple, and we have:
    \begin{align*}
        f_{n,\epsilon}(1) &= \frac{1}{n} \EE \ln \int_{\bbK^p} P_0^{(\bS)}(\mathrm{d}\bz)  \exp \Big\{-\frac{\beta}{2} \sum_{k=1}^p \Big|\tilde{Y}_{1,k} - \Big(\epsilon_1 + \int_0^1 r_\epsilon(t) \mathbbm{1}_\beta \mathrm{d}t\Big)^{1/2} \star z_k\Big|^2 \Big\} \\
        &+ \frac{m}{n} \EE_{Y_1,V} \ln P_\mathrm{out}\Big(Y_{1} \Big | \Big(\epsilon_2 + \int_0^1 q_\epsilon(t)\mathrm{d}t\Big)^{1/2}V + \Big(Q_z + 2 s_n - \epsilon_2 - \int_0^1 q_\epsilon(t)\mathrm{d}t\Big)^{1/2} a \Big), \nonumber \\
        &=\frac{1}{n} \sum_{i=1}^{\min(n,p)} \int_\bbK \mathrm{d}Y \mathcal{D}_\beta X \frac{e^{-\frac{\beta}{2} |Y - S_i (R_1(1,\epsilon))^{1/2}\star X|^2}}{(2\pi/\beta)^{\beta/2}} \ln \Big\{\int \mathcal{D}_\beta x \ e^{-\frac{\beta}{2} |Y - S_i (R_1(1,\epsilon))^{1/2}\star x|^2}\Big\} \nonumber \\ 
         &+ \frac{1}{n} \sum_{i=\min(n,p)+1}^p \int_\bbK \mathrm{d}Y \frac{e^{-\frac{\beta}{2} |Y|^2}}{(2\pi/\beta)^{\beta/2}} \ln \Big\{e^{-\frac{\beta}{2} |Y|^2}\Big\} \Big)+ \alpha \Psi_\mathrm{out}\Big(\int_0^1 q_\epsilon(t) \mathrm{d}t\Big) + \smallO_n(1) .
    \end{align*}
   Recall that $R_1(1,\epsilon) = (\int_0^1 r_\epsilon(t) \mathrm{d}t) \mathbbm{1}_\beta + \smallO_n(1)$, so that up to $\smallO_n(1)$ terms
   the Gaussian integration on $X,x$ can be performed, which yields a Gaussian integration on $Y$, and we reach in the end:
   \begin{align*}
    f_{n,\epsilon}(1) &= -\frac{\beta p}{2n} - \frac{\beta}{2n} \sum_{i=1}^{\min(n,p)} \ln\Big(1+ S_i^2 \int_0^1 r_\epsilon(t) \mathrm{d}t\Big) + \alpha \Psi_\mathrm{out}\Big(\int_0^1 q_\epsilon(t) \mathrm{d}t\Big) + \smallO_n(1).
   \end{align*}
   Recall that $\nu_B$ is defined as the asymptotic eigenvalue distribution of $\bS^\intercal \bS$.
   By \ref{hyp:convergence_spectrum_B} we have:
   \begin{align*}
    f_{n,\epsilon}(1) &= \Psi_0^{(\nu)}\Big(\int_0^1 r_\epsilon(t)\mathrm{d}t\Big) - \frac{\beta}{2} \Big[\delta + \EE_{\nu_B}[X] \int_0^1 r_\epsilon(t) \mathrm{d}t\Big] + \alpha \Psi_\mathrm{out}\Big(\int_0^1 q_\epsilon(t) \mathrm{d}t\Big) + \smallO_n(1).
   \end{align*}
   which is what we wanted to show.
\end{proof}

\subsection{Free entropy variation}

Lemma~\ref{lemma:free_entropy_t01} gives a way to compute the free entropy $f_n$ by the fundamental theorem of analysis:
\begin{align}\label{eq:fundamental_analysis_theorem}
    f_n &= f_{n,\epsilon}(0) + \frac{\beta \delta}{2} + \smallO_n(1) = \frac{\beta \delta}{2} + f_{n,\epsilon}(1) - \int_0^1 f_{n,\epsilon}'(t) \mathrm{d}t.
\end{align}
We define the \emph{overlap} $Q$ and the \emph{overlap matrix} $Q^{(M)}$ as
\begin{subnumcases}{\label{eq:def_Q_QM}}
    Q \equiv \frac{1}{p} (\bZ^\star)^\intercal \bz, \hspace{1cm} Q^{(M)} \equiv Q & if $\beta = 1$, \\
    Q \equiv \frac{1}{p} (\bZ^\star)^\dagger \bz, \hspace{1cm} Q^{(M)} \equiv 
   \frac{1}{p} 
   \begin{pmatrix}
    \mathrm{Re}[\bZ^\star]^\intercal \mathrm{Re}[\bz] & \mathrm{Re}[\bZ^\star]^\intercal \mathrm{Im}[\bz] \\
    \mathrm{Im}[\bZ^\star]^\intercal \mathrm{Re}[\bz] & \mathrm{Im}[\bZ^\star]^\intercal \mathrm{Im}[\bz]
   \end{pmatrix} 
    & if $\beta = 2$.
\end{subnumcases}
Note that $Q \in \bbK$, $Q^{(M)} \in \mathcal{S}_\beta(\bbR)$ for $\beta = 1,2$, and that $\mathrm{Re}[Q] = \mathrm{Tr}_\beta[Q^{(M)}]$.
Finally, the Gibbs bracket $\langle \cdot \rangle_{n,t,\epsilon}$ is defined as the average over the posterior distribution 
of eq.~\eqref{eq:def_posterior_interpolated}.
 Recall that $u_y(z) \equiv \ln P_\mathrm{out}(y|z)$.
We can now state our identity for $f_{n,\epsilon}'(t)$, a counterpart to Proposition~3 of \cite{barbier2019optimal} and Proposition~5.2 of \cite{aubin2018committee}:
\begin{lemma}[Free entropy variation]\label{lemma:free_entropy_variation}
   For all $t \in (0,1)$ and $\epsilon \in \mathcal{D}_n^\beta \times [s_n,2s_n]$: 
   \begin{align*}
    f_{n,\epsilon}'(t) &= -\frac{1}{2\beta} \EE \Big \langle \Big(\frac{1}{n} \sum_{\mu=1}^m u'_{Y_{t,\mu}}(S_{t,\mu})^\dagger u'_{Y_{t,\mu}}(s_{t,\mu}) - \beta^2 \delta r_\epsilon(t) \Big) \cdot \Big(Q-q_\epsilon(t)\Big)\Big\rangle_{n,t,\epsilon} \nonumber \\ 
    &+ \frac{\beta \delta r_\epsilon(t)}{2}(q_\epsilon(t) - Q_z) + \smallO_n(1),
   \end{align*}
   in which $\smallO_n(1)$ is uniform in $t,\epsilon,q_\epsilon,r_\epsilon$.
\end{lemma}
\begin{proof}[Proof of Lemma~\ref{lemma:free_entropy_variation}]
     The proof is done in two steps. First, we show the following:
     \begin{align}\label{eq:df_dt_first_step}
        f_{n,\epsilon}'(t) &= -\frac{\beta \delta r_\epsilon(t)}{2} (Q_z -q_\epsilon(t)])
        + \frac{1}{2n\beta}\sum_{\mu=1}^m\EE\Big[\Big(Q_z - \frac{\norm{\bZ^\star}^2}{p}\Big) \frac{\Delta P_\mathrm{out}(Y_{t,\mu}|S_{t,\mu})}{P_\mathrm{out}(Y_{t,\mu}|S_{t,\mu})} \ln \mathcal{Z} \Big] \\
       & +  \frac{1}{2\beta}\EE\Big \langle \Big(\frac{1}{n} \sum_{\mu=1}^m u_{Y_{t,\mu}}'(S_{t,\mu})^\dagger u'_{Y_{t,\mu}}(s_{t,\mu}) - \beta^2 \delta r_\epsilon(t)\Big) \cdot \Big(q_\epsilon(t)- Q\Big) \Big \rangle_{n,t,\epsilon}. \nonumber
     \end{align}
     We will then build on this result by using the concentration of the free entropy of the interpolated model, cf.\ Theorem~\ref{thm:concentration_free_entropy} (which is independent of Lemma~\ref{lemma:free_entropy_variation}).
     From the definition of $f_{n,\epsilon}(t)$, we have (denoting $\mathcal{Z} \equiv \mathcal{Z}_{n,t,\epsilon}(\bY_t,\tilde{\bY}_t)$ to lighten the notations):
     \begin{align}\label{eq:df_dt_1}
         f_{n,\epsilon}'(t) &= -\frac{1}{n} \EE [\partial_t \mathcal{H}_{t,\epsilon}(\bZ^\star,\bA^\star;\bY_t,\tilde{\bY}_t,\bW,\bV) \ln \mathcal{Z}] - \frac{1}{n} \EE \langle \partial_t \mathcal{H}_{t,\epsilon}(\bz,\ba;\bY_t,\tilde{\bY}_t,\bW,\bV) \rangle_{n,t,\epsilon}.
     \end{align}
     The definition of $\mathcal{H}$ in eq.~\eqref{eq:def_Hamiltonian} gives, up to $\smallO_n(1)$ terms\footnote{Our conventions for derivatives of real functions of complex variables are reminded in Section~\ref{subsec:app_derivatives}.}:
     \begin{align}\label{eq:dH_dt}
       \partial_t \mathcal{H}_{t,\epsilon}(\bZ^\star,\bA^\star; \bY_t,\tilde{\bY}_t, \bW, \bV) 
       &= - \frac{\beta r_\epsilon(t)}{2 \sqrt{\int_0^t r_\epsilon(u) \mathrm{d}u}} \sum_{k=1}^p Z^\star_k \cdot \zeta_k + \sum_{\mu=1}^m \partial_t S_{t,\mu} \cdot u'_{Y_{t,\mu}}(S_{t,\mu}). 
     \end{align}
     By Proposition~\ref{prop:nishimori} (the Nishimori identity), we have:
     \begin{align*}
    \EE \langle \partial_t \mathcal{H}_{t,\epsilon}(\bz,\ba;\bY_t,\tilde{\bY}_t,\bW,\bV) \rangle_{n,t,\epsilon}
    &= \EE [\partial_t \mathcal{H}_{t,\epsilon}(\bZ^\star,\bA^\star;\bY_t,\tilde{\bY}_t,\bW,\bV) ]  \nonumber, \\
    &= \EE \Big[\sum_{\mu=1}^m \partial_t S_{t,\mu} \cdot \frac{P_\mathrm{out}'(Y_{t,\mu}|S_{t,\mu})}{P_\mathrm{out}(Y_{t,\mu}|S_{t,\mu})} \Big] + \smallO_n(1) = \smallO_n(1),
     \end{align*}
     as can be seen from eq.~\eqref{eq:dH_dt}. The first term of eq.~\eqref{eq:df_dt_1} can be written (up to $\smallO_n(1)$ terms) as 
     the sum of four contributions that we will compute successively, using Stein's lemma (see eqs.~\eqref{eq:stein_lemma_order_1},\eqref{eq:stein_lemma_order_2}).
     We start with the first one:
     \begin{align}
        \nonumber
        \frac{\beta r_{\epsilon(t)}}{2 n \sqrt{\int_0^tr_\epsilon(u)\mathrm{d}u}} \sum_{k=1}^p \EE[Z_k^\star \cdot \zeta_k \ln \mathcal{Z}]
        &= \frac{r_\epsilon(t)}{2n \sqrt{\int_0^tr_\epsilon(u)\mathrm{d}u}} \sum_{k=1}^p \EE\Big[Z_k^\star \cdot \frac{\mathrm{d}}{\mathrm{d}\zeta_k}\ln \mathcal{Z}\Big], \\
        \nonumber
         &\hspace{-2cm}=\frac{-\beta r_\epsilon(t)}{2n \sqrt{\int_0^tr_\epsilon(u)\mathrm{d}u}} \sum_{k=1}^p \EE [Z_k^\star \cdot \langle R_1(t,\epsilon)^{1/2} \star (Z^\star_k - z_k) + \zeta_k\rangle_{n,t,\epsilon}], \\
         \nonumber
         &= \frac{-\beta r_\epsilon(t)}{2n} \sum_{k=1}^p \EE [|Z_k^\star|^2 - Z_k^\star \cdot \langle z_k \rangle_{n,t,\epsilon}] + \smallO_n(1) \\ 
         \label{eq:first_term_df_dt}
         &= \frac{-\beta \delta r_\epsilon(t)}{2}(Q_z - \EE[\langle Q\rangle_{n,t,\epsilon}]) + \smallO_n(1).
     \end{align}
     We used the Nishimori identity Proposition~\ref{prop:nishimori} in the last equation.
     We now turn to the second term, and in a similar way we reach, by integration by parts with respect to $\bW$ (recall the definition of the Laplace operator in eq.~\eqref{eq:def_laplace}):
     \begin{align*}
        &\frac{1}{\sqrt{p(1-t)}} \sum_{\mu=1}^m \EE\Big[[\bW \bZ^\star]_\mu \cdot u'_{Y_{t,\mu}}(S_{t,\mu})\ln \mathcal{Z}\Big] \nonumber \\ 
        &= \frac{1}{\beta}\sum_{\mu=1}^m\EE\Big[\frac{\norm{\bZ^\star}^2}{p}  (\Delta u_{Y_{t,\mu}}(S_{t,\mu}) + |u'_{Y_{t,\mu}}(S_{t,\mu})|^2) \ln \mathcal{Z}
        \\
        &\hspace{3cm}+\Big \langle \Big[(u_{Y_{t,\mu}}'(S_{t,\mu}))^\dagger u'_{Y_{t,\mu}}(s_{t,\mu})\Big] \cdot \Big[\frac{(\bZ^\star)^\dagger \bz}{p}\Big] \Big \rangle_{n,t,\epsilon}\Big],\nonumber \\
        &= \frac{1}{\beta}\sum_{\mu=1}^m\EE\Big[\frac{\norm{\bZ^\star}^2}{p} \frac{\Delta P_\mathrm{out}(Y_{t,\mu}|S_{t,\mu})}{P_\mathrm{out}(Y_{t,\mu}|S_{t,\mu})} \ln \mathcal{Z}
        +  \Big \langle \big[(u_{Y_{t,\mu}}'(S_{t,\mu}))^\dagger u'_{Y_{t,\mu}}(s_{t,\mu})\big] \cdot \Big[\frac{(\bZ^\star)^\dagger \bz}{p}\Big] \Big \rangle_{n,t,\epsilon}\Big].
     \end{align*}
     We used in the last equation that $\Delta u_y(x) + |u'_y(x)|^2 = \Delta P_\mathrm{out}(y|x) / P_\mathrm{out}(y|x)$. 
     Integrating by parts with respect to $V_\mu,A_\mu^\star$, we obtain in a similar way:
     \begin{align*}
        &\EE \sum_{\mu=1}^m \Big[\frac{q_\epsilon(t) V_\mu}{\sqrt{R_2(t,\epsilon)}} + \frac{(Q_z - q_\epsilon(t)) A_\mu^\star}{\sqrt{Q_z t - R_2(t,\epsilon) + 2 s_n}}\Big] \cdot u'_{Y_{t,\mu}}(S_{t,\mu})\ln \mathcal{Z} \nonumber \\ 
        &= \frac{1}{\beta}\sum_{\mu=1}^m\EE\Big[Q_z \frac{\Delta P_\mathrm{out}(Y_{t,\mu}|S_{t,\mu})}{P_\mathrm{out}(Y_{t,\mu}|S_{t,\mu})} \ln \mathcal{Z} 
        + q_\epsilon(t) \langle u'_{Y_{t,\mu}}(S_{t,\mu}) \cdot u'_{Y_{t,\mu}}(s_{t,\mu}) \rangle_{n,t,\epsilon} \Big].
     \end{align*}
     By using the Nishimori identity, we obtain after summing all the previous terms the sought eq.~\eqref{eq:df_dt_first_step}:
     \begin{align*}
        f_{n,\epsilon}'(t) &= -\frac{\beta \delta r_\epsilon(t)}{2} (Q_z -q_\epsilon(t))
        + \frac{1}{2n\beta}\sum_{\mu=1}^m\EE\Big[\Big(Q_z - \frac{\norm{\bZ^\star}^2}{p}\Big) \frac{\Delta P_\mathrm{out}(Y_{t,\mu}|S_{t,\mu})}{P_\mathrm{out}(Y_{t,\mu}|S_{t,\mu})} \ln \mathcal{Z} \Big] \\
       & +  \frac{1}{2\beta}\EE\Big \langle \Big(\frac{1}{n} \sum_{\mu=1}^m u_{Y_{t,\mu}}'(S_{t,\mu})^\dagger u'_{Y_{t,\mu}}(s_{t,\mu}) - \beta^2 \delta r_\epsilon(t)\Big) \cdot (q_\epsilon(t)- Q) \Big \rangle_{n,t,\epsilon}. \nonumber
     \end{align*}
     To finish the proof, we must therefore just show that $\lim_{n \to \infty} B_n = 0$ 
     uniformly in $t,\epsilon,q_\epsilon,r_\epsilon$, with
     \begin{align*}
        B_n &\equiv \frac{1}{n}\sum_{\mu=1}^m\EE\Big[\Big(Q_z - \frac{\norm{\bZ^\star}^2}{p}\Big) \frac{\Delta P_\mathrm{out}(Y_{t,\mu}|S_{t,\mu})}{P_\mathrm{out}(Y_{t,\mu}|S_{t,\mu})} \ln \mathcal{Z} \Big].
     \end{align*}
     First, note that 
     \begin{align*}
        \EE\Big[\Big(Q_z - \frac{\norm{\bZ^\star}^2}{p}\Big) \frac{\Delta P_\mathrm{out}(Y_{t,\mu}|S_{t,\mu})}{P_\mathrm{out}(Y_{t,\mu}|S_{t,\mu})} \Big] &= \EE\Big[\Big(Q_z - \frac{\norm{\bZ^\star}^2}{p}\Big) \EE \Big[\frac{\Delta P_\mathrm{out}(Y_{t,\mu}|S_{t,\mu})}{P_\mathrm{out}(Y_{t,\mu}|S_{t,\mu})} \Big| \bZ^\star, \bS_t\Big]\Big] = 0, 
     \end{align*}
     since $\int \mathrm{d}Y \nabla P_\mathrm{out}(Y|S) = 0$.
     Using this, we can write 
     \begin{align}\label{eq:Bn_after_removing_fn}
        B_n &= \frac{1}{n}\sum_{\mu=1}^m\EE\Big[\Big(Q_z - \frac{\norm{\bZ^\star}^2}{p}\Big) \frac{\Delta P_\mathrm{out}(Y_{t,\mu}|S_{t,\mu})}{P_\mathrm{out}(Y_{t,\mu}|S_{t,\mu})} (\ln \mathcal{Z} - f_{n,\epsilon}(t))\Big].
     \end{align}
     We then follow exactly the lines of Appendix~A.5.2 of \cite{barbier2019optimal}, let us recall its main steps.
     Starting from eq.~\eqref{eq:Bn_after_removing_fn}, one uses the Cauchy-Schwarz inequality alongside Theorem~\ref{thm:concentration_free_entropy} (which is independent of Lemma~\ref{lemma:free_entropy_variation}), that gives
     $\EE[(\ln \mathcal{Z}/n - f_{n,\epsilon}(t))^2] \to 0$ uniformly in $t$.
     The expectation of the square of the other terms in eq.~\eqref{eq:Bn_after_removing_fn} can easily be bounded using hypotheses~\ref{hyp:channel},\ref{hyp:bounded_prior},\ref{hyp:convergence_spectrum_B}, uniformly in $t$. 
     Combining these bounds then shows that $B_n \to 0$ uniformly in $t$, which finishes the proof.
\end{proof}

\subsection{Concentration of the free entropy and the overlap}\label{subsec:app_concentrations}
We denote the mean over $\epsilon$ as:
\begin{align*}
    \EE_\epsilon[\cdot] &\equiv \frac{1}{s_n \mathrm{Vol}(\mathcal{D}_n^\beta)} \int_{\mathcal{D}_n^\beta} \mathrm{d}\epsilon_1 \int_0^1 \mathrm{d}\epsilon_2 [\cdot].
\end{align*}
In \cite{barbier2019optimal,aubin2018committee,barbier2019overlap}, the authors give a quite technical proof of the concentration of the free entropy and the overlap of an interpolated system 
close to the one described in Section~\ref{subsec:app_interpolating_model}. 
We present here two results of this type.
The first one concerns the concentration of the free entropy of the interpolated system\footnote{Recall the definition of $\mathcal{Z}_{n,t,\epsilon}$ in eq.~\eqref{eq:def_posterior_interpolated}.}. 
It is very similar to Theorem~6 of \cite{barbier2019optimal}.
\begin{theorem}[Free entropy concentration]\label{thm:concentration_free_entropy}
    Under the assumptions of Theorem~\ref{thm:main_product_gaussian_rotinv}, there exists a constant $C > 0$ that does not depend on $n,t,\epsilon$ and such that
    for all $n,t,\epsilon,q_\epsilon,r_\epsilon$:
    \begin{align*}
        \EE\Big[\Big(\frac{1}{n} \ln \mathcal{Z}_{n,t,\epsilon}(\bY_t,\tilde{\bY_t}) - \frac{1}{n} \EE \ln \mathcal{Z}_{n,t,\epsilon}(\bY_t,\tilde{\bY_t})\Big)^2\Big] &\leq \frac{C}{n}.
    \end{align*}
\end{theorem}
\noindent
Our second theorem concerns the concentration of the overlap. It will follow as an almost immediate consequence of a result of \cite{barbier2019overlap}. 
Before stating it, we introduce a regularity notion for our interpolation functions of eq.~\eqref{eq:def_R1R2}:
\begin{definition}[Regularity]\label{def:regularity}
    The families of functions $(q_\epsilon),(r_\epsilon)$ for $\epsilon \in \mathcal{D}_n^\beta \times [s_n, 2 s_n]$ 
    are said to be \emph{regular} if there exists $\gamma > 0$ such that for all $t \in [0,1]$ the mapping $\epsilon \mapsto R(t,\epsilon) \equiv (R_1(t,\epsilon),R_2(t,\epsilon))$ is a 
    $\mathcal{C}^1$ diffeomorphism whose Jacobian $J_{n,\epsilon}(t)$ satisfies $J_{n,\epsilon}(t) \geq \gamma$ for all $t\in [0,1]$ and all $\epsilon$.
\end{definition}
\noindent
We can now state our theorem on the concentration of the overlap $Q$:
\begin{theorem}[Overlap concentration] \label{thm:concentration_overlap}
    Under \ref{hyp:channel},\ref{hyp:bounded_prior},\ref{hyp:phi_product},\ref{hyp:convergence_spectrum_B}, and if the functions 
    $(q_\epsilon,r_\epsilon)$ are regular (cf.\ Definition~\ref{def:regularity}), then there exists a sequence $s_n$ going to $0$ (arbitrarily slowly) such that
    \begin{align*}
        \EE_\epsilon \int_0^1 \mathrm{d}t \  \EE \langle |Q-\EE\langle Q\rangle_{n,t,\epsilon}|^2 \rangle_{n,t,\epsilon} &= \smallO_n(1),
    \end{align*}
    with $\smallO_n(1)$ uniform in the choice of $r_\epsilon,q_\epsilon$.
\end{theorem}
\noindent
The rest of this section is dedicated to the proofs of Theorem~\ref{thm:concentration_free_entropy} and Theorem~\ref{thm:concentration_overlap}.

\subsubsection{Proof of Theorem~\ref{thm:concentration_free_entropy}}

The proof described in Section~E.1 of \cite{barbier2019optimal} can be adapted verbatim in this setting.
It relies on two concentration inequalities \cite{boucheron2013concentration}, that we recall here in the complex and real settings.
\begin{proposition}[Gaussian Poincar\'e inequality]\label{prop:gaussian_poincare}
    Let $\bU \in \bbK^n$ be distributed according to $\mathcal{N}_\beta(0,\mathbbm{1}_n)$, and $g : \bbK^n \to \bbR$ a $\mathcal{C}^1$ function. Recall our conventions for derivatives, see Section~\ref{subsec:app_derivatives}.
    Then 
    \begin{align*}
        \EE[g(\bU)^2] - \EE[g(\bU)]^2 &\leq \frac{1}{\beta} \EE[\norm{\nabla g(\bU)}^2].
    \end{align*}
\end{proposition}
\begin{proposition}[Bounded differences inequality]\label{prop:bounded_differences}
    Let $\mathcal{B} \subset \bbK$, and $g : \mathcal{B}^n \to \bbR$ a function such that there exists $c_1,\cdots,c_n \geq 0$ that satisfy for all $i \in \{1,\cdots,n\}$:
    \begin{align*}
        \sup_{\substack{u_1,\cdots,u_n \in \mathcal{B}^n \\ u'_i \in \mathcal{B}}} |g(u_1,\cdots,u_i,\cdots,u_n) - g(u_1,\cdots,u_{i-1},u'_i,u_{i+1},\cdots,u_n)| \leq c_i.
    \end{align*}
    Then if $\bU \in \bbK^n$ is a random vector of independent random variables with value in $\mathcal{B}$, we have:
    \begin{align*}
        \EE[g(\bU)^2] - \EE[g(\bU)]^2 &\leq \frac{\beta}{4} \sum_{i=1}^n c_i^2.
    \end{align*}
\end{proposition}
\noindent
Proposition~\ref{prop:gaussian_poincare} is used to show the concentration of $(\ln \mathcal{Z}_{n,t,\epsilon})/n$ with respect to the Gaussian 
variables $\bzeta$, $\bW$, $\bA^\star$, $\bV$, while Proposition~\ref{prop:bounded_differences} is used to show the concentration  
with respect to $\bZ^\star$. Using this strategy, the proof of \cite{barbier2019optimal} is directly transposed here, and we do not repeat it.

\subsubsection{Proof of Theorem~\ref{thm:concentration_overlap}}

We start with a lemma on the average value of $Q^{(M)}$ under $\EE \langle \cdot \rangle$, in the complex case.
\begin{lemma}\label{lemma:average_value_complex}
    Assume $\beta = 2$. Then
\begin{subnumcases}{}
\nonumber
    \EE \langle Q^{(M)}_{12} \rangle_{n,t,\epsilon} = \EE \langle Q^{(M)}_{21} \rangle_{n,t,\epsilon} = \smallO_n(1), & \\
    \nonumber
    \EE \langle Q^{(M)}_{11} \rangle_{n,t,\epsilon} - \EE \langle Q^{(M)}_{22} \rangle_{n,t,\epsilon} = \smallO_n(1), &
\end{subnumcases}
in which $\smallO_n(1)$ is uniform in $t,\epsilon,q_\epsilon,r_\epsilon$.
\end{lemma}
\begin{proof}[Proof of Lemma~\ref{lemma:average_value_complex}]
    By the classical theorems of continuity and derivability under the integral sign, it is easy to see that $\EE \langle Q^{(M)} \rangle_{n,t,\epsilon}$ is a continuous function of 
    $(R_1,R_2)$, and moreover that it admits a Lipschitz constant $K > 0$, independent of $t,\epsilon,q_\epsilon,r_\epsilon$.
    Indeed, thanks to hypotheses~\ref{hyp:channel},\ref{hyp:bounded_prior},\ref{hyp:phi_product},\ref{hyp:convergence_spectrum_B}, the domination hypotheses of these theorems are satisfied, and one can easily bound the differential of 
    $\EE \langle Q \rangle$ to obtain the existence of the Lipschitz constant $K > 0$.
    Moreover, for $\epsilon_1 = 0$, $\epsilon_2 = 0$, it is easy to check by the Nishimori identity Proposition~\ref{prop:nishimori} that we have: 
\begin{subnumcases}{}
    \EE \langle Q^{(M)}_{12} \rangle_{n,t,\epsilon} = \EE \langle Q^{(M)}_{21} \rangle_{n,t,\epsilon} = 0, & \nonumber \\
    \EE \langle Q^{(M)}_{11} \rangle_{n,t,\epsilon} = \EE \langle Q^{(M)}_{22} \rangle_{n,t,\epsilon}. &\nonumber
\end{subnumcases}
Using the Lipschitz constant $K > 0$ (which does not depend on the parameters $t,\epsilon,q_\epsilon,r_\epsilon$) and the fact that $\epsilon_1,\epsilon_2 = \mathcal{O}(s_n) = \smallO_n(1)$, this ends the proof.
\end{proof}
\noindent
Moreover, once averaged over $\epsilon_2 \in [s_n,2s_n]$ and $t \in (0,1)$, and using the concentration of the free entropy (Theorem~\ref{thm:concentration_free_entropy}), the results of \cite{barbier2019overlap} 
imply the thermal and total concentration of the overlap matrix $Q^{(M)}$ defined in eq.~\eqref{eq:def_Q_QM}:
\begin{lemma}\label{lemma:overlap_concentration}
    Assuming that $(q_\epsilon,r_\epsilon)$ are regular, 
    there exists a sequence $s_n \to 0$ (slowly enough) and $\eta,C > 0$ such that (with $\norm{\cdot}_F$ the Frobenius norm):
    \begin{subnumcases}{}
    \nonumber
        \EE_\epsilon \int_0^1 \mathrm{d}t \ \EE \langle \lVert Q^{(M)} - \langle Q^{(M)} \rangle_{n,t,\epsilon}\rVert_F^2 \rangle_{n,t,\epsilon} \leq \frac{C}{n^\eta}, &\\
        \nonumber
        \EE_\epsilon \int_0^1 \mathrm{d}t \ \EE \langle \lVert Q^{(M)} - \EE \langle Q^{(M)} \rangle_{n,t,\epsilon}\rVert_F^2 \rangle_{n,t,\epsilon} \leq \frac{C}{n^\eta}. &
    \end{subnumcases}
\end{lemma}
\begin{proof}[Proof of Lemma~\ref{lemma:overlap_concentration}]
    We can use the results of \cite{barbier2019overlap}, under two conditions: $(i)$ the concentration of the free entropy, 
    which is given here by Theorem~\ref{thm:concentration_free_entropy}, and $(ii)$ the regularity of $(q_\epsilon,r_\epsilon)$. Indeed, the results of \cite{barbier2019optimal} 
    give the concentration results as integrated over the matrix $R_1(t,\epsilon)$. Using the regularity assumption, we can lower bound these integrals by 
    integrals over the perturbation matrix $\epsilon_1$ (up to a multiplicative constant, which is uniform in all the relevant parameters), which then yields Lemma~\ref{lemma:overlap_concentration}.
    This argument was also made in a very close setting in \cite{barbier2019optimal,aubin2018committee}.
\end{proof}
\noindent
Using Lemma~\ref{lemma:average_value_complex} (if $\beta = 1$ this lemma is not needed) alongside Lemma~\ref{lemma:overlap_concentration} 
yields Theorem~\ref{thm:concentration_overlap}, since $Q = \mathrm{Tr}_\beta[Q^{(M)}]$.

\subsection{Upper and lower bounds}

\begin{proposition}[Fundamental sum rule]\label{prop:fundamental_sum_rule}
    Assume that $(q_\epsilon,r_\epsilon)$ are regular (cf Definition~\ref{def:regularity}), and that for all $\epsilon \in \mathcal{D}_n^\beta \times [s_n,2s_n]$ and 
    $t \in (0,1)$ we have $q_\epsilon(t) = \mathrm{Tr}_\beta[\EE\langle Q^{(M)}\rangle_{n,t,\epsilon}]$. 
    Then:
    \begin{align*}
        f_n &= \EE_\epsilon \Big[\Psi_0^{(\nu)}\Big(\int_0^1 r_\epsilon(t) \mathrm{d}t\Big) + \alpha \Psi_\mathrm{out}\Big(\int_0^1 q_\epsilon(t)\mathrm{d}t\Big) - \frac{\beta \delta}{2} \int_0^1 q_\epsilon(t) r_\epsilon(t) \mathrm{d}t\Big] + \smallO_n(1),
    \end{align*}
    in which $\smallO_n(1)$ is uniform in the choice of $q_\epsilon,r_\epsilon$.
\end{proposition}
\begin{proof}[Proof of Proposition~\ref{prop:fundamental_sum_rule}]
    The proof is based on Lemma~\ref{lemma:free_entropy_t01} and Lemma~\ref{lemma:free_entropy_variation}. Replacing their results into eq~\eqref{eq:fundamental_analysis_theorem}, 
    in order to finish the proof, we only need to show that $\lim_{n \to \infty} \Gamma_n = 0$ (uniformly in $r_\epsilon,q_\epsilon$), with 
    \begin{align*}
       \Gamma_n &\equiv \Big(\EE_\epsilon \int_0^1 \mathrm{d}t \ \EE\Big \langle \Big(\frac{1}{n} \sum_{\mu=1}^m u_{Y_{t,\mu}}'(S_{t,\mu})^\dagger u'_{Y_{t,\mu}}(s_{t,\mu}) - \beta^2 \delta r_\epsilon(t)\Big) \cdot \Big(q_\epsilon(t)- Q\Big) \Big \rangle_{n,t,\epsilon} \Big)^2.
    \end{align*}
    By the Cauchy-Schwarz inequality, we can bound:
    \begin{align*}
        \Gamma_n &\leq \EE_\epsilon \int_0^1 \mathrm{d}t \ \EE\Big \langle \Big|\frac{1}{n} \sum_{\mu=1}^m u_{Y_{t,\mu}}'(S_{t,\mu})^\dagger u'_{Y_{t,\mu}}(s_{t,\mu}) - \beta^2 \delta r_\epsilon(t)\Big|^2\Big\rangle_{n,t,\epsilon}  \\ 
        & \hspace{1cm} \times \EE_\epsilon \int_0^1 \mathrm{d}t \ \EE\langle |Q - q_\epsilon(t)|^2\rangle_{n,t,\epsilon}.
    \end{align*}
    The first term is bounded by a constant $C > 0$ by Lemma~\ref{lemma:boundedness_fluctuation} (recall that $r_\epsilon(t)$ is bounded as well by $r_\mathrm{max}$). 
    By Theorem~\ref{thm:concentration_overlap}, the second term is $\smallO_n(1)$, uniformly in $q_\epsilon,r_\epsilon$,
    since we assumed that $q_\epsilon(t) = \mathrm{Tr}_\beta[\EE \langle Q \rangle]$.
   As the vanishing terms are uniform in $q_\epsilon,r_\epsilon$, this shows that $\lim_{n \to \infty} \Gamma_n = 0$, which ends the proof.
\end{proof}
\noindent
Before obtaining the two bounds from the fundamental sum rule, we need a final preparatory lemma, that will 
imply the regularity of the functions $(q_\epsilon,r_\epsilon)$ that we will chose to derive the bounds.
\begin{lemma}[Regularity]\label{lemma:diffeomorphism}
    We define $F_n(t,R(t,\epsilon)) = (F_n^{(1)}(t,R(t,\epsilon)),F_n^{(2)}(t,R(t,\epsilon)))$, with:
    \begin{subnumcases}{}
    \nonumber
    F_n^{(1)}(t,R(t,\epsilon)) \equiv \Big(\frac{2 \alpha}{\beta \delta} \Psi_{\mathrm{out}}'(\mathrm{Tr}_\beta[\EE \langle Q^{(M)} \rangle_{n,t,\epsilon}])\Big) \mathbbm{1}_\beta, & \\
    \nonumber
    F_n^{(2)}(t,R(t,\epsilon)) \equiv \mathrm{Tr}_\beta[\EE \langle Q^{(M)} \rangle_{n,t,\epsilon}]. &
    \end{subnumcases}
    Then $F_n$ is a continuous function from its domain to $\bbR^2$. Moreover, it admits partial derivatives with respect to both $R_1$ and $R_2$ on the interior 
    of its domain. We have, uniformly over the choice of $(q_\epsilon,r_\epsilon)$:
    \begin{subnumcases}{}
    \nonumber
        \liminf_{n \to \infty} \inf_{t \in (0,1)} \inf_{\substack{\epsilon_1 \in \mathcal{D}_n \\ \epsilon_2 \in [s_n,2s_n]}} \sum_{l=1}^\beta \frac{\partial (F_n^{(1)})_{ll}}{\partial (R_1)_{ll}}(t,R(t,\epsilon)) \geq 0, &\\
        \nonumber
        \frac{\partial F_n^{(2)}}{\partial R_2}(t,R(t,\epsilon)) \geq 0. &
    \end{subnumcases}
\end{lemma}
\begin{proof}[Proof of Lemma~\ref{lemma:diffeomorphism}]
    The proof is very close to the arguments of Lemma~5.5 of \cite{aubin2018committee}.
    The continuity and derivability follow from standard theorems of continuity and derivation 
    under the integral sign, thanks to hypotheses \ref{hyp:channel},\ref{hyp:bounded_prior},\ref{hyp:convergence_spectrum_B}.
    Indeed, under these boundedness assumptions, the domination hypotheses of these theorems are straightforwardly satisfied.
    Let us start with the first inequality.
    We can easily write:
    \begin{align*}
         \sum_{l=1}^\beta \frac{\partial (F_n^{(1)})_{ll}}{\partial (R_1)_{ll}}= \frac{2 \alpha}{\beta \delta} \Psi_\mathrm{out}''(\mathrm{Tr}_\beta[\EE\langle Q^{(M)} \rangle]) \sum_{l=1}^\beta \frac{\partial \mathrm{Tr}_\beta \EE \langle Q^{(M)} \rangle}{\partial (R_1)_{ll}}.
    \end{align*}
    The convexity of $\Psi_\mathrm{out}$ was already derived so that $\Psi_\mathrm{out}'' \geq 0$.
    Moreover, since $R_1$ is the SNR matrix of a linear channel, we know that the matrix $\nabla_{R_1} \EE \langle Q^{(M)} \rangle$ is positive \cite{aubin2018committee}. In particular, its trace is always positive, 
    and by Lemma~\ref{lemma:average_value_complex}:
    \begin{align*}
    \sum_{l=1}^\beta \frac{\partial \mathrm{Tr}_\beta \EE \langle Q^{(M)} \rangle}{\partial (R_1)_{ll}} &= \underbrace{\mathrm{Tr}_\beta[\nabla_{R_1} \EE \langle Q^{(M)} \rangle]}_{\geq 0} + \smallO_n(1),
    \end{align*}
    with a $\smallO_n(1)$ uniform in $t,\epsilon,r_\epsilon,q_\epsilon$. This shows the first inequality.
    Let us sketch the argument for the second inequality. The trace of $Q^{(M)}$ is directly related to the MMSE on the complex vector $\bZ^\star$ by:
    \begin{align*}
        \frac{1}{p} \mathrm{MMSE}(\bZ^\star|\bY_t,\tilde{\bY}_t,\bV,\bW) &= \frac{1}{p} \EE [\lVert\bZ^\star - \langle \bz \rangle\rVert^2] = Q_z - \mathrm{Tr}_\beta[\EE \langle Q^{(M)} \rangle].
    \end{align*} 
    The fact that the MMSE should decrease as the SNR $R_2$ increases, for a channel of the type of eq.~\eqref{eq:auxiliary_channel_Pout},
    is very natural, and it was proven in Proposition~6 of \cite{barbier2019optimal}, which applies here. This proposition yields that $\mathrm{Tr}_\beta[\EE \langle Q^{(M)} \rangle]$ is a nondecreasing function of $R_2$, 
    which ends the proof.
\end{proof}
\noindent
Finally, we define the \emph{replica-symmetric potential}, that appears in Proposition~\ref{prop:replicas_product_structure}:
\begin{align*}
 f_\mathrm{RS}(q,r) &\equiv -\frac{\beta \delta r q}{2} + \Psi_0^{(\nu)}(r) + \alpha \Psi_{\mathrm{out}}(q).
\end{align*}

\subsubsection{Lower bound}\label{subsubsec:lower_bound}

\begin{proposition}[Lower bound]\label{prop:lower_bound}
    Under the assumptions of Theorem~\ref{thm:main_product_gaussian_rotinv}, the free entropy $f_n$ satisfies:
    \begin{align*}
        \liminf_{n \to \infty} f_n \geq \sup_{r \geq 0} \inf_{q \in [0,Q_z]} f_{\mathrm{RS}}(q,r).
    \end{align*}
\end{proposition}
\begin{proof}[Proof of Proposition~\ref{prop:lower_bound}]
    We fix $r \geq 0$ and $R_1(t) = \epsilon_1 + r t \mathbbm{1}_\beta$. We then choose $R_2(t)$ as the unique solution to the ordinary differential equation:
    \begin{align}\label{eq:ode_lower_bound}
        R_2'(t) = \mathrm{Tr}_\beta[\EE \langle Q^{(M)}\rangle_{n,t,\epsilon}], 
    \end{align}
    with boundary condition $R_2(0) = \epsilon_2$. We denote this unique solution as $R_2(t) = \epsilon_2 + \int_0^t q_\epsilon(r;v)\mathrm{d}v$.
    The ODE of eq.~\eqref{eq:ode_lower_bound} can easily be seen to satisfy the hypotheses of the parametric Cauchy-Lipschitz theorem (as a function of the initial condition $\epsilon_2$), 
    and by the Liouville formula (cf Lemma~A.3 of \cite{aubin2018committee}), the Jacobian $J_{n,\epsilon}(t)$ of $\epsilon \mapsto R(t,\epsilon) \equiv (R_1(t,\epsilon),R_2(t,\epsilon))$ 
    verifies:
    \begin{align*}
        J_{n,\epsilon}(t) &= \exp \Big(\int_0^t \frac{\partial \mathrm{Tr}_\beta[\EE \langle Q \rangle_{n,u,\epsilon}]}{\partial R_2}(u,R(u,\epsilon)) \mathrm{d}u\Big) \geq 1,
    \end{align*}
    in which the inequality is a consequence of Lemma~\ref{lemma:diffeomorphism}. The functions are thus regular in the sens of Definition~\ref{def:regularity}, and moreover 
    the local inversion theorem implies that $\epsilon \mapsto R(t,\epsilon)$ is a $\mathcal{C}^1$ diffeomorphism. 
    We can therefore use the fundamental sum rule Proposition~\ref{prop:fundamental_sum_rule} as all its hypotheses are verified. 
    We reach:
    \begin{align*}
        f_n &= \EE_\epsilon \Big[\Psi_0^{(\nu)}(r) + \alpha \Psi_\mathrm{out}\Big(\int_0^1q_\epsilon(r;t)\mathrm{d}t\Big) - \frac{\beta \delta r}{2} \int_0^1 q_\epsilon(r;t) \mathrm{d}t\Big] + \smallO_n(1), \\ 
        &= \EE_\epsilon\Big[f_\mathrm{RS}\Big(\int_0^1 q_\epsilon(r;t) \mathrm{d}t,r\Big)\Big] + \smallO_n(1), \\
        &\geq \inf_{q \in [0,Q_z]} f_\mathrm{RS}(q,r) + \smallO_n(1).
    \end{align*}
    Since this is true for all $r \geq 0$ we easily obtain the sought lower bound.
\end{proof}

\subsubsection{Upper bound}
We now prove the final upper bound, which will end the proof of Lemma~\ref{lemma:noniid_phi}.
\begin{proposition}[Upper bound]\label{prop:upper_bound}
    Under the assumptions of Theorem~\ref{thm:main_product_gaussian_rotinv}, the free entropy $f_n$ satisfies:
    \begin{align*}
        \limsup_{n \to \infty} f_n \leq \sup_{r \geq 0} \inf_{q \in [0,Q_z]} f_{\mathrm{RS}}(q,r).
    \end{align*}
\end{proposition}
\begin{proof}[Proof of Proposition~\ref{prop:upper_bound}]
    We will choose $R(t,\epsilon) = (R_1(t,\epsilon),R_2(t,\epsilon))$ as the solution to the ordinary differential equation:
    \begin{align}\label{eq:ode_upper_bound}
      \partial_t R_1(t,\epsilon) &= \frac{2 \alpha}{\beta \delta} \Psi_\mathrm{out}\Big[\mathrm{Tr}_\beta[\EE \langle Q^{(M)}\rangle_{n,t,\epsilon}]\Big] \mathbbm{1}_\beta, &&\partial_t R_2(t,\epsilon) = \mathrm{Tr}_\beta[\EE \langle Q^{(M)}\rangle_{n,t,\epsilon}], 
    \end{align}
    with initial conditions $R(0,\epsilon) = (\epsilon_1,\epsilon_2)$. Let us denote this equation as $\partial_t R(t) = (F_{n,1}(t,R(t)),F_{n,2}(t,R(t)))$. 
     As in Section~\ref{subsubsec:lower_bound}, the parametric Cauchy-Lipschitz theorem implies the existence, unicity and $\mathcal{C}^1$ regularity of $R(t,\epsilon)$ as a function of $(t,\epsilon)$.
    We denote this unique solution\footnote{Notice in particular that the first equation of eq.~\eqref{eq:ode_upper_bound} implies that the derivative $\partial_t R_1(t,\epsilon)$ is always a diagonal matrix in $\mathcal{S}_\beta(\bbR)$.}
     as $R_1(t,\epsilon) = \epsilon_1 + (\int_0^t r_\epsilon(v) \mathrm{d}v) \mathbbm{1}_\beta$, $R_2(t,\epsilon) = \epsilon_2 + \int_0^t q_\epsilon(v) \mathrm{d}v$.
     Again, the Liouville formula yields that the Jacobian $J_{n,\epsilon}(t)$ of the map $\epsilon \mapsto R(t,\epsilon)$ is given by:
     \begin{align}
        J_{n,\epsilon}(t) &= \exp \Big(\int_0^t \Big\{\sum_{l=1}^\beta \frac{\partial (F_{n,1})_{ll}}{\partial (R_1)_{ll}}(s,R(s,\epsilon)) + \frac{\partial F_{n,2}}{\partial R_2}(s,R(s,\epsilon)) \Big\} \mathrm{d}s\Big).
     \end{align}
     Then, by Lemma~\ref{lemma:diffeomorphism}, we have that $\liminf_{n \to \infty} \inf_t \inf_\epsilon J_{n,\epsilon}(t) \geq 1$.
     In particular, this implies that $(q_\epsilon,r_\epsilon)$ are regular in the sense of Definition~\ref{def:regularity}.
     We have all that is needed to apply Proposition~\ref{prop:fundamental_sum_rule} and we reach:
    \begin{align*}
        f_n &= \EE_\epsilon \Big[\Psi_0^{(\nu)}\Big(\int_0^1 r_\epsilon(t) \mathrm{d}t\Big) + \alpha \Psi_\mathrm{out}\Big(\int_0^1q_\epsilon(t)\mathrm{d}t\Big) - \frac{\beta \delta}{2} \int_0^1 q_\epsilon(t) r_\epsilon(t) \mathrm{d}t\Big] + \smallO_n(1).
    \end{align*}
    Since $\Psi_{\mathrm{out}}$ and $\Psi_0^{(\nu)}$ are convex, Jensen's inequality implies:
    \begin{align*}
        f_n &\leq \EE_\epsilon \int_0^1 \mathrm{d}t \Big[\Psi_0^{(\nu)}(r_\epsilon(t)) + \alpha \Psi_\mathrm{out}(q_\epsilon(t)) - \frac{\beta \delta}{2} q_\epsilon(t) r_\epsilon(t) \Big] + \smallO_n(1), \\ 
        &\leq \EE_\epsilon \int_0^1 \mathrm{d}t f_\mathrm{RS}(q_\epsilon(t),r_\epsilon(t)) + \smallO_n(1)
    \end{align*}
    Note that we have 
    \begin{align*}
        f_\mathrm{RS}(q_\epsilon(t),r_\epsilon(t)) &= \inf_{q \in [0,Q_z]}f_\mathrm{RS}(q,r_\epsilon(t)).
    \end{align*}
    Indeed, the function $q \mapsto f_\mathrm{RS}(q,r_\epsilon(t))$ is convex, and its derivative is zero for $q = q_\epsilon(t)$ by definition of $(r_\epsilon,q_\epsilon)$, cf eq.~\eqref{eq:ode_upper_bound}. Therefore, we have:
    \begin{align*}
        f_n &\leq \EE_\epsilon \int_0^1 \mathrm{d}t \inf_{q \in [0,Q_z]} f_\mathrm{RS}(q,r_\epsilon(t)) + \smallO_n(1), \\ 
        &\leq \sup_{r \geq 0} \inf_{q \in [0,Q_z]} f_\mathrm{RS}(q,r_\epsilon(t)) + \smallO_n(1), 
    \end{align*}
    which ends the proof.
\end{proof}

\subsection{Proof of the MMSE limit}\label{subsec:app_proof_mmse}

As mentioned in the main part of this work, the MMSE statement in Conjecture~\ref{conjecture:replicas_general} is stated informally. The main reason is that obtaining the MMSE limit generically requires many technicalities, to account for the possible symmetries of the system, see e.g.\ Theorem~2 of \cite{barbier2019optimal} which performs such an analysis. 
To simplify the analysis, we ``break'' this symmetry by adding a side channel with an arbitrarily small signal-to-noise ratio.
Formally, we consider the following inference problem made of two channels:
\begin{subnumcases}{\label{eq:side_channel_mmse}}
    \label{eq:main_channel_Pout}
    Y_{t,\mu} \sim P_\mathrm{out}\Big(\cdot\Big|\frac{1}{\sqrt{n}} \sum_{i=1}^n \Phi_{\mu i} X^\star_{i}\Big)&$\mu = 1,\cdots,m$  \\
    \label{eq:side_channel}
    \tilde{\bY}_t = \sqrt{\Lambda} \bX^\star + \bZ', & $\bZ' \sim \mathcal{N}_\beta(0,\mathbbm{1}_n)$,
\end{subnumcases}
with $\Lambda > 0$ (arbitrarily small).
We can now state our precise statement on the MMSE:
\begin{proposition}\label{prop:mmse_precise}
    Consider the inference problem of eq.~\eqref{eq:side_channel_mmse}, under \ref{hyp:channel},\ref{hyp:prior},\ref{hyp:phi_product},\ref{hyp:convergence_spectrum_B}. We denote $\langle \cdot \rangle$ the average with respect to the posterior distribution of $\bx$ under the problem of eq.~\eqref{eq:side_channel_mmse}. The minimum mean squared error is achieved by the Bayes-optimal estimator $\hat{\bX}_\mathrm{opt} = \langle \bx \rangle$, and it satisfies as $n \to \infty$:
    \begin{align}
       \lim_{n \to \infty} \mathrm{MMSE} &= \lim_{n \to \infty} \frac{1}{n} \EE\lVert \bX^\star - \langle \bx \rangle \rVert^2 = 1 - q_x^\star,
    \end{align}
    with $q_x^\star$ the solution of the extremization problem in eq.~\eqref{eq:replica:freeen}, taking into account the additional side information of eq.~\eqref{eq:side_channel}.
\end{proposition}
\begin{proof}[Proof of Proposition~\ref{prop:mmse_precise}]
    With the side channel added, this proposition will follow from an application of the classical I-MMSE theorem \cite{Guo_2005}. 
    We denote $\langle \cdot \rangle$ the mean under the posterior distribution of $\bx$ under the channels of eq.~\eqref{eq:side_channel_mmse}, and $\EE$ the average with respect to the ``quenched'' variables $\bPhi,\bZ',\bX^\star$. The free entropy $f_n(\Lambda)$ is defined as the average of the log-normalization of the posterior distribution:
    \begin{align*}
        f_n(\Lambda) &\equiv \frac{1}{n} \EE \ln \int_{\bbK^n} P_0(\mathrm{d}\bx) \Big[\prod_{\mu=1}^m P_\mathrm{out}\Big(Y_{t,\mu}\Big| \frac{1}{\sqrt{n}} \sum_{i=1}^n \Phi_{\mu i} x_i \Big) \Big] \frac{e^{-\frac{\beta}{2} \sum_{i=1}^n \big| \tilde{Y}_{t,i} - \sqrt{\Lambda} x_i \big|^2}}{(2\pi/\beta)^{n\beta/2}}.
    \end{align*}
    
    We can easily replicate the adaptive interpolation analysis of Theorem~\ref{thm:main_product_gaussian_rotinv} (see Section~\ref{sec:app_proof_theorem}) to this case, and we reach the following result for the asymptotic free entropy $f(\Lambda)$ of eq.~\eqref{eq:side_channel_mmse}:
    \begin{lemma}\label{lemma:free_entropy_side_channel}
        For all $\Lambda > 0$, we have $\lim_{n \to \infty} f_n(\Lambda) = f(\Lambda)$, given by:
        \begin{align}
            f(\Lambda) = \sup_{q_x \in [0,1]} \sup_{q_z \in [0,Q_z]} [I_0(q_x,\Lambda)+\alpha I_\mathrm{out}(q_z) + I_\mathrm{int}(q_x,q_z)], 
        \end{align}
        with $I_\mathrm{out},I_\mathrm{int}$ given in Conjecture~\ref{conjecture:replicas_general}, and:
        \begin{align*}
           I_0(q_x,\Lambda) \equiv \inf_{\hat{q}_x \geq 0} \Big[-\frac{\beta \hat{q}_x q_x}{2} +&\int_{\bbK^2} \mathcal{D}_\beta \xi \ \mathrm{d}\tilde{y} \int P_0(\mathrm{d}x) \frac{e^{-\frac{\beta \hat{q}_x}{2}|x|^2 + \beta \sqrt{\hat{q}_x} x \cdot \xi - \frac{\beta}{2} |\tilde{y} - \sqrt{\Lambda} x|^2}}{(2\pi/\beta)^{\beta/2}}  \\ 
           &\ln \int P_0(\mathrm{d}x) \frac{e^{-\frac{\beta \hat{q}_x}{2}|x|^2 + \beta \sqrt{\hat{q}_x} x \cdot \xi - \frac{\beta}{2} |\tilde{y} - \sqrt{\Lambda} x|^2}}{(2\pi/\beta)^{\beta/2}} \Big].
        \end{align*}
    \end{lemma}
    \begin{proof}[Proof of Lemma~\ref{lemma:free_entropy_side_channel}]
    By Proposition~\ref{prop:replicas_product_structure}, one can simply replicate the adaptive interpolation analysis of Section~\ref{sec:app_proof_theorem} to this model, and this will prove the required formula. The precise form of $I_0(q_x)$ is very easy to compute.
    \end{proof}
    We can then use the I-MMSE formula \cite{Guo_2005}, that yields that for any $\Lambda$, 
    \begin{align}
        \lim_{n \to \infty} \mathrm{MMSE} = -\frac{2}{\beta}\partial_\Lambda f(\Lambda). 
    \end{align}
    Moreover, by Lemma~\ref{lemma:free_entropy_side_channel}, $q_x^\star,\hat{q}_x^\star$ is a solution of the equation: 
    \begin{align}\label{eq:se_qx}
       q_x^\star = \frac{1}{(2\pi/\beta)^{\beta/2}} \int \mathcal{D}_\beta \xi \mathrm{d}\tilde{y} \frac{\Big| \int P_0(\mathrm{d}x) \ x \ e^{-\frac{\beta \hat{q}_x^\star}{2}|x|^2 + \beta \sqrt{\hat{q}_x^\star} x \cdot \xi - \frac{\beta}{2} |\tilde{y} - \sqrt{\Lambda} x|^2} \Big|^2}{\int P_0(\mathrm{d}x)e^{-\frac{\beta \hat{q}_x^\star}{2}|x|^2 + \beta \sqrt{\hat{q}_x^\star} x \cdot \xi - \frac{\beta}{2} |\tilde{y} - \sqrt{\Lambda} x|^2} }.
    \end{align}
    From the expression of $I_0$ in Lemma~\ref{lemma:free_entropy_side_channel} and eq.~\eqref{eq:se_qx}, it is then a straightforward calculation to see that 
    \begin{align*}
        - (2/\beta)\partial_\Lambda f(\Lambda) &= 1-q_x^\star,
    \end{align*}
    which ends the proof.
\end{proof}

\subsection{Proof of Theorem~\ref{thm:main_product_gaussian_rotinv}: the Gaussian matrix case}\label{subsec:app_gaussian_matrix_proof}

In this subsection, we place ourselves under \ref{hyp:channel},\ref{hyp:any_p0_gaussian_phi} and sketch how the proof performed in the previous sections directly extends under these hypotheses. Note that here $\langle \lambda \rangle_\nu = \alpha$, so $Q_z = Q_x = \rho$.
First, we can state a very similar result to Proposition~\ref{prop:replicas_product_structure}, simplifying Conjecture~\ref{conjecture:replicas_general} in this setting:
\begin{proposition}\label{prop:replicas_gaussian_phi}
    Under \ref{hyp:channel},\ref{hyp:any_p0_gaussian_phi}, the replica conjecture~\ref{conjecture:replicas_general} reduces to:
    \begin{align*}
       \lim_{n \to \infty} \frac{1}{n} \EE \ln \mathcal{Z}_n(\bY) &= \sup_{\hat{q} \geq 0} \inf_{q \in [0,\rho]} \Big[- \frac{\beta q \hat{q}}{2} +\Psi_{P_0}(\hat{q})+ \alpha \Psi_\mathrm{out}(q)\Big].
    \end{align*}
    with $q_z,\Psi_\mathrm{out}$ defined in Proposition~\ref{prop:replicas_product_structure}, and $\Psi_{P_0}(\hat{q})$ defined for $\hat{q} \geq 0$ by:
    \begin{align*}
       \Psi_{P_0}(\hat{q}) &\equiv \EE_\xi \mathcal{Z}_0(\sqrt{\hat{q}}\xi,\hat{q}) \ln \mathcal{Z}_0(\sqrt{\hat{q}}\xi,\hat{q}),
    \end{align*}
    with $\mathcal{Z}_0$ defined in eq.~\eqref{eq:auxiliary}.
\end{proposition}
\begin{proof}[Proof of Proposition~\ref{prop:replicas_gaussian_phi}]
   The proof follows similar lines to the proof of Proposition~\ref{prop:replicas_product_structure}, see Section~\ref{sec:app_equivalence_replicas}. Let us briefly sketch the main steps. Since $\bPhi$ is Gaussian, $\nu$ is the Marchenko-Pastur distribution \cite{marchenko1967distribution}, and one can easily simplify $I_\mathrm{int}(q_x,q_z)$ as:
   \begin{align*}
      I_\mathrm{int}(q_x,q_z) &= -\frac{\alpha \beta}{2}\Big[\frac{q_x (\rho-q_z)}{2 \rho(\rho-q_x)} + \ln(\rho-q_x) \Big]. 
   \end{align*}
   Using then the exact same sup-inf inversion arguments as in Section~\ref{sec:app_equivalence_replicas}, the supremum and infimum over $q_z$ and $\hat{q}_z$ are solved by:
   \begin{subnumcases}{}
      q_z = q_x + \frac{2}{\beta}(\rho-q_x)^2 \Psi_\mathrm{out}'(q_x), &\\ 
      \hat{q}_z = \frac{q_x}{\rho(\rho-q_x)} &.
   \end{subnumcases}
   And finally, we reach that (with the notations of Conjecture~\ref{conjecture:replicas_general}) $\alpha I_\mathrm{out}(q_z) + I_\mathrm{int}(q_x,q_z) = \alpha \Psi_\mathrm{out}(q_x)$. Posing $q = q_x,\hat{q} = \hat{q}_x$ finishes the proof.
\end{proof}
We turn now to proving the formula of Proposition~\ref{prop:replicas_gaussian_phi}. The proof goes exactly as in the previous sections of Section~\ref{sec:app_proof_theorem}, by considering instead of eq.~\eqref{eq:auxiliary_channels} the interpolation problem:
\begin{subnumcases}{\label{eq:auxiliary_channels_gaussian_phi}}
    \Big\{Y_{t,\mu} \sim P_\mathrm{out}\Big(\cdot\Big| \sqrt{\frac{1-t}{p}} [\bPhi \bX^\star]_\mu + \sqrt{R_2(t,\epsilon)} V_\mu + \sqrt{\rho t - R_2(t,\epsilon) + 2 s_n} A_\mu^\star\Big)\Big\}_{\mu=1}^m & \  \\
    \tilde{\bY}_t = (R_1(t,\epsilon))^{1/2} \star \bX^\star + \bzeta, &
\end{subnumcases}
where $V_\mu,A^\star_\mu \overset{\textrm{i.i.d.}}{\sim} \mathcal{N}_\beta(0,1)$, and $\bzeta \sim \mathcal{N}_\beta(0,\mathbbm{1}_n)$.
The prior distribution on $\bX^\star$ is $P_0$. 
The rest of the proof is then a trivial verbatim of Sections~\ref{subsec:app_interpolating_model} to \ref{subsec:app_proof_mmse}.
\section{Proof of Proposition~\ref{prop:replicas_product_structure}}\label{sec:app_equivalence_replicas}

In this section, we prove Proposition~\ref{prop:replicas_product_structure}: we start from Conjecture~\ref{conjecture:replicas_general} and derive eq.~\eqref{eq:Phi_noniid}.
Note that by \ref{hyp:phi_product} we have $\langle\lambda\rangle_\nu = \alpha \EE_{\nu_B}[X] / \delta$. 
We begin by recalling some sup-inf formulas, before turning to the actual proof.

\subsection{Some sup-inf formulas}

We recall Corollary~8 of \cite{barbier2019optimal}, stated here as a lemma:
\begin{lemma}[\cite{barbier2019optimal}]\label{lemma:supinf_barbier}
    Let $f : \bbR_+ \to \bbR$ be a $\mathcal{C}^1$ convex, non-decreasing, Lipschitz function. Define $\rho \equiv ||f'||_\infty$. Let $g : [0,\rho] \to \bbR$ be a convex, non-decreasing, Lipschitz function. For $(q_1,q_2) \in \bbR_+ \times [0,\rho]$ we define $\psi(q_1,q_2) \equiv f(q_1) + g(q_2) - q_1 q_2$. Then:
    \begin{align*}
        \sup_{q_1 \geq 0} \inf_{q_2 \in [0,\rho]} \psi(q_1,q_2) &= \sup_{q_2 \in [0,\rho]} \inf_{q_1 \geq 0} \psi(q_1,q_2).
    \end{align*}
\end{lemma}
We can state a corollary for functions of two variables.
\begin{corollary}\label{corollary:supinf_two_variables}
    Let $f : \bbR_+^2 \to \bbR$ be a $\mathcal{C}^1$ convex, Lipschitz function which is nondecreasing in each of its variables. Define $\rho_1 \equiv ||\partial_1 f||_\infty,\rho_2 \equiv ||\partial_2 f||_\infty$. Let $g : [0,\rho_1] \to \bbR$, $g_2 : [0,\rho_2] \to \bbR$ be two convex, non-decreasing, Lipschitz functions. For $(x_1,x_2,y_1,y_2) \in \bbR_+^2 \times [0,\rho_1] \times [0,\rho_2]$ we define $\psi(x_1,x_2,y_1,y_2) \equiv f(x_1,x_2) + g_1(y_1) + g_2(y_2) - x_1 y_1 - x_2 y_2$. Then:
    \begin{align*}
        \sup_{x_1,x_2 \geq 0} \inf_{y_1,y_2 \in [0,\rho_1]\times[0,\rho_2]} \psi(x_1,x_2,y_1,y_2) &= 
       \sup_{y_1,y_2 \in [0,\rho_1]\times[0,\rho_2]} \inf_{x_1,x_2 \geq 0} \psi(x_1,x_2,y_1,y_2).
    \end{align*}
\end{corollary}

\begin{proof}[Proof of Corollary~\ref{corollary:supinf_two_variables}]
The proof is a verbatim of the proof of Corollary~8 in \cite{barbier2019optimal}, using that at fixed $y$, $x \mapsto f(x,y)$ is $\rho_1$-Lipschitz, while at fixed $x$, $y \mapsto f(x,y)$ is $\rho_2$-Lipschitz.
\end{proof}

\subsection{Core of the proof}

We now turn to the proof of Proposition~\ref{prop:replicas_product_structure}.
We begin by simplifying the free entropy potential using the Gaussian prior. We start from Conjecture~\ref{conjecture:replicas_general}.
Since $P_0$ is Gaussian by \ref{hyp:prior}, we can easily simplify the prior term $I_0$ as:
\begin{align*}
   I_0(q_x) &= \inf_{\hat{q}_x \geq 0} \Big[\frac{\beta \hat{q}_x (1-q_x)}{2} - \frac{\beta}{2} \ln(1+\hat{q}_x)\Big] = \frac{\beta q_x}{2} + \frac{\beta}{2} \ln(1-q_x).
\end{align*}
We now turn to the term $I_\mathrm{int}(q_x,q_z)$.
We can write it as:
\begin{align}\label{eq:Iint_gaussian_prior}
    I_\mathrm{int}(q_x,q_z) &= \inf_{\gamma_x,\gamma_z \geq 0} \Big[\frac{\beta}{2} (1 - q_x) \gamma_x + \frac{\alpha \beta}{2} (Q_z - q_z) \gamma_z - \frac{\beta}{2} \langle \ln (1 + \gamma_x + \lambda \gamma_z) \rangle_\nu \Big] \\
    & \hspace{1cm} - \frac{\beta}{2} \ln(1-q_x) - \frac{\beta q_x}{2} - \frac{\alpha \beta}{2} \ln(Q_z - q_z) - \frac{\alpha \beta q_z}{2 Q_z}. \nonumber
\end{align}
So we have, using Corollary~\ref{corollary:supinf_two_variables}, that if $f \equiv \sup_{q_x \in [0,1]} \sup_{q_z \in [0,Q_z]} [I_0(q_x) + \alpha I_\mathrm{out}(q_z) + I_\mathrm{int}(q_x,q_z)]$ is the conjectured limit of the free entropy:
\begin{align}
f &= \sup_{q_x \in [0,1]} \sup_{q_z \in [0,Q_z]} \inf_{\gamma_x,\gamma_z \geq 0} \Big[\alpha I_\mathrm{out}(q_z) + 
\frac{\beta}{2} (1 - q_x) \gamma_x + \frac{\alpha \beta}{2} (Q_z - q_z) \gamma_z \nonumber\\
&\hspace{2cm}- \frac{\beta}{2} \langle \ln (1 + \gamma_x + \lambda \gamma_z) \rangle_\nu - \frac{\alpha \beta}{2} \ln(Q_z - q_z) - \frac{\alpha \beta q_z}{2 Q_z}\Big], \nonumber \\
\label{eq:f_supinf}
&=\sup_{\gamma_x,\gamma_z \geq 0} \inf_{q_z \in [0,Q_z]} \inf_{q_x \in [0,1]}  \Big[\alpha I_\mathrm{out}(q_z) + 
\frac{\beta}{2} (1 - q_x) \gamma_x + \frac{\alpha \beta}{2} (Q_z - q_z) \gamma_z \\
&\hspace{2cm}- \frac{\beta}{2} \langle \ln (1 + \gamma_x + \lambda \gamma_z) \rangle_\nu - \frac{\alpha \beta}{2} \ln(Q_z - q_z) - \frac{\alpha \beta q_z}{2 Q_z}\Big]. \nonumber
\end{align}
The infimum on $q_x$ is very easily solved, as we have $\inf_{q_x \in [0,1]}[-\beta q_x \gamma_x / 2] = - \beta \gamma_x / 2$. Note that at fixed $\gamma_z \geq 0$, the variables $\gamma_x,q_z$ are completely decoupled in eq.~\eqref{eq:f_supinf}, so we have $\sup_{\gamma_x} \inf_{q_z} = \inf_{q_z} \sup_{\gamma_x}$. This yields:
\begin{align*}
f &= \sup_{\gamma_z \geq 0} \inf_{q_z \in [0,Q_z]} \sup_{\gamma_x \geq 0}  \Big[\alpha I_\mathrm{out}(q_z) + 
\frac{\alpha \beta}{2} (Q_z - q_z) \gamma_z \\
&\hspace{2cm}- \frac{\beta}{2} \langle \ln (1 + \gamma_x + \lambda \gamma_z) \rangle_\nu - \frac{\alpha \beta}{2} \ln(Q_z - q_z) - \frac{\alpha \beta q_z}{2 Q_z}\Big], \\ 
&= \sup_{\gamma_z \geq 0} \inf_{q_z \in [0,Q_z]} \Big[ \frac{\beta}{2} \big[
\alpha (Q_z - q_z) \gamma_z - \alpha \frac{q_z}{Q_z} - \langle \ln (1 + \lambda \gamma_z) \rangle_\nu -\alpha \ln(Q_z - q_z) \big] +\alpha I_\mathrm{out}(q_z)\Big].
\end{align*}
Recall the form of $I_\mathrm{out}$ in Conjecture~\ref{conjecture:replicas_general} and that $\hat{Q}_z = 1/Q_z$. Using the form of $I_\mathrm{out}$, we have with the notations of Proposition~\ref{prop:replicas_product_structure}:
\begin{align*}
f &= \sup_{\gamma_z \geq 0} \inf_{q_z \in [0,Q_z]} \inf_{\hat{q}_z \geq 0} \Big[ \frac{\beta}{2} \big[
\alpha (Q_z - q_z) \gamma_z - \alpha \frac{q_z}{Q_z} - \langle \ln (1 + \lambda \gamma_z) \rangle_\nu -\alpha \ln(Q_z - q_z) \\ 
&- \alpha q_z \hat{q}_z - \alpha \ln(\hat{q}_z+1/Q_z) + \alpha Q_z \hat{q}_z\big]+  \alpha \Psi_\mathrm{out}(\sqrt{Q_z^2 \hat{q}_z / (1+Q_z \hat{q}_z)})
\Big].
\end{align*}
Again, we use that at fixed $q_z$, the variables $\hat{q}_z,\gamma_z$ are decoupled. So using again Lemma~\ref{lemma:supinf_barbier}, we have schematically $\sup_{\gamma_z} \inf_{q_z} \inf_{\hat{q}_z} = \sup_{q_z} \inf_{\hat{q}_z} \inf_{\gamma_z} = \sup_{\hat{q}_z}  \inf_{\gamma_z} \inf_{q_z}$. We can then explicitly solve the infimum on $q_z$, which yields:
\begin{align*}
f &= \sup_{\hat{q}_z \geq 0} \inf_{\gamma_z \geq 0} \Big[ \frac{\beta}{2} \big[
- \langle \ln (1 + \lambda \gamma_z) \rangle_\nu + \alpha \ln(1+\gamma_z(Q_z-q(\hat{q}_z))) \big]
+  \alpha \Psi_\mathrm{out}(q(\hat{q}_z))
\Big],
\end{align*}
with 
\begin{align}
   q(\hat{q}_z) &\equiv \frac{Q_z^2 \hat{q}_z}{1 + Q_z \hat{q}_z}. 
\end{align}
Note that $q$ is a strictly increasing smooth function of $\hat{q}_z$, with $q(0) = 0$ and $q(+\infty) = Q_z$. 
So we have:
\begin{align}\label{eq:f_gammaz}
f &= \sup_{q \in [0,Q_z]} \inf_{\gamma_z \geq 0} \Big[ \frac{\beta}{2} \big[
- \langle \ln (1 + \lambda \gamma_z) \rangle_\nu + \alpha \ln(1+\gamma_z(Q_z-q)) \big]
+  \alpha \Psi_\mathrm{out}(q)
\Big],
\end{align}
We then state a technical lemma:
\begin{lemma}\label{lemma:technical_phi_structure}
Under hypothesis~\ref{hyp:phi_product}, one has for every $q \in [0,Q_z]$:
   \begin{align*}
        \inf_{\gamma_z \geq 0} [\alpha \ln(1+\gamma_z(Q_z-q))- \langle \ln (1 + \lambda \gamma_z) \rangle_\nu] &= \inf_{\hat{q} \geq 0} [\delta\hat{q}(Q_z - q)- \EE_{\nu_B} \ln(1+\hat{q}X)].
   \end{align*} 
\end{lemma}
Using Lemma~\ref{lemma:technical_phi_structure} in eq.~\eqref{eq:f_gammaz}, and inverting the sup-inf by Lemma~\ref{lemma:supinf_barbier} finishes the proof of Proposition~\ref{prop:replicas_product_structure}.
In the remaining of the section we prove Lemma~\ref{lemma:technical_phi_structure}

\subsection{Proof of Lemma~\ref{lemma:technical_phi_structure}}
If $q=Q_z$, the equality is trivially satisfied, so let us assume $0 \leq q < Q_z$.
Let us denote $h(\gamma_z) \equiv \alpha \ln(1+\gamma_z(Q_z-q))- \langle \ln (1 + \lambda \gamma_z) \rangle_\nu$.
Recall that $Q_z = \EE_{\nu_B}[X]/\delta$.
Since $\alpha \geq 1-\nu(\{0\})$ and $q < Q_z$, one easily checks that $h$ is lower-bounded, so the infimum is always well-defined. 
We introduce $\mu$ the asymptotic measure of $\bPhi \bPhi^\dagger / n$, and we denote $g_\mu(z) \equiv \langle (\lambda-z)^{-1}\rangle_\mu$ its Stieltjes transform. For every function $f$, one has $\langle f(\lambda) \rangle_\nu = \alpha \langle f(\lambda) \rangle_\mu + (1-\alpha) f(0)$. This allows to write:
\begin{align*}
    h(\gamma_z)  = \alpha \ln(1+\gamma_z(Q_z-q))- \alpha \langle \ln (1 + \lambda \gamma_z) \rangle_\mu.
\end{align*}
We will use the following equation, valid for every $\gamma_z \geq 0$ and any positively supported measure $\mu$:
\begin{align}\label{eq:log_potential}
    \langle \ln(\gamma_z+\lambda)\rangle_\mu &= \inf_{\tilde{\gamma_z} \geq 0} \Big[\gamma_z \tilde{\gamma_z} + \int_0^{\tilde{\gamma_z}} \mathcal{R}_\mu(-t) \mathrm{d}t - \ln \tilde{\gamma_z} -1\Big],
\end{align}
in which $\mathcal{R}_\mu$ is the so-called ``$R$-transform'' of $\mu$, defined as $\mathcal{R}_\mu(-x) \equiv g_\mu^{-1}(x) + 1/x$. It is a classical result of random matrix theory \cite{tulino2004random} that if $\mu$ is positively supported, $t \mapsto \mathcal{R}_\mu(-t)$ is well-defined on $\bbR_+$. We finish the proof of Lemma~\ref{lemma:technical_phi_structure}, before proving eq.~\eqref{eq:log_potential}. 
By a classical result of random matrix theory \cite{marchenko1967distribution}, we know the $R$-transform of $\mu$ as a function of $\nu_B$:
\begin{align}\label{eq:R_transform_product}
   \mathcal{R}_\mu(-t) &= \EE_{\nu_B} \Big[\frac{X}{\delta + \alpha t X}\Big].
\end{align}
Combining eq.~\eqref{eq:log_potential} and eq.~\eqref{eq:R_transform_product}, we reach:
\begin{align*}
   \inf_{\gamma_z \geq 0} h(\gamma_z) &= \inf_{\gamma_z \geq 0} \sup_{\tilde{\gamma_z} \geq 0} \Big[\alpha \ln(1+\gamma_z(Q_z-q))+ \alpha - \alpha\frac{\tilde{\gamma_z}}{\gamma_z} + \alpha\ln \frac{\tilde{\gamma_z}}{\gamma_z} - \EE_{\nu_B} \ln \Big(1 + \frac{\alpha}{\delta} X \tilde{\gamma_z}\Big) \Big].
\end{align*}
Using Lemma~\ref{lemma:supinf_barbier} to invert the inf-sup, we have:
\begin{align*}
   \inf_{\gamma_z \geq 0} h(\gamma_z) &= \inf_{\tilde{\gamma_z} \geq 0}\sup_{\gamma_z \geq 0}  \Big[\alpha \ln(1+\gamma_z(Q_z-q))+ \alpha - \alpha\frac{\tilde{\gamma_z}}{\gamma_z} + \alpha\ln \frac{\tilde{\gamma_z}}{\gamma_z} - \EE_{\nu_B} \ln \Big(1 + \frac{\alpha}{\delta} X \tilde{\gamma_z}\Big) \Big].
\end{align*}
The supremum on $\gamma_z$ is now completely tractable, and we have:
\begin{align*}
   \inf_{\gamma_z \geq 0} h(\gamma_z) &= \inf_{\tilde{\gamma_z} \geq 0} \Big[\alpha (Q_z-q)\tilde{\gamma_z} - \EE_{\nu_B} \ln \Big(1 + \frac{\alpha}{\delta} X \tilde{\gamma_z}\Big) \Big].
\end{align*}
Doing the replacement $\hat{q} \equiv \alpha \tilde{\gamma_z} / \delta$ yields Lemma~\ref{lemma:technical_phi_structure}.
We now prove eq.~\eqref{eq:log_potential}, which will finish the proof. It follows from a classical result used in random matrix theory, see e.g.~\cite{guionnet2005fourier} for an application of these calculations to spherical integrals. Recall that $g_\mu$ is smooth and strictly increasing on $(-\infty,0)$, as $\mu$ is positively supported. It is easy to see by differentiation that the infimum in eq.~\eqref{eq:log_potential} is attained at $\tilde{\gamma_z} = g_\mu(-\gamma_z)$. We then use some manipulations:
\begin{align*}
    \inf_{\tilde{\gamma_z} \geq 0} \Big[\gamma_z \tilde{\gamma_z} + \int_0^{\tilde{\gamma_z}} \mathcal{R}_\mu(-t) \mathrm{d}t - \ln \tilde{\gamma_z}\Big]
    &= \gamma_z g_\mu(-\gamma_z) + \int_0^{g_\mu(-\gamma_z)} \mathcal{R}_\mu(-t) \mathrm{d}t - \ln g_\mu(-\gamma_z), \\ 
    &\hspace{-1cm}= \gamma_z g_\mu(-\gamma_z) + \int_\epsilon^{g_\mu(-\gamma_z)} g_\mu^{-1}(t) \mathrm{d}t - \ln \epsilon + \int_0^\epsilon \mathcal{R}_\mu(-t) \mathrm{d}t,
\end{align*}
this equation being valid for all $\epsilon > 0$ sufficiently small.
By regularity of the $R$-transform around $0$ \cite{tulino2004random}, $\int_0^\epsilon \mathcal{R}_\mu(-t) \mathrm{d}t = \smallO_\epsilon(1)$. Moreover, we can change variables in the other integral, and we reach:
\begin{align*}
    \inf_{\tilde{\gamma_z} \geq 0} \Big[\gamma_z \tilde{\gamma_z} + \int_0^{\tilde{\gamma_z}} \mathcal{R}_\mu(-t) \mathrm{d}t - \ln \tilde{\gamma_z}\Big] 
    &= \gamma_z g_\mu(-\gamma_z) + \int_{-g_\mu^{-1}(\epsilon)}^{\gamma_z} u g_\mu(-u) \mathrm{d}u - \ln \epsilon + \smallO_\epsilon(1), \\
    &\overset{(a)}{=} - \ln \epsilon - \epsilon g_\mu^{-1}(\epsilon) + \int_{-g_\mu^{-1}(\epsilon)}^{\gamma_z} g_\mu(-u) \mathrm{d}u + \smallO_\epsilon(1), \\ 
    &\overset{(b)}{=} 1 + \langle \ln(\lambda+\gamma_z)\rangle_\mu + \smallO_\epsilon(1),
\end{align*}
in which we used integration by parts in $(a)$ and the definition of the Stieltjes transform in $(b)$.  Since $\epsilon$ was taken arbitrarily small, taking the limit $\epsilon \to 0$ ends the proof.
\section{Technical lemmas and definitions}\label{sec:app_technical}

\subsection{Some definitions}\label{subsec:app_definitions}

Let $\beta \in \{1,2\}$. We denote $\bbK = \bbR$ if $\beta = 1$ and $\bbK = \bbC$ if $\beta = 2$.
$\mathcal{U}_\beta(n)$ denotes the orthogonal (respectively unitary) group, and 
$\mathcal{S}_\beta(\bbR),\mathcal{S}_\beta^+(\bbR)$ the space of \emph{real} symmetric (resp.\ positive symmetric) matrices of size $\beta$. 
$\mathbbm{1}_\beta$ is the identity matrix of size $\beta$.
To improve clarity, we write $\mathrm{Tr}_\beta$ when taking the trace of a matrix in the space $\mathcal{S}_\beta(\bbR)$.
The standard Gaussian measure is defined on $\bbK$ as:
\begin{align}
    {\cal D}_\beta z &\equiv \Big(\frac{\beta}{2 \pi}\Big)^{\beta/2} \ \exp\Big(-\frac{\beta}{2} |z|^2\Big) \ \mathrm{d}z.
\end{align}
We define three different types of products in $\bbK$, using the identification $\bbK \simeq \bbR^\beta$.
\begin{subnumcases}{\label{eq:definitions_products}}
    \label{eq:def_product_complex}
    z z' & the usual product in $\bbK$, \\ 
    \label{eq:def_product_scalar}
    z \cdot z' \equiv \mathrm{Re}[\overline{z}z'] & the dot product in $\bbR^\beta$.
\end{subnumcases}
For $\beta = 1$, and $M,z \in \bbR$, we also denote $M \star z \equiv Mz$.
For $\beta = 2$, with $z = x + i y \in \bbC$, and $M \in \mathcal{S}_2$ written as:
\begin{align}
    M &\equiv a \mathbbm{1}_2 + 
    \begin{pmatrix}
        b & c \\
        c & -b
    \end{pmatrix},
\end{align}
we define $M \star z$ as the matrix-vector product in $\bbR^\beta$:
\begin{align}\label{eq:def_product_star}
    M \star z &\equiv M
    \begin{pmatrix}
        x \\ y
    \end{pmatrix} 
    = a z + (b+ic) \overline{z}.
\end{align}
Note that in the $\beta = 1$ case, all three products are equivalent.

\subsection{Conventions for derivatives}\label{subsec:app_derivatives}

We often consider functions $f : \bbK \to \bbR$. The derivatives for such functions are defined 
in the usual sense if $\bbK = \bbR$, while for $\bbK = \bbC$ we set it in the ``function of two variables'' sense (with $z = x+i y$):
\begin{align}
    f'(z) &\equiv \partial_x f + i \partial_y f.
\end{align}
We will also define its Laplacian if $\bbK = \bbC$ (if $\bbK = \bbR$ then $\Delta f(x) = f''(x)$):
\begin{align}\label{eq:def_laplace}
    \Delta f(z) &\equiv \partial_x^2 f + \partial_y^2 f.
\end{align}
Importantly, this definition is different from the usual Wirtinger definition of a complex derivative, because we do not consider
holomorphic functions here, but merely differentiable real functions of two variables.
This definition satisfies the following chain rule formula, for $h(x) \equiv f(g(x))$ and $f: \bbK \to \bbR$, $g : \bbR \to \bbK$:
\begin{align}
   h'(x) &= g'(x) \cdot f'(g(x)). 
\end{align}
As a particular case, we have if $f(x) = x \cdot z$ that $f'(x) = z$.
We then have the Stein lemma (or Gaussian integration by parts), for any $\mathcal{C}^2$ function $f : \bbK \to \bbR$:
\begin{align}
    \label{eq:stein_lemma_order_1}
    \int \mathcal{D}_\beta z \ (z f(z)) &= \frac{1}{\beta} \int \mathcal{D}_\beta z \ f'(z),  \\
    \label{eq:stein_lemma_order_2}
    \int \mathcal{D}_\beta z \ (z \cdot f'(z)) &= \frac{1}{\beta} \int \mathcal{D}_\beta z \ \Delta f(z).
\end{align}

\subsection{Nishimori identity}

We state here the Nishimori identity, a classical consequence of Bayes optimality.
\begin{proposition}[Nishimori identity]\label{prop:nishimori}
    Let $(X,Y)$ be random variables on a Polish space E. Let $k \in \bbN^\star$ and 
    $(X_1,\cdots,X_k)$ i.i.d.\ random variables sampled from the conditional distribution $\mathbb{P}(X|Y)$.
    We denote $\langle \cdot \rangle_Y$ the average with respect to $\mathbb{P}(X|Y)$, and $\EE[\cdot]$ the average with respect to
    the joint law of $(X,Y)$. Then, for all $f : E^{k+1} \to \bbK$ continuous and bounded:
    \begin{align}
        \EE [\langle f(Y,X_1,\cdots,X_k) \rangle_Y ] &= \EE [\langle (Y,X_1,\cdots, X_{k-1},X) \rangle_Y ].
    \end{align}
\end{proposition}

\begin{proof}[Proof of Proposition~\ref{prop:nishimori}]
    The proposition arises as a trivial consequence of Bayes' formula:
   \begin{align*}
    \EE [\langle f(Y,X_1,\cdots,X_{k-1},X) \rangle_Y ] &= \EE_Y \EE_{X|Y} [\langle f(Y,X_1,\cdots,X_{k-1},X) \rangle_Y ],\\
    &= \EE_Y [\langle f(Y,X_1,\cdots,X_k) \rangle_Y ].
   \end{align*}
\end{proof}

\subsection{Boundedness of an overlap fluctuation}

\begin{lemma}[Boundedness of an overlap fluctuation]\label{lemma:boundedness_fluctuation}
    Under \ref{hyp:channel}, one can find a constant $C > 0$ independent of $n,t,\epsilon$ such that for any $r \geq 0$:
    \begin{align}
        \EE\Big \langle \Big|\frac{1}{n} \sum_{\mu=1}^m u_{Y_{t,\mu}}'(S_{t,\mu})^\dagger u'_{Y_{t,\mu}}(s_{t,\mu}) - \beta^2 \delta r\Big|^2\Big\rangle_{n,t,\epsilon} &\leq 2 \beta^4 \delta^2 r^2 + C.
    \end{align}
\end{lemma}

\begin{proof}[Proof of Lemma~\ref{lemma:boundedness_fluctuation}]
    We directly have:
    \begin{align*}
        \EE\Big \langle \Big|\frac{1}{n} \sum_{\mu=1}^m u_{Y_{t,\mu}}'(S_{t,\mu})^\dagger & u'_{Y_{t,\mu}}(s_{t,\mu}) - \beta^2 \delta r\Big|^2\Big\rangle_{n,t,\epsilon} \\
        &\leq 2 \beta^4 \delta^2 r^2 
         + 2 \EE\Big \langle \Big|\frac{1}{n} \sum_{\mu=1}^m u_{Y_{t,\mu}}'(S_{t,\mu})^\dagger u'_{Y_{t,\mu}}(s_{t,\mu})\Big|^2\Big\rangle_{n,t,\epsilon}
    \end{align*}
    We can bound $|u'_{Y_{t,\mu}}(s)|$ for any $s \in \bbK$ by using the formulation of the channel described in eq.~\eqref{eq:alternative_pout}, 
    which allows to formally write:
    \begin{align*}
        u'_{Y_{t,\mu}}(s) &= \lim_{\Delta \downarrow 0}\frac{\int P_A(\mathrm{d}a) \partial_s\varphi_\mathrm{out}(s,a) (Y_{t,\mu} - \varphi_\mathrm{out}(s,a)) e^{-\frac{1}{2 \Delta}(Y_{t,\mu} - \varphi_\mathrm{out}(s,a))^2}}{\int P_A(\mathrm{d}a) e^{-\frac{1}{2 \Delta}(Y_{t,\mu} - \varphi_\mathrm{out}(s,a))^2}},
    \end{align*}
    in which we used a Gaussian representation of the delta distribution. This amounts to add a small Gaussian noise to the model of eq.~\eqref{eq:alternative_pout}, 
    and effectively write it as:
    \begin{align}
        Y_{\mu} \sim \varphi_\mathrm{out}(S_\mu,A_\mu) + \sqrt{\Delta} Z'_\mu,
    \end{align}
    with $Z'_\mu \overset{\mathrm{i.i.d.}}{\sim}\mathcal{N}(0,1)$, and then take the $\Delta \to 0$ limit.
    We have $|Y_{t,\mu}| \leq \norm{\varphi_\mathrm{out}}_\infty + \sqrt{\Delta} |Z'_\mu|$, 
    and thus taking $\Delta \to 0$ we reach:
    \begin{align*}
        |u'_{Y_{t,\mu}}(s)| &\leq 2 \norm{\varphi_\mathrm{out}}_\infty \norm{\partial_s \varphi_\mathrm{out}}_\infty.
    \end{align*}
    The right-hand side of the last inequality is bounded by hypothesis~\ref{hyp:channel}, and in the end, we have: 
    \begin{align*}
        \EE\Big \langle \Big|\frac{1}{n} \sum_{\mu=1}^m u_{Y_{t,\mu}}'(S_{t,\mu})^\dagger u'_{Y_{t,\mu}}(s_{t,\mu}) - \beta^2 \delta r\Big|^2\Big\rangle_{n,t,\epsilon} &\leq 2 \beta^4 \delta^2 r^2 + 2^5 \norm{\varphi_\mathrm{out}}_\infty^4 \norm{\partial_s \varphi_\mathrm{out}}_\infty^4,
    \end{align*}
    which ends the proof.
\end{proof}

\end{document}